\documentclass[a4paper,reqno]{amsart}
\usepackage{amsfonts}
\usepackage{amssymb}
\usepackage{enumitem}
\usepackage{color}
\usepackage[colorlinks=true]{hyperref}
\hypersetup{urlcolor=blue,citecolor=blue}

\setlist[enumerate]{wide,label=\emph{(\roman*)}}
\numberwithin{equation}{section}
\newtheorem{thm}{Theorem}[section]
\newtheorem{lem}[thm]{Lemma} 

\newtheorem{prop}[thm]{Proposition}
\newtheorem{defn}[thm]{Definition}
\theoremstyle{definition}
\newtheorem{rem}[thm]{Remark}

\newcommand\R{{\mathbb R}}

\newcommand\N{{\mathbb N}}

\newcommand\Comp{{\mathrm{c}}}
\newcommand\cont{{\mathcal C}}
\newcommand\Ens{{\mathcal E}}
\newcommand\nub{{\boldsymbol \nu}}
\newcommand\norm{{\mathcal N}}
\newcommand\hnorm{{\mathcal M}}
\newcommand\energy{{\mathcal H}}
\newcommand\higher{{\mathcal K}}
\newcommand\goto{\mathop{\longrightarrow}}
\newcommand\Loc{{\mathrm{loc}}}
\newcommand\Map{ \Lambda }

\newcommand\MScN[1]{\href{http://www.ams.org/mathscinet-getitem?mr=#1}{\nolinkurl{(#1)}}}
\newcommand\DOI[1]{\href{http://dx.doi.org/#1}{(doi: \nolinkurl{#1})}}
\newcommand\LINK[1]{\href{#1}{(link: \nolinkurl{#1})}}

\title[Prescribed blow-up surface for NLW]{Solutions with prescribed local blow-up surface for the nonlinear wave equation}
\author[T. Cazenave]{Thierry Cazenave$^1$}
\email{\href{mailto:thierry.cazenave@sorbonne-universite.fr}{thierry.cazenave@sorbonne-universite.fr}}

\author[Y. Martel]{Yvan Martel$^{2}$}
\email{\href{mailto:yvan.martel@polytechnique.edu}{yvan.martel@polytechnique.edu}}

\author[L. Zhao]{Lifeng Zhao$^3$}
\email{\href{mailto:zhaolf@ustc.edu.cn}{zhaolf@ustc.edu.cn}}

\thanks{L. Zhao was partially supported by the NSFC Grant of China No. 11771415}

\address{$^1$Sorbonne Universit\'e, CNRS, Universit\'e de Paris, Laboratoire Jacques-Louis Lions,
B.C. 187, 4 place Jussieu, 75252 Paris Cedex 05, France}

\address{$^2$CMLS, \'Ecole Polytechnique, CNRS, 91128 Palaiseau Cedex, France}

\address{$^3$Wu Wen-Tsun Key Laboratory of Mathematics and School of Mathematical Sciences, University of Science and Technology of China, Hefei 230026, Anhui, China}



\begin{document}

\begin{abstract}
We prove that any sufficiently differentiable space-like hypersurface of ${\mathbb R}^{1+N} $ coincides locally around any of its points with the blow-up surface of a  finite-energy solution of the focusing nonlinear wave equation $\partial_{tt} u - \Delta u=|u|^{p-1} u$ on ${\mathbb R} \times {\mathbb R} ^N$,
for any $1\leq N\leq 4$ and  $1 < p \le \frac {N+2} {N-2}$.
We follow the strategy developed in our previous work~\cite{CaMaZhwave1} on the construction of solutions of the nonlinear wave equation blowing up at any prescribed compact set. Here to prove blowup on a local space-like hypersurface, we first apply a change of variable to reduce the problem to blowup on a small ball at $t=0$ for a transformed equation. The construction of an appropriate approximate solution is then combined with an energy method for the existence of a solution of the transformed problem that blows up at $t=0$.
To obtain a finite-energy solution of the original problem from trace arguments, we need to work with $H^2\times H^1$ solutions for the transformed problem.
\end{abstract}
\maketitle

\begin{center} 
\it Dedicated to Laurent V\'eron on the occasion of his 70th birthday
\end{center} 

\tableofcontents

\section{Introduction}
 \subsection{Main result}
We consider the nonlinear energy-subcritical or -critical wave equation 
\begin{equation}\label{wave}
\partial_{tt} u - \Delta u=|u|^{p-1} u,\quad (t,x)\in \R\times\R^N,
\end{equation}
for $N\geq 1$ and $1< p\le  \frac {N+2} {N-2}$ ($1<p<\infty $ if $N=1,2$).
For simplicity, we restrict ourselves to space dimensions $1\leq N\leq 4$.
In this case, it is well-known that the Cauchy problem for \eqref{wave} is locally well-posed in the energy space $H^1(\R^N)\times L^2(\R^N)$. (See Remark~\ref{eLinCo5}.)

When a solution $u$ with initial data at $t=t_0$ is not globally defined (\cite{Keller, Levine, Alinhac}), we introduce its \emph{maximal influence domain} whose upper boundary is a $1$-Lipschitz graph.
See \cite[Section~III.2]{Alinhac} and, for the present setting, Section~\ref{S1.3}.

We prove that any sufficiently differentiable space-like hypersurface of $\R^{1+N} $ coincides locally around any of its points with the blow-up surface of a  finite-energy solution of the focusing nonlinear wave equation~\eqref{wave}.
More precisely, our main result is the following.

\begin{thm}\label{TH:2}
Let $1\leq N\leq 4$ and $1< p\le  \frac {N+2} {N-2}$.
Let 
\begin{equation}\label{th2:1}
q_0=2 \left\lfloor \frac{2p+2}{p-1}\right\rfloor+3.
\end{equation}
Let $\varphi:\R^N\to\R$ be a function of class $\cont^{q_0}$ such that
\begin{equation}\label{th2:2}
\varphi(0)=0 \quad \mbox{and}\quad |\nabla \varphi(x)|< 1\mbox{ for all $x\in\R^N$.}
\end{equation}
There exist $\varepsilon >0$, $\tau_0>0$ and $(u_0,u_1)\in H^1(\R^N)\times L^2(\R^N)$ such that 
the upper boundary of the maximal influence domain
 of the solution~$u$ of \eqref{wave} with initial data $(u,\partial_tu)(0)=(u_0,u_1)$ contains the local hypersurface $\{(t,x): t= \tau _0 + \varphi(x)\mbox{ and } |x|<\varepsilon\}$.
Moreover, $u$ blows up on this local hypersurface in the sense that if $ |x_0| < \varepsilon $ and $\sigma \in (  | \nabla \varphi (0)| ,1)$, then
\begin{equation} \label{fBUest} 
\liminf _{ t\uparrow \tau _0+ \varphi (x_0)}  
\frac {1} {\tau _0+ \varphi (x_0) -t } \int _t ^{\tau _0+ \varphi (x_0)} dt' \int  _{ \{ |x-x_0|< \sigma (\tau _0+ \varphi (x_0) -t') \}}  |\partial _t u|^2 dx >0.
\end{equation} 
\end{thm}

It follows from~\eqref{fBUest} that $\partial _tu$ concentrates on the local hypersurface $\{(t,x): t= \tau _0 + \varphi(x)\mbox{ and } |x|<\varepsilon\}$ in the sense of $L^2$. In particular, this local hypersurface is a \emph{blow-up surface} for the solution $u$. 

Compared to previous results (see Section~\ref{s:1.2}), Theorem~\ref{TH:2} applies to any space dimension $N\le 4$ and any subcritical or critical $p$. Moreover, our strategy  is different. It mainly relies on the construction of an ansatz by elementary ODE arguments.  (See Section~\ref{s:1:4}.)

\begin{rem}
In the definition of $q_0$ above, we use the notation $y\mapsto \lfloor y \rfloor$ for the floor function which maps $y$ to the greatest integer less than or equal to $y$.
Note that $q_0=7$ for $p> 5$ and $q_0\to \infty$ as $p\to 1^+$. 
See Remark~\ref{eRems2} for comments on this condition.
\end{rem}

\subsection{Definition of the maximal influence domain}\label{S1.3} 
We adapt the presentation of \cite{Alinhac}, Chapter~III (see also~\cite{Lindblad}) to the framework of $ H^1\times L^2$ solutions
for the energy subcritical or critical wave equation in space dimension $N\ge 1$.
Let 
\[
\R_+^{1+N} =  [0,+\infty)\times \R^N.
\]
For any $(t,x)\in \R_+^{1+N}$, we define the open (in $\R_+^{1+N}$) backward cone 
\begin{equation} \label{defcone} 
C(t,x)=\left\{(s,y)\in \R_+^{1+N}\mbox{ such that } |x-y| < t-s \right\}.
\end{equation} 
\begin{defn}
An open set $\Omega$ of $\R_+^{1+N}$ is called an influence domain if $(t,x)\in \Omega$ implies
$\overline C(t,x)\subset \Omega$.
\end{defn}
For $\Omega$ an influence domain containing $\{0\}\times \R^N$, define for any $x\in \R^N$,
\[
\phi(x)=\sup\left\{ t\geq 0\mbox{ such that } (t,x)\in\Omega\right\}.
\]
From the above definition, either $\phi$ is identically $\infty$, or it is finite for all $x\in \R^N$.
In the latter case, $\phi$ is a $1$-Lipschitz continuous function.
\medskip

 Recall that by the Cauchy theory in the energy space $ H^1 (\R^N ) \times L^2 (\R^N ) $, for any $(u_0,u_1)\in H^1 (\R^N ) \times L^2 (\R^N ) $ there exist $T>0$ and a solution $(u,\partial_tu)$ of~\eqref{wave} belonging to $C([0,T],H^1(\R^N))\cap C([0,T],L^2(\R^N))$. These solutions are unique in that class, except for the 3D critical case $p=5$, where uniqueness is known in $C([0,T],H^1(\R^3))\cap C([0,T],L^2(\R^3)) \cap L^ {8} ( (0,T) \times \R^3  )$.
(See Remark~\ref{eLinCo5} for details.)

From the local Cauchy theory, it is standard to define the notion of maximal solution and  maximal time of 
existence $T_{\rm max}(u_0,u_1)>0$; if $T_{\rm max}(u_0,u_1)=\infty$, the solution is globally defined, otherwise it blows up as   $t\uparrow T_{\rm max}(u_0,u_1)$ (in a suitable norm related to the resolution of the Cauchy problem).

To define the notion of maximal influence domain  corresponding to an initial data
we first extend the Cauchy theory of $\R^N$ to truncated
cones. For $x_0\in\R^N$  and $0\leq \tau\leq R$, we define 
\begin{equation} \label{fTrCo1} 
E(x_0,R,\tau)=\left\{(t,x)\in \R_+^{1+N}\mbox{ such that $0\leq t< \tau$ and $|x-x_0|< R- t $}\right\}.
\end{equation} 
Suppose that $x_0\in\R^N$ and $R>0$, and
let $(u_0,u_1)\in H^1(B(x_0,R))\times L^2(B(x_0,R))$. Consider any extension $(  \widetilde{u} _0,  \widetilde{u} _1)
\in  H^1(\R^N)\times L^2(\R^N)$ of $(u_0,u_1)$ \emph{i.e.} any function satisfying
\[
 \widetilde{ u} _0=u_0\quad \mbox{and}\quad   \widetilde{u} _1=u_1 \quad \text{on $B(x_0,R)$}.
\]
Next,  consider the solution $(  \widetilde{u} ,\partial_t   \widetilde{u} )$ of \eqref{wave} corresponding to the initial data 
$( \widetilde{ u} _0, \widetilde{ u} _1)$  defined on a time interval $[0,  \widetilde{\tau} ]$ where $ \widetilde{\tau} >0$, given by the above Cauchy theory.
Note that if $(\check u_0,\check u_1)
\in H^1(\R^N)\times L^2(\R^N)$ is another extension of $(u_0,u_1)$  and $(\check u,\partial_t \check u)$ is the corresponding solution of \eqref{wave}  on a time interval $[0,\check \tau]$ ($\check\tau>0$) then  by finite speed of propagation (see Proposition~\ref{eLinCo4}),
the two solutions $(  \widetilde{u} ,\partial_t   \widetilde{u} )$ and $(\check u,\partial_t \check u)$ are identically equal
on the truncated cone  $E(x_0,R,\min(   \widetilde{\tau} ,\check\tau))$.
In this way, we have defined a notion of solution of \eqref{wave} on $E(x_0,R, \tau)$ for some $\tau>0$ which is independent of the extension chosen and includes a uniqueness property. \emph{From now on, for any $(u_0,u_1)\in  H^1(B(x_0,R))\times L^2(B(x_0,R))$ and any $\tau>0$, we   refer to the solution of \eqref{wave} on $E(x_0,R, \tau)$ in this sense.}

By time-translation invariance of the equation and  considering the map $(t,x)\in E( x_0, R, \tau ) \mapsto u(t_0 +t, x)$, we extend this definition to any truncated cone in $\R^{1+N}_+ $.

Now, we define the notion of solution in an influence domain.

\begin{defn}\label{def:1}
Let $(u_0,u_1)\in H^1(\R^N)\times L^2(\R^N)$.
Let $\Omega$ be an influence domain.
We say that $(u,\partial_t u)$ is a solution of~\eqref{wave}  on $\Omega$ with initial data $(u_0,u_1)$ if the following
hold.
\begin{enumerate}
\item $u\in  H^1_\Loc (\Omega)$;
\item For any $t_0\geq 0$, $x_0\in \R^N$ and $R>0$ such that $[0,t_0]\times B(x_0,R)\subset \Omega$,
it holds $u_{|[0,t_0]\times B(x_0,R)}\in C([0,t_0], H^1(B(x_0,R)))\cap C^1([t_1,t_2],L^2(B(x_0,R)))$;
moreover, $u(0)\equiv u_0$ and $\partial_t u(0)\equiv u_1$ on $B(x_0,R)$;
\item For any $(t_0,x_0)\in\Omega$ and $R>0$ such that $\{t_0\}\times B(x_0,R)\subset \Omega$,
there exists $0<\tau<R$ such that $(t,x)\in E(x_0,R,\tau)\mapsto u(t_0+t,x)$ is solution of \eqref{wave} in the above sense.
\end{enumerate}
\end{defn}

\begin{defn}
For any $(u_0,u_1)\in  H^1(\R^N)\times L^2(\R^N)$, we denote $\Omega_{\rm max}(u_0,u_1)$ the union of
all the influence domains $\Omega$ such that there exists a solution $(u,\partial_t u)$ with initial data
$(u_0,u_1)$ on $\Omega$ in the sense of Definition $\ref{def:1}$.
\end{defn}

It follows that, for any initial data $(u_0,u_1)\in  H^1(\R^N)\times L^2(\R^N)$, $\Omega_{\rm max}(u_0,u_1)$ is the maximal influence domain on which a (unique) solution of~\eqref{wave}  with initial data $(u_0,u_1)$ exists.
Finally, in the case $T_{\rm max}(u_0,u_1)<+\infty$  the upper boundary of the maximal influence domain is the graph of the $1$-Lipschitz application
 \[
x\in \R^N \mapsto \phi(x)=\sup\left\{ t\geq 0\mbox{ such that } (t,x)\in\Omega_{\rm max}(u_0,u_1)\right\}\in (0,\infty).
\]

\subsection{Previous results}\label{s:1.2}

Under certain assumptions, it is known that the upper boundary of the maximal influence domain is a blow-up surface in the sense that the solution blows up (at the same rate as the ODE) on the surface, and the blow-up surface is $C^1$. See~\cite{CaFr1, CaFr2} and~\cite[Chapter~III]{Alinhac}. 
See also~\cite{Don-Sch, MerleZ2, MerleZ3} and the references therein for further blow-up results.

Constructing solutions of the wave equation \eqref{wave} with prescribed blow-up surface is a classical question.
Results similar to Theorem~\ref{th2:1} have been proved in several cases. For the wave equation with cubic nonlinearity, it is proved in~\cite[Theorem~10.14, p.~192]{Kibook} that there exist solutions (locally defined around the blow-up surface) blowing up exactly on a prescribed surface of class $H^r (\R^N )$ with $r> \frac {N} {2} + 7$. 
In~\cite[Theorem~1.1]{KiVi}, an analogous result is proved in space dimension $1$ for equation~\eqref{wave}  for any $p>1$. For previous results, see~\cite{Alinhac, KiLi1, KiLi2, Kiche1,Kiche2}.

A related question is the study of the blow-up set, which is the intersection of the blow-up surface with the hyperplane $\{t= T_{\max } \}$. 
In~\cite[Corollary~1.2]{KiVi}, it is proved for~\eqref{wave} in space dimension $1$  that, given any compact subset $K$ of $\R$, there exist smooth initial data for which the blow-up set is precisely $K$. This result is extended in~\cite[Theorem~1.1]{CaMaZhwave1} to any space dimension and any energy-subcritical $p$. 
See~\cite{Kicontrol} for a related result. 

\subsection{Strategy of the proof of Theorem~$\ref{TH:2}$} \label{s:1:4} 

We follow closely the strategy of~\cite{CaMaZhwave1} (see also~\cite{CaMaZh}). It is based on the construction of an appropriate approximate solution which blows up at $t=0$, combined with an energy method for the existence of an exact solution that also blows up at $t=0$. 
Here, we wish to prove blowup on a local space-like hypersurface. In order to apply the previously recalled strategy, we therefore apply a change of variable to reduce the problem to blowup at $t=0$ (Section~\ref{ssec:2:1}). 
By doing so, we are led to study the transformed equation 
\begin{equation*}
(1-|\nabla \psi|^2) \partial_{ss} v - 2 \nabla \psi\cdot \nabla \partial_s v
-(\Delta \psi) \partial_s v - \Delta v = f(v)
\end{equation*}
in the dual variables $(s, y)\in \R \times \R^N $. 
The construction of an appropriate ansatz for this equation (Sections~\ref{s22} and~\ref{s23}) is similar to the construction made in~\cite{CaMaZhwave1}. 
In particular, it is based on elementary ODE arguments. 
The energy method for this transformed equation requires a smallness condition on $  \| \nabla \psi \| _{ L^\infty  }$, and yields an existence time that depends on $\psi $.
See Section~\ref{sec:3}. 
This smallness condition can be met through a  localization argument (Section~\ref{sAdjustment}) and a Lorentz transform (Sections~\ref{sConsPsi}-\ref{s:4:4}). 
Going back to the original variables, to obtain a solution in the framework of $H^1 \times L^2$, we are forced to apply a trace argument which requires higher regularity of the solution $v$ (Section~\ref{s:4:5}). This is why we use the energy method for $v$ in the framework of $H^2 \times H^1$. 
The restriction $1\leq N\leq 4$ implies that $H^2 \hookrightarrow L^q$ for every $2\le q<\infty $, which simplifies the energy argument. The blow-up estimate~\eqref{fBUest} is a consequence of an ODE blow-up estimate for the solution of the transformed equation, and the change of variable (Section~\ref{s:4:6}).

\subsection{Notation}
We fix  a smooth, even function  $\chi:\R\to\R$
satisfying:
\begin{equation}\label{def:chi}
\hbox{$\chi\equiv 1$ on $[0,1]$, $\chi\equiv 0$ on $[2,+\infty)$ and $\chi'\leq 0\leq \chi\leq 1$ on $[0,+\infty)$.}
\end{equation}

Let $f(u)=  |u|^{p-1} u $ and $F(u)= \int _0^u f(v)\, dv$.
For future reference, we state and justify two Taylor formulas involving the functions $F$ and $f$
(see Introduction of \cite{CaMaZh} for proofs). 
Let $\bar p = \min(2,p)$.
For any $u>0$ and any $v$, it holds
\begin{align}\label{taylor0}
\Big|F(u+v)-F(v)-F'(u)v-\frac 12F''(u)v^2\Big|
&\lesssim |v|^{p+1}+u^{p-\bar p} |v|^{\bar p+1},
\\\label{taylor1} 
| ( f(u+v)- f( u )- f'(u)v) v |
&\lesssim |v|^{p+1}+u^{p-\bar p} |v|^{\bar p+1},
\\ \label{taylor10} 
|f'(u+v)-f'(u)|&\lesssim |u|^{-1} |v|^p+|u|^{p-2}|v|.
\end{align} 
and
\begin{equation}\label{taylor}
\Big|f(u+v)-f(u)-f'(u)v-\frac 12f''(u)v^2\Big|\lesssim u^{-1}|v|^{p+1}+u^{p-\bar p-1}|v|^{\bar p+1}.
\end{equation}

In the present article, we use multi-variate notation and results from \cite{CoSa}.
For $\beta=(\beta_1,\ldots,\beta_N)\in \N^N$ and $x=(x_1,\ldots,x_N)\in \R^N$, we set
\begin{align*}
&|\beta|=\sum_{j=1}^N \beta_j,\quad \beta!=\prod_{j=1}^N (\beta_j!),\\
& x^\beta=\prod_{j=1}^N x_j^{\beta_j},\quad \partial_x^{\beta} = \frac{\partial^{|\beta|}}{\partial_{x_1}^{\beta_1}\ldots\partial_{x_N}^{\beta_N}}
\quad\mbox{for $|\beta|>0$,}\quad \partial_x^0=\mathop{\textnormal{Id}}.
\end{align*}
For $\beta,\beta'\in \N^N$, we write
$\beta'\leq \beta$ provided $\beta_j'\leq \beta_j$, for all $j=1,\ldots,N$. 
Note that in this case $ |\beta -\beta '| =  |\beta | -  |\beta '|$.
For $\beta'\leq \beta$, we denote
\[
\binom{\beta}{\beta'}=\prod_{j=1}^N\binom{\beta_j}{\beta_j'} = \frac{\beta!}{(\beta'!)(\beta-\beta')!}.
\]
Recall that for two functions $g,h:\R^{1+N}\to \R$, Leibniz's formula writes:
\begin{equation}\label{lbz0}
\partial_x^\beta \left(gh\right)=
\sum_{\beta'\leq\beta}\binom{\beta}{\beta'}\left(\partial^{\beta'}g\right)\left(\partial^{\beta-\beta'}h\right).
\end{equation}

We write $\beta'\prec\beta$ provided one of the following holds
\begin{itemize}
\item $|\beta'|<|\beta|$;
\item $|\beta'|=|\beta|$ and $\beta_1'<\beta_1$;
\item $|\beta'|=|\beta|$, $\beta_1'=\beta_1$,\ldots, $\beta_ \ell '=\beta_ \ell $ and
$\beta_{\ell +1}'<\beta_{ \ell +1}$ for some $1\leq  \ell <N$.
\end{itemize}
We recall the Faa di Bruno formula (see Corollary~2.10 in~\cite{CoSa}).
Let $n=|\beta|\geq 1$. Then, for functions $f:\R\to \R$ and $g:\R^{1+N}\to \R$,
\begin{equation}\label{fdb0}
\partial_x^{\beta} (f\circ g)=
\sum_{r=1}^n \left(f^{(r)}\circ g\right)\sum_{P(\beta,r)}(\beta!)
\prod_{ \ell =1}^n \frac{\left(\partial_x^{\beta_ \ell }g\right)^{r_ \ell }}{(r_ \ell !)(\beta_ \ell !)^{r_ \ell }}
\end{equation}
where
\begin{align*}
P(\beta,r)
&=\Big\{(r_1,\ldots,r_n;\beta_1,\ldots,\beta_n) : \mbox{there exists $1\leq m\leq n$ such that}\\
&\qquad r_ \ell =0 \mbox{ and } \beta_ \ell =0 \mbox{ for $1\leq  \ell \leq n-m$} ;\, r_ \ell >0 \mbox{ for $n-m+1\leq  \ell \leq n$};\\
&\qquad \mbox{and } 0\prec\beta_{n-m+1}\prec\cdots\prec\beta_n \mbox{ are such that }
\sum_{ \ell =1}^n r_ \ell =r,\ \sum_{ \ell =1}^n r_ \ell  \beta_ \ell  = \beta \Big\}.
\end{align*}

We will also need to differentiate in space and time, so we define multi-index notation in space-time:
$\nub=(\alpha,\beta_1,\ldots,\beta_N)\in \N^{1+N}$, $\beta=(\beta_1,\ldots,\beta_N)$, and 
\begin{equation*}
|\nub|=\alpha+|\beta|,\quad \nub!=\alpha ! \beta!,\quad
\partial^{\nub} = \partial_s^\alpha \partial_x^\beta.
\end{equation*}
For $\nub,\nub'\in \N^{1+N}$, we write
$\nub'\leq \nub$ provided $\alpha'\leq \alpha$ and $\beta_j'\leq \beta_j$, for all $j=1,\ldots,N$. In such a case, we denote
\[
\binom{\nub}{\nub'}= \binom{\alpha}{\alpha'}\binom{\beta}{\beta'}.
\]
Then, for two functions $g,h:\R^{1+N}\to \R$:
\begin{equation}\label{lbz}
\partial^\nub\left(gh\right)=
\sum_{\nub'\leq\nub}\binom{\nub}{\nub'}\left(\partial^{\nub'}g\right)\left(\partial^{\nub-\nub'}h\right).
\end{equation}

We write $\nub'\prec\nub$ provided one of the following holds
\begin{itemize}
\item $|\nub'|<|\nub|$;
\item $|\nub'|=|\nub|$ and $\alpha'<\alpha$;
\item $|\nub'|=|\nub|$, $\alpha'=\alpha$ and $\beta_1'<\beta_1$; or
\item $|\nub'|=|\nub|$, $\alpha'=\alpha$, $\beta_1'=\beta_1$,\ldots, $\beta_ \ell '=\beta_ \ell $ and
$\beta_{ \ell +1}'<\beta_{ \ell +1}$ for some $1\ \leq  \ell <N$.
\end{itemize}
Last, we write in this context the Faa di Bruno formula.
Let $n=|\nub|\geq 1$. Then, for functions $f:\R\to \R$ and $g:\R^{1+N}\to \R$,
\begin{equation}\label{fdb}
\partial^{\nub} (f\circ g)=
\sum_{r=1}^n \left(f^{(r)}\circ g\right)\sum_{P(\nub,r)}(\nub!)
\prod_{ \ell =1}^n \frac{\left(\partial^{\nub_ \ell }g\right)^{r_ \ell }}{(r_ \ell !)(\nub_ \ell !)^{r_ \ell }}
\end{equation}
where
\begin{align*}
P(\nub,r)
&=\Big\{(r_1,\ldots,r_n;\nub_1,\ldots,\nub_n) : \mbox{there exists $1\leq m\leq n$ such that}\\
&\qquad r_ \ell =0 \mbox{ and } \nub_ \ell =0 \mbox{ for $1\leq  \ell \leq n-m$} ;\, r_ \ell >0 \mbox{ for $n-m+1\leq  \ell \leq n$};\\
&\qquad \mbox{and } 0\prec\nub_{n-m+1}\prec\cdots\prec\nub_n \mbox{ are such that }
\sum_{ \ell =1}^n r_ \ell =r,\ \sum_{ \ell =1}^n r_ \ell  \nub_ \ell  = \nub \Big\}.
\end{align*}

\section{Blow up ansatz}\label{sec:2}

\subsection{Change of variables} \label{ssec:2:1} 
Let $\psi \in \cont^{q_0} (\R^N ,\R)$, where $q_0 $ is defined by~\eqref{th2:1}, be such that for some $R\geq 2$,
\begin{equation}\label{on:psi}
\mbox{$\psi (x)= 0$ for $|x|\ge R$ and $ \| \nabla \psi \| _{ L^\infty } < 1.$}
\end{equation} 
We perform a change of variable related to $\psi$:
\begin{equation*}
u(t,x)=v(s,x),\quad s=\psi(x)-t
\end{equation*}
so that $ s>0$ is equivalent to $t<\psi(x)$.

Then, the following holds, for $j=1,\ldots,N$,
\begin{align*}
\partial_{tt} u &= \partial_{ss} v,\\
\partial_{x_j} u & = (\partial_{x_j} \psi) \partial_s v + \partial_{x_j} v,\\
\partial_{x_jx_j} u & = (\partial_{x_jx_j} \psi) \partial_s v + (\partial_{x_j}\psi)^2 \partial_{ss}v
+2 (\partial_{x_j}\psi) \partial_{x_js } v + \partial_{x_jx_j} v,\\
\Delta u &= (\Delta \psi)\partial_s v + |\nabla \psi|^2 \partial_{ss} v +2 \nabla \psi \cdot \nabla\partial_s v+\Delta v.
\end{align*}
Therefore, the equation \eqref{wave} on $u(t,x)$ rewrites
\begin{equation}\label{eq:v}
(1-|\nabla \psi|^2) \partial_{ss} v - 2 \nabla \psi\cdot \nabla \partial_s v
-(\Delta \psi) \partial_s v - \Delta v = f(v).
\end{equation}
In this section, we focus on finding ansatz for this equation under assumption~\eqref{on:psi}.

\subsection{First blow up ansatz}\label{s22}
Let 
\begin{equation}\label{defJ}
J=\left\lfloor \frac{2p+2}{p-1}\right\rfloor\quad\mbox{so that}\quad q_0= 2J+3,
\end{equation}
where $q_0 $ is defined by~\eqref{th2:1}, and let 
\begin{equation}\label{onqk}
k\geq q_0 +1
\end{equation}
 be an integer. 

We consider the function $A:\R^N\to[0,+\infty[$ given by
\begin{equation} \label{fExfnA} 
A(x):=
\begin{cases} 
0& \hbox{if $|x|\leq 1$}\\
(|x|-\chi(x))^k&\hbox{if $1< |x|\leq 2$}\\
|x|^k&\hbox{if $|x|>2$} .
\end{cases}
\end{equation}
It follows that $A$ is of class $\cont^{k-1}$ and that, for any $\beta\in \N^N$, with 
$|\beta|\leq k-1$, 
\begin{equation}\label{on:A}
\left.
\begin{aligned}
\hbox{on $\R^N$, } A\geq 0 \hbox{ and }
|\partial_x^\beta A|\lesssim A^{1-\frac {|\beta|}{k}}\\
\hbox{for any $x\in\R^N$ such that $|x|\geq 2$, } A(x)=|x|^k
\end{aligned}\right\} .
\end{equation}

We define a basic blow up ansatz $V_0$, for $s>0$ and $x\in \R^N$,
\begin{equation}\label{def:V0}
V_0(s,x)=\kappa(x) (s+A(x))^{-\frac 2{p-1}},
\end{equation}
where
\begin{equation*} 
 \kappa(x)= \kappa _0\left(1-|\nabla \psi(x)|^2\right)^{\frac 1{p-1}}, 
\quad \kappa_0 =\left[\frac{2(p+1)}{(p-1)^2} \right]^{\frac 1{p-1}},
\end{equation*} 
which satisfies $(1-|\nabla \psi|^2) \partial_{ss} V_0 =V_0^p$ on $(0,+\infty)\times\R^N$.
Since the functions $\psi$ and $A$ are of class $\cont^{q_0}$, we remark that the function $V_0$ is of class $\cont^\infty $ in the variable $s>0$ and of class $\cont^{q_0 -1}$ in the variable $x \in \R^N $.

In view of \eqref{eq:v}, it is natural to set
\begin{align}
\Ens_0
&=-(1-|\nabla \psi|^2)\partial_{ss}V_0 +2 \nabla \psi\cdot \nabla \partial_s V_0 
+(\Delta \psi) \partial_s V_0 + \Delta V_0+f(V_0)\nonumber\\
&= 2 \nabla \psi\cdot \nabla \partial_s V_0
+(\Delta \psi) \partial_s V_0 + \Delta V_0.\label{def:E0}
\end{align}
We gather in the next lemma the properties of $V_0$ and $\Ens_0$.

\begin{lem}\label{le:0}
The function $V_0$ satisfies
\begin{equation}\label{eq:V}
(1-|\nabla \psi|^2)^{\frac 12} \partial_s V_0 = - \left(\frac 2{p+1}V_0^{p+1}\right)^{\frac 12},\quad
(1-|\nabla \psi|^2) \partial_{ss} V_0 =V_0^p.
\end{equation}
Moreover, for any $\alpha\in \N$, $\beta\in \N^N$, $\rho\in \R$, $0<s<1$, $x\in \R^N$, the following hold:
\begin{enumerate}
\item If $0\leq|\beta|\leq q_0 -1$ and $|x|\leq R$, then
\begin{equation}\label{V1}
|\partial_s^\alpha\partial_x^\beta(V_0^\rho)|\lesssim V_0^{\rho+ (\alpha+\frac{|\beta|}k)\frac{p-1}2};
\end{equation}
\item If $|\beta|\leq q_0- 3$ and $|x|\leq R$, then
\begin{equation}\label{V2}
|\partial_s^\alpha\partial_x^\beta \Ens_0|\lesssim V_0^{\frac{p+1}2+(\alpha+\frac{1+|\beta|}k)\frac{p-1}2};
\end{equation}
\item If $|x|> R$, then
\begin{align}
|\partial_s^\alpha\partial_x^\beta V_0|&\lesssim |x|^{-(\frac {2}{p-1}+\alpha)k-|\beta|};\label{V3}
\\ \label{V4}
|\partial_s^\alpha\partial_x^\beta \Ens_0|&\lesssim |x|^{-(\frac {2}{p-1}+\alpha)k-|\beta|-2}.
\end{align}
\end{enumerate}
Furthermore if $  |x_0| < 1$,  then for any $\sigma >0$, 
\begin{equation} \label{dtVrsigma}
\liminf _{ s\downarrow 0 } s^{\frac {N+2- (N-2)p} {2(p-1)}} \|\partial_s V_0(s)\|_{L^2(|x-x_0|< \sigma s)} >0.
\end{equation} 
\end{lem}

\begin{proof} 
First, we observe that the function $\kappa $ is constant for $|x|>R$ and satisfies $\kappa\gtrsim 1$, $|\partial_x^{\beta} \kappa|\lesssim 1$ on $\R^N$, for any $|\beta|\leq q_0-1$.

\smallskip

Proof of \eqref{eq:V}. This follows from direct computations.

\smallskip

Proof of \eqref{V1}.
For $0<s<1$ and $|x|\leq R$, one has $0<s+A(x)\lesssim 1$ and thus, $V_0\gtrsim 1$.
We introduce some notation:
\[
h(z)=z^{-\frac 2{p-1}} \mbox{ for $z>0$},\quad W(s,x)=s+A(x).
\]
In particular, $V_0(s,x)=\kappa(x)h(W(s,x))$.
Let $\alpha\geq 0$. Since $|h^{(\alpha)}(z)|\lesssim |z|^{-\frac 2{p-1}-\alpha}$, we have
\begin{equation*}
\partial_s^\alpha V_0=\kappa(x)h^{(\alpha)}(W(s,x))\quad\mbox{and so}\quad
|\partial_s^\alpha V_0|\lesssim |V_0|^{1+\alpha\frac{p-1}2}.
\end{equation*}
Let $\beta\in \N^N$ be such that $1\leq|\beta|\leq q_0 -1$. Using \eqref{lbz0}, we have
\begin{equation*}
\partial_s^\alpha\partial_x^\beta V_0=
\sum_{\beta'\leq\beta}\binom{\beta}{\beta'}\left(\partial_x^{\beta-\beta'}\kappa\right)
\left(\partial_x^{\beta'}\left[h^{(\alpha)}(W)\right]\right).
\end{equation*}
For $0\leq\beta'\leq\beta$, it holds $|\partial_x^{\beta-\beta'}\kappa|\lesssim 1$.
Thus, for $\beta'=0$ in the above sum, we have
\[
|\left(\partial_x^{\beta}\kappa\right) h^{(\alpha)}(W)|
\lesssim|V_0|^{1+\alpha\frac{p-1}2}.
\]
For $1\leq|\beta'|$, $\beta'\leq\beta$, setting $n'=|\beta'|$ and using \eqref{fdb0},
\begin{equation*}
\partial_x^{\beta'}\left[h^{(\alpha)}(W)\right]=
\sum_{r=1}^{n'} \left[h^{(\alpha+r)}(W)\right]\sum_{P(\beta',r)}(\beta'!)
\prod_{l=1}^{n'} \frac{\left(\partial_x^{\beta_l}W\right)^{r_l}}{(r_l!)(\beta_l!)^{r_l}}
\end{equation*}
where
\begin{align*}
P(\beta',r)
&=\Big\{(r_1,\ldots,r_n;\beta_1,\ldots,\beta_{n'}) : \mbox{there exists $1\leq m\leq n'$ such that}\\
& r_ \ell =0 \mbox{ and } \beta_ \ell =0 \mbox{ for $1\leq  \ell \leq n'-m$} ;\, r_ \ell >0 \mbox{ for $n'-m+1\leq  \ell \leq n'$};\\
& \mbox{and } 0\prec\beta_{n'-m+1}\prec\cdots\prec\beta_{n'} \mbox{ are such that }
\sum_{ \ell =1}^{n'} r_ \ell =r,\ \sum_{ \ell =1}^{n'} r_ \ell  \beta_ \ell  = \beta' \Big\}.
\end{align*}
As before, we use for $r\geq 1$, $|h^{(\alpha+r)}(W)|\lesssim W^{-\frac 2{p-1}-r-\alpha}$.
Moreover, using the assumption \eqref{on:A} on $A$, we have, for $1\leq|\beta_ \ell |\leq q_0 -1$,
\[|\partial_x^{\beta_ \ell }W|\lesssim |\partial_x^{\beta_ \ell }A|\lesssim A^{1-\frac{|\beta_ \ell |}{k}}.\]
Since $\sum_{ \ell =1}^{n'} r_ \ell =r$, $\sum_{ \ell =1}^{n'} r_ \ell  |\beta_ \ell |= |\beta'|$ and $|\beta'|\leq q_0-1 \leq k-1$, we obtain
\begin{align*}
\left|\partial_x^{\beta'}\left[h^{(\alpha)}(W)\right]\right|
&\lesssim\sum_{r=1}^{n'}W^{-\frac2{p-1}-r-\alpha}\sum_{P(\beta',r)}\left[A^{1-\frac{|\beta_ \ell |}k}\right]^{r_ \ell }
\\& \lesssim \sum_{r=1}^{n'}W^{-\frac2{p-1}-r-\alpha}A^{r-\frac{|\beta'|}{k}}
\lesssim W^{-\frac2{p-1}-\alpha-\frac{|\beta'|}{k}}\lesssim V_0^{1+(\alpha+\frac{|\beta'|}k)\frac{p-1}2}.
\end{align*}
We obtain, for all 
$0\leq |\beta|\leq q_0 -1$ and $|x|\leq R$,
\begin{equation}\label{V1bis}
|\partial_s^\alpha\partial_x^\beta V_0|\lesssim V_0^{1+(\alpha+\frac{|\beta|}k)\frac{p-1}2},
\end{equation}
which proves \eqref{V1} for $\rho=1$.

We use the notation $\nub=(\alpha,\beta_1,\ldots,\beta_N)$ as in the context of formula \eqref{fdb}.
Let $n=|\nub|\geq 1$. Then, by \eqref{fdb}, for $\rho\in \R$,
\begin{equation*}
\partial^{\nub}( V_0^{\rho})=
\sum_{r=1}^n \left[\rho\cdots(\rho-r+1)\right]V_0^{\rho-r}\sum_{P(\nub,r)}(\nub!)
\prod_{ \ell =1}^n \frac{\left(\partial^{\nub_ \ell }V_0\right)^{r_ \ell }}{(r_ \ell !)(\nub_ \ell !)^{r_ \ell }}
\end{equation*}
where
\begin{align*}
P(\nub,r)
&=\Big\{(r_1,\ldots,r_n; \nub_1,\ldots,\nub_n) : \mbox{there exists $1\leq m\leq n$ such that}\\
&\qquad r_ \ell =0 \mbox{ and } \nub_ \ell =0 \mbox{ for $1\leq  \ell \leq n-m$} ;\, r_ \ell >0 \mbox{ for $n-m+1\leq  \ell \leq n$};\\
&\qquad \mbox{and } 0\prec\nub_{n-m+1}\prec\cdots\prec\nub_n \mbox{ are such that }
\sum_{ \ell =1}^n r_ \ell =r,\ \sum_{ \ell =1}^n r_ \ell  \nub_ \ell  = \nub \Big\}.
\end{align*}
Using \eqref{V1bis} and $\sum_{ \ell =1}^n r_ \ell =r$, $\sum_{ \ell =1}^n r_ \ell \alpha_ \ell =\alpha$,
$\sum_{ \ell =1}^n r_ \ell \beta_ \ell =\beta$ in $P(\nub,r)$, we estimate
\begin{align*}
|\partial^{\nub}( V_0^{\rho})|&\lesssim
\sum_{r=1}^n V_0^{\rho-r}\sum_{P(\nub,r)}
\prod_{ \ell =1}^n V_0^{r_ \ell \left[ 1+(\alpha_ \ell +\frac{|\beta_ \ell |}{k})\frac{p-1}2\right]}\\
&\lesssim\sum_{r=1}^n V_0^{\rho-r}V_0^{r+(\alpha+\frac{|\beta|}{k})\frac{p-1}2}\lesssim V_0^{\rho+(\alpha+\frac{|\beta|}{k})\frac{p-1}2}.
\end{align*}

\smallskip

Proof of \eqref{V2}.
We estimate the three terms in \eqref{def:E0}.
It follows from Leibniz's formula \eqref{lbz}, the properties of $\psi$, 
$V_0\gtrsim 1$, and estimate \eqref{V1} that, for $|\beta|\leq q_0 -3$, and $|x|\leq R$,
\begin{align*}
|\partial_s^\alpha\partial_x^\beta[\nabla \psi\cdot \nabla \partial_s V_0]|
&\lesssim V_0^{1+(1+\alpha+\frac{1+|\beta|}k)\frac{p-1}2} ,
\\
|\partial_s^\alpha\partial_x^\beta[(\Delta \psi) \partial_s V_0]|
&\lesssim V_0^{1+(1+\alpha+\frac{|\beta|}k)\frac{p-1}2} ,\\
|\partial_s^\alpha\partial_x^\beta [\Delta V_0]|
& \lesssim V_0^{1+(\alpha+\frac{2+|\beta|}k)\frac{p-1}2} .
\end{align*}
Using once more that $V_0\gtrsim 1$ for $|x|\leq R$ and $k\geq 1$, these estimates imply \eqref{V2}.

\smallskip

Proof of \eqref{V3}. It follows from the properties of the functions $\psi$ and $A$ that
$V_0(s,x)= {\kappa _0} (s+|x|^k)^{-\frac{2}{p-1}}$ for any $|x|\geq R$. 
Estimate \eqref{V3} follows immediately. Then, we have, for any $|x|\geq R$,
\begin{gather*}
 |\partial_s^\alpha\partial_x^\beta[\nabla \psi\cdot \nabla \partial_s V_0]|
 =0,
\quad 
|\partial_s^\alpha\partial_x^\beta[(\Delta \psi) \partial_s V_0]|
 =0,\\
 |\partial_s^\alpha\partial_x^\beta [\Delta V_0]|
 \lesssim |x|^{-(\frac {2}{p-1}+\alpha)k-|\beta|-2},
\end{gather*}
which implies \eqref{V4}.

\smallskip

Finally, we prove~\eqref{dtVrsigma}.
Since $ |x_0| <1$, we have for $s$ small $ |\partial _t V_0| \gtrsim s ^{- \frac {p+1} {p-1} }$, and~\eqref{dtVrsigma} follows.
\end{proof}

\subsection{Refined blow up ansatz}\label{s23}
Starting from $V_0$, we define by induction a refined ansatz to the nonlinear wave equation \eqref{eq:v}.

Let $V_0$ be defined in \eqref{def:V0} and let $\Ens_0$ be defined in \eqref{def:E0}. Let $s_0=1$.
For $j\geq 1$, let
\begin{align*}
v_j & = - \frac1{3p+1} \left[\frac{2(p+1)}{1-|\nabla\psi|^2}\right]^{\frac 12}
\left(V_0^{\frac {p+1}2} \int_0^s V_0^{-p} \Ens_{j-1} ds' 
+ V_0^{-p} \int_{s}^{s_{j-1}} V_0^{\frac {p+1}2} \Ens_{j-1} ds'\right),
\\
V_j&=V_0+ \sum_{ \ell =1}^j \chi_ \ell  v_ \ell ,\\
\Ens_j& =-(1-|\nabla \psi|^2)\partial_{ss}V_j +2 \nabla \psi\cdot \nabla \partial_s V_j
+(\Delta \psi) \partial_s V_j + \Delta V_j+f(V_j),
\end{align*}
where $\chi_j(x)=\chi( A(x)/{r_j})$ and $0<r_j\leq 1$, $0<s_j\leq 1$ are parameters to be defined for each $j=1,\ldots,J$.
Since $V_0$ is of class $\cont^\infty$ in $s$ and of class $\cont^{q_0-1}$ in $x$, the above expressions make sense as continuous functions for $j$ such that $j \le J$. This restriction is due to the spatial derivatives in $V_j$ in the expression of $\Ens_j$.

\begin{lem}\label{le:3}
There exist $0<r_J\leq\cdots\leq r_1\leq 1$ and $0<s_J\leq\cdots\leq s_1\leq 1$ such that
for any $0\leq j\leq J$,
for any $\alpha\in \N$, $\beta\in \N^N$, $0<s\leq s_j$, $x\in \R^N$, the following hold:
\begin{enumerate}
\item If $1\leq j\leq J$, $|\beta|\leq q_0-2j -1$ and $|x|\leq R$, then
\begin{equation}\label{v4}
|\partial_s^\alpha\partial_x^\beta v_j| \lesssim V_0^{1+(-j+\alpha+\frac{j+|\beta|}k)\frac{p-1}2};
\end{equation}
\item If $1\leq j\leq J$, then
\begin{align}
&|V_j-V_0|\leq \frac 14 (1-2^{-j}) V_0 ,\quad |V_j-V_0|\leq (1-2^{-j})(1+V_0)^{-\frac{p-1}4}V_0,\label{v5}\\
&|\partial_s V_j-\partial_s V_0|\lesssim V_0^{1+\frac{p-1}{2k}};\label{v5bis}
\end{align}
\item If $|\beta|\leq q_0-2(j+1) -1$ and $|x|\leq R$, then
\begin{equation}\label{v1}
|\partial_s^\alpha\partial_x^\beta \Ens_j|\lesssim V_0^{\frac{p+1}2+(-j+\alpha+\frac{1+j+|\beta|}k)\frac{p-1}2};
\end{equation}
\item If $|x|> R$, then
\begin{align}
|\partial_s^\alpha\partial_x^\beta V_j|&\lesssim |x|^{-(\frac {2}{p-1}+\alpha)k-|\beta|};\label{v2}\\
|\partial_s^\alpha\partial_x^\beta \Ens_j|&\lesssim |x|^{-(\frac {2}{p-1}+\alpha)k-|\beta|-2}.\label{v3}
\end{align}
\end{enumerate}
\end{lem}

\begin{rem} \label{eRems2} 
To complete the energy control in Section~\ref{sec:3}, we need an error estimate of the form $ \| {\mathcal E}_J \| _{ L^2 }\lesssim s^{ \frac {2} {p-1}+ \delta }$, as in~\cite{CaMaZhwave1} (see~\eqref{e:EJ}), as well as an estimate of the form $ \| \partial _s {\mathcal E}_J \| _{ L^2 }\lesssim s^{-1 + \delta  }$ (see the proof of~\eqref{e:dK}), with $\delta >0$. 
This requires a sufficiently large $J$, see~\eqref{defJ}, and then a sufficiently large $k$. 
Compared with Lemma~2.3 (see also Remark~2.4) in~\cite{CaMaZhwave1}, we need twice as many steps. This is due to the terms depending on $\partial _s V_j$ in the expression of the error term ${\mathcal E}_j$.
These necessary restrictions have the important consequence that the minimal regularity of the hypersurface that we can consider in Theorem~\ref{TH:2} depends on $p$, see~\eqref{defJ}.
\end{rem} 

\begin{proof} [Proof of Lemma~$\ref{le:3}$] We observe that \eqref{v1}, \eqref{v2} and \eqref{v3} for $j=0$ are exactly \eqref{V2}, \eqref{V3} and \eqref{V4} in Lemma~\ref{le:0}.
We proceed by induction on $j$: for any $1\leq j\leq J$, we prove that estimate~\eqref{v1} for $\Ens_{j-1}$ implies
\eqref{v4}--\eqref{v1} for $v_j$, $V_j$ and $\Ens_j$. Let $s_0=1$.

\smallskip

Proof of \eqref{v4}. Let $1\leq j\leq J$.
First, assuming \eqref{v1} for $\Ens_{j-1}$, we show the following estimates related to the two components of $v_j$,
for $|\beta|\leq q_0-2j -1$, $0<s<s_{j-1}$ and $|x|\leq R$,
\begin{align}
\left| \partial_s^\alpha\partial_x^{\beta}\left(\int_0^s V_0^{-p} \Ens_{j-1} ds'\right)\right|&
\lesssim V_0^{-\frac{p-1}2+(-j+\alpha+\frac{j+|\beta|}k)\frac{p-1}2}; \label{v6}\\
\left| \partial_s^\alpha\partial_x^{\beta}\left(\int_{s}^{s_{j-1}} V_0^{\frac {p+1}2}\Ens_{j-1}ds'\right)\right|&
\lesssim V_0^{p+1+(-j+\alpha+\frac{j+|\beta|}k)\frac{p-1}2}. \label{v7}
\end{align}
Indeed, we have by Leibniz's formula
\begin{equation*}
\partial_s^\alpha\partial_x^{\beta} \left(V_0^{-p}\Ens_{j-1}\right)
=\sum_{\alpha'\leq \alpha}\sum_{\beta'\leq \beta} 
\binom{\alpha}{\alpha'}\binom{\beta}{\beta'}
\left(\partial_s^{\alpha'}\partial_x^{\beta'} (V_0^{-p})\right)
\left(\partial_s^{\alpha-\alpha'}\partial_x^{\beta-\beta'}\Ens_{j-1}\right),
\end{equation*}
and thus
using \eqref{V1} and \eqref{v1} for $\Ens_{j-1}$, we obtain
\begin{align*}
\left|\partial_s^\alpha\partial_x^{\beta} \left(V_0^{-p}\Ens_{j-1}\right)\right|
&\lesssim \sum_{\alpha'\leq \alpha}\sum_{\beta'\leq \beta} V_0^{-p+(\alpha'+\frac{|\beta'|}k)\frac{p-1}2} V_0^{\frac{p+1}2+(1-j+\alpha-\alpha'+\frac{j+|\beta-\beta'|}k)\frac{p-1}2}\\
&\lesssim V_0^{(-j+\alpha+\frac{j+|\beta|}k)\frac{p-1}2}.
\end{align*}
For $\alpha= 0$, $|\beta|\leq q_0 -1\leq k-1$ and $1\leq j\leq J$, we note that
\[
\left|\partial_x^{\beta} \left(V_0^{-p}\Ens_{j-1}\right)\right|
\lesssim V_0^{-a\frac{p-1}2} \lesssim (s+A)^{a},
\]
where
\[
a = j- \frac{j+|\beta|}k =j \left(1-\frac1k\right)-\frac{|\beta|}k\geq 0.
\]
This means that we can integrate this term on $(0,s)$ for $0<s\leq s_{j-1}$.
We obtain
\begin{equation*}
\int_0^s \left|\partial_x^\beta \left(V_0^{-p} \Ens_{j-1}\right)\right| ds' \lesssim 
(s+A)^{a+1}\lesssim V_0^{-\frac{p-1}2+(-j+\frac{j+|\beta|}k)\frac{p-1}2}.
\end{equation*}
For $\alpha\geq1$,
\begin{equation*}
\left|\partial_s^\alpha\partial_x^\beta\left(\int_0^s V_0^{-p} \Ens_{j-1} ds'\right)\right|
=\left|\partial_s^{\alpha-1}\partial_x^{\beta}\left(V_0^{-p} \Ens_{j-1}\right)\right| \lesssim
V_0^{-\frac{p-1}2+(-j+\alpha+\frac{j+|\beta|}k)\frac{p-1}2},
\end{equation*}
which proves \eqref{v6}.
Similarly, using Leibniz's formula, we check the estimate
\begin{equation*}
\left|\partial_s^\alpha\partial_x^{\beta} \left(V_0^{\frac{p+1}2}\Ens_{j-1}\right)\right|
\lesssim V_0^{p+1+(1-j+\alpha+\frac{j+|\beta|}k)\frac{p-1}2}.
\end{equation*}
In particular, for $\alpha=0$,
\begin{equation*}
\left|\partial_x^{\beta} \left(V_0^{\frac{p+1}2}\Ens_{j-1}\right)\right|
\lesssim V_0^{b\frac{p-1}2} \lesssim (s+A)^{-b},
\end{equation*}
where, using $1\leq j\leq J\leq \frac{2p+2}{p-1}$,
\[
b=\frac{2p+2}{p-1}+ 1-j+\frac{j+|\beta|}k \geq 1+ \frac{j+|\beta|}k >1.
\]
Thus, by integration on $(s,s_{j-1})$,
\begin{equation*}
\left| \partial_x^\beta \left(\int_s^{s_{j-1}} V_0^{\frac{p+1}2}\Ens_{j-1}ds'\right)\right|
\lesssim (s+A)^{-b+1}\lesssim V_0^{p+1+(-j+\frac{j+|\beta|}k)\frac{p-1}2}.
\end{equation*}
For $\alpha\geq 1$,
\begin{equation*}
\left|\partial_s^\alpha\partial_x^{\beta}\left(\int_s^{s_{j-1}} V_0^{\frac{p+1}2}\Ens_{j-1} ds'\right)\right|
=\left|\partial_s^{\alpha-1}\partial_x^\beta\left(V_0^{\frac{p+1}2}\Ens_{j-1}\right)\right| \lesssim
V_0^{p+1+(-j+\alpha+\frac{j+|\beta|}k)\frac{p-1}2},
\end{equation*}
which proves \eqref{v7}.

Using estimates \eqref{V1}, \eqref{v6}, \eqref{v7} and again Leibniz's formula, we obtain, for all $s\in (0,s_{j-1}]$,
\begin{align*}
\left|\partial_s^\alpha\partial_x^{\beta}\left(V_0^{\frac{p+1}2}\int_0^s V_0^{-p} \Ens_{j-1} ds'\right)\right|&
\lesssim V_0^{1+(-j+\alpha+\frac{j+|\beta|}k)\frac{p-1}2},\\
\left|\partial_s^\alpha\partial_x^{\beta}\left(V_0^{-p}\int_{s}^{s_{j-1}} V_0^{\frac {p+1}2}\Ens_{j-1}ds'\right)\right|&
\lesssim V_0^{1+(-j+\alpha+\frac{j+|\beta|}k)\frac{p-1}2}.
\end{align*}
These estimates implies \eqref{v4} for $v_j$ on $(0,s_{j-1}]$.

\smallskip

Proof of \eqref{v5}--\eqref{v5bis}. 
For $j=1$, we prove \eqref{v5} as a consequence of \eqref{v4}.
For $2\leq j\leq J$, we prove \eqref{v5} as a consequence of \eqref{v4} for $j$ and \eqref{v5} for $j-1$.

For $|x|>R\geq 1$, \eqref{on:A} implies $A(x)\geq 2^k\geq2 r_1\geq \cdots\geq2 r_j$, 
thus $\chi_j=0$ and $V_j=V_0$.

For $0<s\leq s_{j-1}$ and $|x|<R$, by \eqref{v4} with $\alpha=0$ and $\beta=0$,
using the definition of $\chi_j$ and the bound $V_0\gtrsim 1$, we have
\begin{equation*}
\chi_j |v_j|\lesssim \chi_j V_0^{1-j(1-\frac1k)\frac{p-1}2}\lesssim \chi_j V_0^{1-(1-\frac1k)\frac{p-1}2} 
\lesssim \chi_j (s+A)^{1-\frac1k} V_0\lesssim (s+r_j)^{1-\frac1k}V_0.
\end{equation*}
Choosing $0<r_j\leq 1$ and $0<s_j\leq s_{j-1}$ sufficiently small, we impose, for $s\in (0,s_j]$,
\begin{equation*}
\chi_j|v_j|\leq 2^{-j-2}V_0\quad\hbox{and}\quad \chi_j|v_j|\leq 2^{-j}(1+V_0)^{-\frac{p-1}4}V_0.
\end{equation*}
In the case $j=1$, this proves \eqref{v5}.
For $j\geq 2$, combining this estimate with \eqref{v5} for $j-1$, we find, for all $s\in (0,s_j]$ and $x\in \R^N$, 
\begin{equation*}
\sum_{ \ell =1}^j \chi_ \ell |v_ \ell |\leq\frac14(1-2^{-j})V_0\quad\hbox{and}\quad \sum_{ \ell =1}^j\chi_ \ell |v_ \ell |\leq(1-2^{-j})(1+V_0)^{-\frac{p-1}4}V_0,
\end{equation*}
which is \eqref{v5}.

To prove \eqref{v5bis}, we note that by \eqref{v4}, and using $A\lesssim V_0^{-\frac{p-1}2}$,
\begin{equation*}
\sum_{ \ell =1}^j \chi_ \ell |\partial_sv_ \ell |\lesssim \sum_{ \ell =1}^j\chi_ \ell  V_0^{1+(1- \ell (1-\frac 1k))\frac{p-1}2}
\lesssim V_0^{1+\frac{p-1}{2k}}.
\end{equation*}

\smallskip

Proof of \eqref{v1}. 
Note that \eqref{v1} for $j=0$ was already checked.
Now, for $1\leq j\leq J$, we prove \eqref{v1} for $\Ens_j$ assuming \eqref{v1} for $\Ens_{j-1}$, \eqref{v4} for $v_j$ and \eqref{v5} for $V_j$. This suffices to complete the induction argument.

By direct computations, we briefly check that the function~$v_j$ satisfies
\begin{equation}\label{eq:vj}
(1-|\nabla \psi|^2) \partial_{ss} v_j=f'(V_0) v_j+\Ens_{j-1}.
\end{equation}
Indeed, we have
\begin{multline*}
\partial_s v_j =- \frac1{3p+1} \left[\frac{2(p+1)}{1-|\nabla\psi|^2}\right]^{\frac 12}
\biggl(\frac{p+1}2\partial_s V_0 V_0^{\frac {p-1}2} \int_0^s V_0^{-p} \Ens_{j-1} ds'\\ 
 -p\partial_s V_0 V_0^{-p-1} \int_{s}^{s_{j-1}} V_0^{\frac {p+1}2} \Ens_{j-1} ds'
\biggr),
\end{multline*}
and thus, using \eqref{eq:V},
\begin{multline*}
(1-|\nabla \psi|^2)^{\frac 12}\partial_{s} v_j =  \frac1{3p+1} \left[\frac{2(p+1)}{1-|\nabla\psi|^2}\right]^{\frac 12}
\biggl(\left(\frac{p+1}2\right)^{\frac 12} V_0^p
\int_0^s V_0^{-p} \Ens_{j-1} ds'\\ 
-p\left(\frac{p+1}2\right)^{-\frac 12} V_0^{-\frac{p+1}2} \int_{s}^{s_{j-1}} V_0^{\frac {p+1}2} \Ens_{j-1} ds'
\biggr).
\end{multline*}
Differentiating in $s$ again, and using \eqref{eq:V}, we obtain
\begin{equation*}
(1-|\nabla \psi|^2)\partial_{ss} v_j =pV_0^{p-1} v_j+\Ens_{j-1},
\end{equation*}
which is \eqref{eq:vj}.

Using \eqref{eq:vj}, $V_j=V_{j-1}+\chi_j v_j$ and the definition of $\Ens_{j-1}$, we have
\begin{align*}
\Ens_j
&=\Ens_{j-1} 
-\chi_j (1-|\nabla \psi|^2)\partial_{ss}v_j+2 \nabla \psi\cdot \nabla \partial_s (\chi_j v_j)
+(\Delta \psi) \partial_s(\chi_j v_j) +\Delta (\chi_j v_j)\\ 
&\quad + f(V_j)-f(V_{j-1})\\
&= (1-\chi_j) \Ens_{j-1}+2 \nabla \psi\cdot \nabla \partial_s (\chi_j v_j)
+(\Delta \psi) \partial_s(\chi_j v_j) +\Delta (\chi_j v_j)\\ 
&\quad+f(V_j)-f(V_{j-1})-f'(V_0)\chi_j v_j.
\end{align*}
We estimate each term of the right-hand side above for $|x|\leq R$.

For the first term, recall that for $A\leq r_j$, and any $\beta'$, $1-\chi_j=0$ and $\partial_x^{\beta'} \chi_j=0$.
Moreover, for $0< s\leq 1$, for $x$ such that $A(x)>r_j$ and $|x|\leq R$, one has $A \approx 1$ and $V_0\approx 1$.
Thus, using \eqref{v1} for $\Ens_{j-1}$, we find
\begin{equation*}
|\partial_s^\alpha\partial_x^\beta [(1-\chi_j) \Ens_{j-1}]|\lesssim V_0^{\frac{p+1}2+(-j+\alpha+\frac{1+j+|\beta|}{k})\frac{p-1}2}.
\end{equation*}

Now, we treat the next three terms in the expression of $\Ens_{j}$.
By Leibniz's formula, the properties of~$\psi$ and $\chi_j$, \eqref{v4} and then $V_0\gtrsim 1$, we have,
for $0<s\leq s_j$ and $|x|<R$,
\begin{align*}
\left|\partial_s^\alpha\partial_x^\beta[\nabla \psi\cdot \nabla \partial_s (\chi_j v_j)]\right|
& \lesssim \sum_{\alpha'=1}^{\alpha+1}\sum_{|\beta'|\leq|\beta|+1} \left|\partial_s^{\alpha'}\partial_x^{\beta'}v_j\right|
\lesssim V_0^{1+(-j+1+\alpha+\frac{1+j+|\beta|}k)\frac{p-1}2};\\
\left|\partial_s^\alpha\partial_x^\beta[(\Delta \psi) \partial_s(\chi_j v_j)]\right|
& \lesssim \sum_{\alpha'=1}^{\alpha+1}\sum_{|\beta'|\leq|\beta|}\left|\partial_s^{\alpha'}\partial_x^{\beta'}v_j\right|
\lesssim V_0^{1+(-j+1+\alpha+\frac{j+|\beta|}k)\frac{p-1}2};\\
\left|\partial_s^\alpha\partial_x^\beta[\Delta (\chi_j v_j)]\right|
& \lesssim \sum_{\alpha'=1}^{\alpha}\sum_{|\beta'|\leq|\beta|+2}\left|\partial_s^{\alpha'}\partial_x^{\beta'}v_j\right|
\lesssim V_0^{1+(-j+\alpha+\frac{2+j+|\beta|}k)\frac{p-1}2};
\end{align*}
we see that these three terms are estimated by $V_0^{\frac{p+1}2+(-j+\alpha+\frac{1+j+|\beta|}k)\frac{p-1}2}$.

Finally, we estimate $\partial_s^\alpha\partial_x^\beta[f(V_j)-f(V_{j-1})-f'(V_0)\chi_j v_j]$ using Taylor expansions on 
$f$ and its derivatives. 
We start with the case $\alpha=\beta=0$.
Recall that by \eqref{v5}, we have $0<\frac 34 V_0 \leq V_j\leq \frac 54 V_0$.
The following Taylor expansions hold:
\begin{equation*}
\left|f(V_j)-f(V_{j-1})-f'(V_{j-1})\chi_j  v_j\right|
\lesssim \chi_j  ^{2} V_0^{p-2}v_j^2,
\end{equation*}
and
\begin{equation*}
|f'(V_{j-1})-f'(V_0)|\lesssim V_0^{p-2}\sum_{ \ell =1}^{j-1} \chi_ \ell |v_ \ell |.
\end{equation*}
These estimates imply
\begin{equation*}
\left|f(V_j)-f(V_{j-1})-f'(V_0)\chi_j v_j\right|
\lesssim \chi_j V_0^{p-2}|v_j| \sum_{ \ell =1}^{j} \chi_ \ell |v_ \ell |.
\end{equation*}
For $1\leq  \ell \leq j$, using \eqref{v4} and next $V_0\gtrsim 1$, we have
\begin{align*}
 V_0^{p-2}|v_j||v_ \ell |
&\lesssim V_0^{p-2}V_0^{1-j(1-\frac1k)\frac{p-1}2}V_0^{1- \ell (1-\frac1k)\frac{p-1}2}
\\&\lesssim V_0^{{p-(j+ \ell )(1-\frac1k)\frac{p-1}2}}\lesssim V_0^{{p-(j+1)(1-\frac1k)\frac{p-1}2}} .
\end{align*}
Thus, 
$\left|f(V_j)-f(V_{j-1})-f'(V_0)\chi_j v_j\right|\lesssim V_0^{\frac{p+1}2+(-j+\frac{1+j}k)\frac{p-1}2}$ is proved.

Now, we consider the case $|\alpha|+|\beta|\geq 1$.
By the Taylor formula with integral remainder we have for any $V$ and $w$
\begin{equation*}
f(V+w)-f(V)-f'(V)w=w^2\int_0^1 (1-\theta)f''(V+\theta w)d\theta.
\end{equation*}
Therefore, using the notation $\nub=(\alpha,\beta_1,\ldots,\beta_N)$,
by the Leibniz formula~\eqref{lbz},
\begin{multline*}
\partial^{\nub} [f(V+w)-f(V)-f'(V)w]
\\=\sum_{\nub'\leq \nub} \binom{\nub}{\nub'} \left(\partial^{\nub-\nub'}(w^2)\right)
 \int_0^1 (1-\theta) \partial^{\nub'} [f''(V+\theta w) ] d\theta
\end{multline*}
and, by the Faa di Bruno formula \eqref{fdb}, for $\nub'\neq 0$, denoting $n'=|\nub'|$,
\begin{equation} \label{OMF} 
\partial^{\nub'} [f''(V+\theta w) ]=
 \sum_{r=1}^{n'} f^{(r+2)}(V+\theta w) \sum_{P(\nub',r)}(\nub'!)
\prod_{ \ell =1}^{n'} \frac{\left(\partial^{\nub_ \ell }(V+\theta w)\right)^{r_ \ell }}{(r_ \ell !)(\nub_ \ell !)^{r_j}}.
\end{equation}
To estimate the term $\partial^\nub [f(V_j)-f(V_{j-1})-f'(V_{j-1})\chi_j v_j]$,
we apply these formulas to $V=V_{j-1}$ and $w=\chi_j v_j$.
First, for $\nub'\leq \nub$, using \eqref{v4} and the properties of $\chi$, we obtain
\begin{align*}
\left|\partial^{\nub-\nub'}\left[(\chi_j v_j)^2\right]\right|
&\lesssim \sum_{\nub''\leq \nub-\nub'} 
\left|\partial^{\nub''}(\chi_j v_j)\right|\left|\partial^{\nub-\nub'-\nub''}(\chi_j v_j)\right|\\
&\lesssim V_0^{2+ (\alpha-\alpha'-2 j+\frac{2j+|\beta-\beta'|}k)\frac{p-1}2}.
\end{align*}
Thus, for $\nub'=0$ and $\theta\in [0,1]$, from \eqref{v5}, we obtain
\begin{align*}
\left|\partial^{\nub}\left[(\chi_j v_j)^2\right] f''(V_{j-1}+\theta \chi_j v_j)\right|
& \lesssim V_0^{2+ (\alpha-2j+\frac{2j+|\beta|}k)\frac{p-1}2} V_0^{p-2}\\
& \lesssim V_0^{\frac{p+1}2+(\alpha-j+\frac{1+j+|\beta|}k)\frac{p-1}2}.
\end{align*}
Second, for $\nub'\neq 0$, $\nub'\leq \nub$ and $\theta\in [0,1]$, from formula~\eqref{OMF}, using \eqref{V1} and \eqref{v5}, we have
(the definition of $P(\nub',r)$ implies $\sum_{ \ell =1}^{n'} r_ \ell =r$, $\sum_{ \ell =1}^{n'} r_ \ell \nub_ \ell =\nub'$)
\begin{align*}
|\partial^{\nub'} [f''(V_{j-1}+\theta \chi_j v_j)]|
&\lesssim \sum_{r=1}^{n'} V_0^{p-r-2} \sum_{P(\nub',r)} 
\prod_{ \ell =1}^{n'} \left(V_0^{1+(\alpha_ \ell +\frac{|\beta_ \ell |}{k})\frac{p-1}2}\right)^{r_ \ell }\\
&\lesssim \sum_{r=1}^{n'} V_0^{p-r-2}V_0^{r+(\alpha'+\frac{|\beta'|}{k})\frac{p-1}2}
\lesssim V_0^{p-2+(\alpha'+\frac{|\beta'|}{k})\frac{p-1}2}.
\end{align*}
Thus, we have proved
\begin{equation*}
\left|\partial^{\nub-\nub'}\left[(\chi_j v_j)^2\right] \partial^{\nub'}\left[f''(V_{j-1}+\theta \chi_j v_j)\right]\right|
\lesssim V_0^{\frac{p+1}2+(-j+\alpha+\frac{1+j+|\beta|}k)\frac{p-1}2};
\end{equation*}
and so by integration in $\theta\in [0,1]$,
\begin{equation}\label{UN}
\left|\partial^{\nub} [f(V_j)-f(V_{j-1})-f'(V_{j-1})\chi_j v_j]\right| 
 \lesssim V_0^{\frac{p+1}2+(-j+\alpha+\frac{1+j+|\beta|}k)\frac{p-1}2}.
\end{equation}

\smallskip

We now estimate $ \partial^\nub [f(V_j)-f(V_{j-1})-f'(V_0)\chi_j v_j] $. For any $V,W,w$, we have
\begin{equation*}
f'(V)-f'(W)=(V-W) \int_0^1 f''(W+\theta(V-W)) d\theta,
\end{equation*}
and thus
\begin{multline*}
\partial^{\nub}[w(f'(V)-f'(W))]
\\=\sum_{\nub'\leq \nub} \binom{\nub}{\nub'} \left(\partial^{\nub-\nub'}[w(V-W)]\right)
\int_0^1 \partial^{\nub'} [f''(W+\theta(V-W)) ] d\theta.
\end{multline*}
Moreover, for $\nub'\neq 0$, formula~\eqref{OMF} (with $V$ replaced by $W$, and $w$ by $V-W$) yields
\begin{multline*}
\partial^{\nub'} [f''(W+\theta(V-W)) ]\\ =
 \sum_{r=1}^{n'} f^{(r+2)}(W+\theta(V-W))\sum_{P(\nub',r)}(\nub'!)
\prod_{ \ell =1}^{n'} \frac{\left(\partial^{\nub_ \ell }(W+\theta(V-W))\right)^{r_ \ell }}{(r_ \ell !)(\nub_ \ell !)^{r_ \ell }}.
\end{multline*}
To estimate the term $\partial^\nub [\chi_j v_j(f'(V_{j-1})-f'(V_0))]$,
we apply these formulas to $V=V_{j-1}$, $W=V_0$ and $w=\chi_j v_j$.

For $\nub'\leq \nub$, using \eqref{v4} and Leibniz's formula, we have, for $1\leq  \ell \leq j-1$,
\begin{align*}
\left| \partial^{\nub-\nub'}[\chi_j v_j \chi_ \ell v_ \ell ]\right|
&\lesssim V_0^{2+ (-j- \ell +\alpha-\alpha'+\frac {j+ \ell +|\beta-\beta'|}k)\frac{p-1}2}\\
&\lesssim V_0^{-\frac{p-5}2+(-j+\alpha-\alpha'+\frac {1+j+|\beta-\beta'|}k)\frac{p-1}2}
\end{align*}
For $\nub'=0$ and $\theta\in [0,1]$, from \eqref{v5}, we obtain
\begin{equation*}
\left|\partial^{\nub}\left[\chi_j v_j(V_{j-1}-V_0)\right] f''(V_0+\theta (V_{j-1}-V_0))\right|
\lesssim V_0^{\frac{p+1}2+(-j+\alpha+\frac{1+j+|\beta|}k)\frac{p-1}2}.
\end{equation*}
Second, for $\nub'\neq 0$, $\nub'\leq \nub$ and $\theta\in [0,1]$, by formula~\eqref{OMF}, using \eqref{V1}, \eqref{v4} and \eqref{v5}, we have
\begin{align*}
|\partial^{\nub'}[f''(V_0+\theta (V_{j-1}-V_0))]|
&\lesssim \sum_{r=1}^{n'} V_0^{p-r-2} \sum_{P(\nub',r)} 
\prod_{ \ell =1}^{n'} \left(V_0^{1+(\alpha_ \ell +\frac{|\beta_ \ell |}{k})\frac{p-1}2}\right)^{r_ \ell }\\
&\lesssim \sum_{r=1}^{n'} V_0^{p-r-2} V_0^{r+(\alpha'+\frac{|\beta'|}{k})\frac{p-1}2}
\lesssim V_0^{p-2+(\alpha'+\frac{|\beta'|}{k})\frac{p-1}2}.
\end{align*}
Thus, we obtain
\begin{equation*}
\left|\partial^{\nub-\nub'}\left[\chi_j v_j(V_{j-1}-V_0)\right] \partial^{\nub'}\left[f''(V_0+\theta (V_{j-1}-V_0))\right]\right|
\lesssim V_0^{\frac{p+1}2+(-j+\alpha+\frac{1+j+|\beta|}k)\frac{p-1}2}.
\end{equation*}
Integrating in $\theta\in [0,1]$ and summing in $\nub'\leq \nub$, we obtain
\begin{equation}\label{DEUX}
\left| \partial^\nub [\chi_j v_j(f'(V_{j-1})-f'(V_0))]\right|
\lesssim V_0^{\frac{p+1}2+(-j+\alpha+\frac{1+j+|\beta|}k)\frac{p-1}2}.
\end{equation}
Combining \eqref{UN} and \eqref{DEUX}, we have proved for $s\in (0,s_{j}]$, $|x|\leq R$,
\begin{equation*}
\left| \partial^\nub [f(V_j)-f(V_{j-1})-f'(V_0)\chi_j v_j] \right|
\lesssim V_0^{\frac{p+1}2+(-j+\alpha+\frac{1+j+|\beta|}k)\frac{p-1}2}.
\end{equation*}

In conclusion, we have estimated all terms in the expression of $\Ens_j$ and \eqref{v1} for $j$ is proved.

\smallskip

Proof of \eqref{v2}--\eqref{v3}. For $|x|>R\geq 2$, \eqref{on:A} implies $A(x)\geq 2^k\geq2 r_1\geq \cdots\geq2 r_j$, 
thus $\chi_j=0$ and $V_j=V_0$, $\Ens_j=\Ens_0$.
Thus, \eqref{v2}--\eqref{v3} follow from \eqref{V3}--\eqref{V4}.
\end{proof}

\section{Construction of a solution of the transformed equation~$\eqref{eq:v}$}\label{sec:3}
Let the function $\chi$ be given by \eqref{def:chi}, let  $\psi \in \cont^{q_0} (\R^N ,\R)$, where $q_0 $ is defined by~\eqref{th2:1}, satisfy~\eqref{on:psi}, let $J$, $q_0$ and $k$ be as in \eqref{defJ}-\eqref{onqk}.
Set
\begin{equation}\label{on:lam}
\lambda=\min\left\{\frac 12 \left(J-\frac{p+3}{p-1}\right); \frac 1p\right\}\in \left(0,\frac 12\right],
\end{equation}
and impose the following additional condition on $k$
\begin{equation}\label{on:kk}
{k\geq \frac{ 2 [ p+1+\lambda(p-1) ]}{\lambda(p-1)}.}
\end{equation}
Recall that $A:\R^N\to[0,+\infty[$ is defined by~\eqref{fExfnA}, and let $V_J$ be defined as in Section~\ref{s23}.

Our main result of this section is the following.

\begin{prop} \label{eExisV} 
Assume that 
\begin{equation}\label{c:psi}
\|\nabla \psi\|_{L^\infty}
\leq\frac \lambda 8  \frac{p-1}{p+1} .
\end{equation}
There exist $0 < \delta _0 <1$ and a function 
\begin{equation} \label{fRegv} 
v\in C ( (0,\delta_0),H^2(\R^N))\cap  C^1 ((0,\delta_0),H^1(\R^N))\cap C^2 ((0,\delta_0),L^2(\R^N))
\end{equation} 
which is a solution of~\eqref{eq:v} in $C ((0,\delta_0),L^2(\R^N))$, and which satisfies
\begin{equation}\label{festV1}
 \| (v  - V_J ) (s) \|_{H^2}^2+\|\partial_s (v  - V_J ) (s)\|_{H^1}^2  \leq C s^{ \lambda}
\end{equation}
for all $0<s<\delta _0$, with $\lambda $ given by~\eqref{on:lam}. 
In addition, there exist a constant $C$ and a function $g \in L^\infty ((0,\delta _0), H^1 (\R^N )  )$ such that 
\begin{equation} \label{fNFestinf} 
 | \partial _s v |^2 -  | \nabla v |^2 \ge \frac {1} {4}  | \partial _s V_0 |^2 - C - g ^2
\end{equation} 
a.e. on $(0,\delta _0) \times \R^N $.
\end{prop} 

We construct the solution $v$ of Proposition~\ref{eExisV} by a compactness argument. 
For any $n$ large, let $S_n=\frac 1n< s_J$ and
\begin{equation*}
B_n=\sup_{s\in [S_n,s_J]}\|V_J(s)\|_{L^\infty} \quad
\mbox{so that}\quad \lim_{n\to\infty} B_n=\infty.
\end{equation*}
We let $n$ be sufficiently large so that $B_n\geq 1$, and we define the function $f_n:\R\to [0,\infty)$ by
\begin{equation*}
f_n(u)=f(u)\chi\left(\frac {|u|}{B_n}\right)\quad
\mbox{so that}\quad
f_n(u)=\begin{cases}f(u)&\mbox{for $|u|<B_n$}\\ 0 &\mbox{for $|u|>2B_n$.} \end{cases}
\end{equation*}
Let $F_n(u)=\int_0^u f_n(w) dw$.
Note that Taylor's estimates such as \eqref{taylor0}--\eqref{taylor} still hold for $F_n$ and $f_n$
with constants independent of $n$. 
We will refer to these inequalities for $F_n$ and $f_n$ with the same numbers~\eqref{taylor0}, \eqref{taylor1} and~\eqref{taylor}. 
In this proof, any implicit constant related the symbol $\lesssim$ is independent of $n$.

We define the sequence of solution $v_n$ of
\begin{equation}\label{eq:vn}\left\{
\begin{aligned}
&(1-|\nabla \psi|^2) \partial_{ss} v_n - 2 \nabla \psi\cdot \nabla \partial_s v_n
-(\Delta \psi) \partial_s v_n - \Delta v_n = f_n(v_n)\\
&v_n(S_n)=V_J(S_n),\quad \partial_s v_n(S_n)=\partial_s V_J(S_n).
\end{aligned}\right.
\end{equation}
The nonlinearity $f_n$ being globally Lipschitz,
the existence of a global solution $(v_n,\partial_s v_n)$ in $H^2\times H^1$ is a consequence of standard arguments
from semigroups theory, see   Appendix~\ref{sAppA}, and in particular Section~\ref{sAppAexist}.

We set, for all $s\in [S_n,s_J]$,
\begin{equation*}
v_n(s)=V_J(s)+w_n(s),
\end{equation*}
thus $(w_n,\partial_s w_n)\in \mathcal C([S_n,s_J],H^2(\R^N)\times H^1(\R^N))
\cap  \mathcal C^1([S_n,s_J],H^1(\R^N)\times L^2(\R^N))$.
The crucial step in the proof of Proposition~\ref{eExisV} is the following estimate.

\begin{prop}\label{pr:unif}
There exist $C>0$, $n_0>0$ and $0<\delta_0<1$ such that
\begin{equation}\label{e:unif}
 \|w_n(s)\|_{H^2}^2+\|\partial_s w_n(s)\|_{H^1}^2+\|\partial_{ss} w_n(s)\|_{L^2}^2\leq C (s-S_n)^{ \lambda}
\end{equation}
for all $n\geq n_0$ and  $s\in [S_n,S_n+\delta_0]$.
\end{prop}
\begin{proof}
 We fix $n\geq n_0$ large, and we denote $w_n$ simply by $w$ in this proof.
By \eqref{eq:vn} and the definition of $\Ens_J$, $w$ satisfies the equation
\begin{equation}\label{eq:w}
\left\{
\begin{aligned}
&(1-|\nabla \psi|^2) \partial_{ss} w -   2 \nabla \psi\cdot \nabla \partial_s w 
-(\Delta \psi) \partial_s w - \Delta w \\ & \hskip 4.5truecm = f_n(V_J+w)-f_n(V_J)+\Ens_J \\
&w(S_n)=0,\quad \partial_s w(S_n)=0.
\end{aligned}
\right.
\end{equation}
We define the auxiliary function $Q$ as follows
\[
Q=(1-\chi+V_0)^{p+1},
\]
where, by abuse of notation, we denote $\chi (x) = \chi (  |x| )$. Note that $Q \gtrsim 1$.
We make the following preliminary observation
\begin{align*}
\partial_{ss}w
&=\partial_{ss}[Q^{\frac 12}(Q^{-\frac 12}w)]
=\partial_{ss}(Q^{\frac 12})(Q^{-\frac 12}w)+2\partial_{s}(Q^{\frac 12})\partial_s(Q^{-\frac 12}w)
+Q^{\frac 12}\partial_{ss}(Q^{-\frac 12}w)\\
&=\partial_{ss}(Q^{\frac 12})(Q^{-\frac 12}w)+Q^{-\frac 12}\partial_s[Q\partial_s(Q^{-\frac 12}w)].
\end{align*}
Thus, setting
\[
G=f'(V_0)Q^{\frac 12}-(1-|\nabla \psi|^2) \partial_{ss}(Q^{\frac 12})
\]
(by the definition of $Q$ and $V_0$, we expect $G$ to be small in some sense),
we rewrite the equation of $w$ as follows
\begin{multline}\label{eq:w2}
(1-|\nabla \psi|^2)\partial_s[Q\partial_s(Q^{-\frac 12}w)]
= Q^{\frac 12}[2 \nabla \psi\cdot \nabla \partial_s w
+(\Delta \psi) \partial_s w +\Delta w]\\
+Q^{\frac 12}[ f_n(V_J+w)-f_n(V_J)-f_n'(V_0)w]+Gw +Q^{\frac 12}\Ens_J.
\end{multline}
The nonlinear term 
$f_n(V_J+w)-f_n(V_J)-f_n'(V_0)w$ is mostly quadratic in $w$ (some linear terms in $w$ remain but they are also small in $V_J-V_0$), which is an important gain with respect to the previous formulation.

 We define the following energy functional related to the above formulation of the equation of $w$
\begin{multline*}
\energy= \int \Bigl\{(1-|\nabla \psi|^2)[Q\partial_{s}(Q^{-\frac12}w)]^2+Q^2|\nabla (Q^{-\frac12}w)|^2+\frac{\lambda}{16} s^{-2}Qw^2\\
-Q\left[2F_n(V_J+w)-2F_n(V_J)-2F_n'(V_J)w-F_n''(V_0)w^2\right]\Bigr\}.
\end{multline*}
We also define a weighted norm related to the above functional
\[
\norm=\left(\int [Q\partial_{s}(Q^{-\frac12}w)]^2+Q^2|\nabla (Q^{-\frac12} w)|^2+\frac{\lambda}{16} s^{-2}Qw^2\right)^{\frac 12}.
\]
Since we may be dealing with $H^1\times L^2$ supercritical nonlinearities (but $H^2\times H^1$ subcritical by the condition $1\leq N\leq 4$),
we need higher order energy functionals. We set
\begin{align*}
\higher_0&=\int \left\{ (1-|\nabla \psi|^2) (\partial_{ss} w)^2 + |\nabla \partial_s w|^2\right\},\\
\higher_ \ell &=\int \left\{ (1-|\nabla \psi|^2) (\partial_{s}\partial_{x_ \ell } w)^2 + |\nabla \partial_{x_ \ell } w|^2\right\},
\quad 1\leq \ell \leq N,
\end{align*}
and
\[
\higher=\sum_{ \ell =0}^N \higher_ \ell ,\quad
\mathcal M=\left( \|w\|_{H^2}^2+\|\partial_s w\|_{H^1}^2+\|\partial_{ss} w\|_{L^2}^2 \right)^{\frac 12}.
\]

For future reference, we establish two estimates on $\partial_s Q$ and $\nabla Q$.
By the expression of $V_0$ in \eqref{eq:V}, we have 
\begin{equation} \label{OMN2} 
\begin{split} 
\partial_s Q&=(p+1)\partial_sV_0(1-\chi+V_0)^p=(p+1)(\partial_sV_0)Q^{\frac{p}{p+1}}\\
&=-\sqrt{2(p+1)}(1-|\nabla\psi|^2)^{-\frac12}V_0^{\frac{p+1}2}Q^{\frac{p}{p+1}}.
\end{split} 
\end{equation} 
Thus, since $Q^{\frac{p-1}{2(p+1)}} \lesssim s^{-1}$,
\begin{equation}\label{dsQ}
|\partial_s Q|\lesssim V_0^{\frac{p+1}2}Q^{\frac{p}{p+1}}\lesssim Q^{1+\frac{p-1}{2(p+1)}}\lesssim s^{-1}Q.
\end{equation}
Similarly, by \eqref{V1},
\begin{equation}\label{gradQ}
|\nabla Q|=(p+1)|\nabla V_0|(1-\chi+V_0)^p\lesssim V_0^{1+\frac{p-1}{2k}}Q^{\frac{p}{p+1}}\\
\lesssim Q^{1+\frac1k\frac{p-1}{2(p+1)}}\lesssim s^{-\frac1k}Q,
\end{equation}
and
\begin{equation}\label{graddsQ}
|\Delta Q|\lesssim s^{-\frac 2k} Q,\quad
|\nabla\partial_s Q|\lesssim s^{-1-\frac1k}Q.
\end{equation}

\medskip

\textbf{Step 1. Coercivity.}
We claim the following estimates.
\begin{lem}\label{le:WW}
It holds
\begin{equation}\label{WW5}
\hnorm^2\lesssim \norm^2+\higher.
\end{equation}
For $0<\delta\leq s_J$ and $0<\omega\leq 1$ sufficiently small, for $n$ large, if 
$\norm\leq \omega$ and $\hnorm\leq \omega$, then
\begin{equation}\label{coer}
2\energy+ \higher \geq  \norm^2.
\end{equation}
\end{lem}
\begin{proof}
First, we prove the following estimates.
For any $\rho\geq 0$, the following holds on $[S_n,\delta_0]$,
\begin{align}
\int Q^\rho|\nabla w|^2
&\lesssim\int Q^{\rho+1}|\nabla (Q^{-\frac12} w)|^2+\int Q^{\rho+\frac{p-1}{k(p+1)}} w^2,\label{WW1}\\
\int Q^{\rho+1}|\nabla (Q^{-\frac12} w)|^2
&\lesssim\int Q^\rho|\nabla w|^2+\int Q^{\rho+\frac{p-1}{k(p+1)}} w^2,\label{WW2}\\
\int Q^\rho|\partial_s w|^2
&\lesssim\int Q^{\rho+1}|\partial_s (Q^{-\frac12} w)|^2+ \int Q^{\rho+\frac{p-1}{p+1}} w^2,\label{WW3}\\
\int Q^{\rho+1}|\partial_s(Q^{-\frac12} w)|^2
&\lesssim\int Q^\rho|\partial_s w|^2+\int Q^{\rho+\frac{p-1}{p+1}} w^2.\label{WW4}
\end{align}
We have, using~\eqref{gradQ},
\begin{align*}
\int Q^\rho|\nabla w|^2
&=\int Q^\rho|Q^{\frac 12} \nabla(Q^{-\frac 12}w)+(\nabla Q^{\frac 12}) Q^{-\frac 12} w|^2\\
&\lesssim\int Q^{1+\rho}|\nabla(Q^{-\frac 12}w)|^2+\int Q^{\rho-1}|\nabla Q^\frac12|^2 w^2\\
&\lesssim\int Q^{1+\rho}|\nabla (Q^{-\frac12} w)|^2+\int Q^{\rho+\frac {p-1}{k(p+1)}}w^2.
\end{align*}
This proves \eqref{WW1} and the proof of \eqref{WW2} is similar.
Moreover, using~\eqref{dsQ}, 
\begin{align*}
\int Q^\rho|\partial_s w|^2
&=\int Q^\rho|Q^{\frac 12}\partial_s(Q^{-\frac 12}w)+(\partial_s Q^{\frac 12}) Q^{-\frac 12} w|^2\\
&\lesssim\int Q^{\rho+1}\partial_s(Q^{-\frac 12}w)|^2+\int Q^{\rho-1}|\partial_s Q^\frac12|^2 w^2\\
&\lesssim\int Q^{\rho+1}|\partial_s(Q^{-\frac12} w)|^2+\int Q^{\rho+\frac{p-1}{p+1}}w^2,
\end{align*}
which proves \eqref{WW3}; the proof of \eqref{WW4} is similar.

We prove \eqref{WW5}.
The inequality $\|w\|_{L^2}\lesssim \norm$ is obvious.
Next, \eqref{WW3} with $\rho=1$ and $Q^{\frac{p-1}{p+1}}\lesssim s^{-2}$ show that
$\int |\partial_s w|^2\lesssim \int Q |\partial_s w|^2 \lesssim \norm^2$.
Since $\|\nabla \psi\|_{L^\infty}\leq \frac12$ (from \eqref{c:psi}),
it follows (using $ \|w\| _{ H^2 } \lesssim  \|\Delta w\| _{ L^2 } + \| w\| _{ L^2 }$) that
$\hnorm^2 \lesssim \norm^2+\higher$, which is~\eqref{WW5}.

Last, we prove \eqref{coer}.
Let
\begin{equation}\label{def:A1}
A_1=\left|F_n(V_J+w) -F_n(V_J)-F_n'(V_J)w-\frac12F_n''(V_0)w^2\right|.
\end{equation}
The triangle inequality and the Taylor inequality \eqref{taylor0} yield 
\begin{equation*} 
\begin{split} 
A_1
\lesssim & \left|F_n(V_J+w) -F_n(V_J)-F_n'(V_J)w-\frac{F_n''(V_J)}2w^2\right|  +|F_n''(V_J)-F_n''(V_0)| w^2\\
\lesssim & \Lambda_1,
\end{split} 
\end{equation*} 
where
\begin{equation}\label{def:La1}
\Lambda_1=|w|^{p+1}+V_0^{p-\bar p}|w|^{\bar p+1}
+V_0^{p-2}|V_J-V_0|w^2.
\end{equation}
From \eqref{v5}, $V_J\lesssim V_0$ and 
$|V_J-V_0|\lesssim (1+V_0)^{-\frac{p-1}4}V_0\lesssim Q^{-\frac{p-1}{4(p+1)}}V_0$.
Moreover, $V_0^{p+1}\leq Q$. Thus,
\begin{equation}\label{sur:La1}
\Lambda_1\lesssim
|w|^{p+1}+Q^{\frac{p-\bar p}{p+1}}|w|^{\bar p+1}+ Q^{\frac34\frac{p-1}{p+1}}w^2,
\end{equation}
and so
\begin{equation*}
\int Q\Lambda_1
\lesssim \int Q|w|^{p+1}+\int Q^{1+\frac{p-\bar p}{p+1}}|w|^{\bar p+1}
 +\int Q^{1+\frac34\frac{p-1}{p+1}}w^2.
\end{equation*}
For the first term, we prove the following general estimate:
for any $0< \zeta\leq 1$,
\begin{equation} \label{ouf}
\int Q^{(1-\zeta) \frac{2p}{p+1}} |w|^{p+1}
\lesssim \norm^{p+1}+\hnorm^{p+1}.
\end{equation}
Indeed, using H\"older's inequality and the embedding $H^2(\R^N) \hookrightarrow L^q(\R^N)$ for $2\le q<\infty $ (recall that $N\le 4$), 
\begin{equation*} 
\begin{split} 
\int Q^{(1-\zeta) \frac{2p}{p+1}} |w|^{p+1}
& \lesssim s^{-2(1-\zeta)}  \int Q^{1-\zeta}  {  |w| }^{p+1}
\\&\lesssim s^{-2(1-\zeta)} \left(\int Q w^2\right)^{1-\zeta}\left(\int {  |w| }^{\frac{p-1}\zeta+2}\right)^{\zeta} \\
&\lesssim  \norm^{2(1-\zeta)} \hnorm^{p-1+2\zeta}
\lesssim \norm^{p+1}+\hnorm^{p+1}.
\end{split} 
\end{equation*} 
In particular,  from \eqref{ouf}, it holds 
\begin{equation*}
\int Q  |w|^{p+1}\lesssim  \norm^{p+1}+\hnorm^{p+1}.
\end{equation*}
In the case $1<p\leq 2$, one has $\bar p=p$ and the second term is identical to the first one.
In the case $p>\bar p=2$, the second term is estimated as follows.
Using the inequality $ |w|^3\le aw^2 + a^{-(p-2)}  |w|^{p+1}$ with $a= \varepsilon Q^{\frac {1} {p-1}}$, $\varepsilon >0$ to be chosen later, and the estimate $Q^{\frac{p-1}{ p+1 }}\lesssim s^{-2}$, we see that
\begin{align*}
Q^{\frac{2p-1}{p+1}}|w|^3
&\lesssim \varepsilon  Q^{\frac{p-1}{ p+1}}Qw^2+ \varepsilon ^{-(p-2) } Q |w|^{p+1}\\
&\lesssim \varepsilon  s^{-2}Q w^2+ \varepsilon ^{-(p-2) } Q |w|^{p+1},
\end{align*}
and so, using \eqref{ouf}
\begin{equation}\label{ouf2}
\int Q^{\frac{2p-1}{p+1}}|w|^{\bar p+1}
\lesssim  \varepsilon \norm^2+ \varepsilon ^{-(p-2) } \left(\norm^{p+1}+ \hnorm^{p+1}\right).
\end{equation}
Last, since $Q^{\frac{p-1}{2(p+1)}}\lesssim s^{-1}$, we observe that
\[
\int Q^{\frac34\frac{p-1}{p+1}+1}w^2\lesssim s^{-\frac 32} \int Q w^2\lesssim  s^{\frac 12} \norm^2.
\]
In conclusion, we have obtained, for $\mathcal N\leq \omega$, $\mathcal M\leq \omega$, $S_n\leq s\leq \delta$,
\begin{align*}
 \int QA_1 &\lesssim \left( \varepsilon  +s^{\frac 12}\right) \norm^2 +(1 + \varepsilon ^{-(p-2) } ) \left(\norm^{p+1}+ \hnorm^{p+1}\right)\\
 & \lesssim \left( \varepsilon  +\delta^{\frac 12}\right) \norm^2+ (1 + \varepsilon ^{-(p-2) } )  \omega ^{p-2} \hnorm^2,
\end{align*}
which, combined with \eqref{WW5}, implies that
for   $\delta>0$ and $\omega>0$ small enough, it holds
$2 \energy + \higher\geq  \norm^2$ on $[S_n,\delta_0]$.
(Recall that $1-  |\nabla \psi |^2 \ge \frac {3} {4}$ by~\eqref{c:psi} and $\lambda \le \frac {1} {2}$.)
\end{proof}

\textbf{Step 2. Energy control.}
We claim that  there exist $C>0$ such that 
\begin{equation}\label{energy}
\frac{d\energy}{ds}\leq 
C s^{-1+\lambda}\norm+
\frac\lambda 4 s^{-1}  \norm^2  +Cs^{-\frac12} \norm^2
+Cs^{-1} \left(\norm^{p+1}+\hnorm^{p+1}\right) 
\end{equation}
provided $\norm \le \omega $ and $\hnorm \le \omega $ with $\omega $ sufficiently small.

\emph{Proof of \eqref{energy}.}
We compute $\frac{d\energy}{ds}$:
\begin{align*}
\frac 12 \frac{d\energy}{ds}&=
\int \Bigl\{(1-|\nabla \psi|^2)Q\partial_{s}(Q^{-\frac12}w)\partial_s[Q\partial_{s}(Q^{-\frac12}w)]\\
&\qquad+Q^2\nabla(Q^{-\frac12} w)\cdot\partial_s[\nabla(Q^{-\frac12} w)]
+\frac\lambda{16} s^{-2}Q^{\frac32}w\partial_s(Q^{-\frac12}w)\\
&\qquad-Q^{\frac 32}\left[f_n(V_J+w)-f_n(V_J)-f_n'(V_0)w\right]\partial_s(Q^{-\frac12}w)\Bigr\}\\
&\quad +\int(\partial_sQ)Q|\nabla(Q^{-\frac12} w)|^2+\frac\lambda{32} s^{-2}(\partial_sQ)w^2-
{\frac \lambda{16} } s^{-3}Qw^2\\
&\quad -\frac12\int\partial_sQ\left[2F_n(V_J+w)-2F_n(V_J)-2F_n'(V_J)w-F_n''(V_0)w^2\right]\\
&\quad -\frac12\int\partial_sQ\left[f_n(V_J+w)-f_n(V_J)-f_n'(V_0)w\right]w\\
&\quad -\frac12\int Q\partial_sV_0\left[2f_n(V_J+w)-2f_n(V_J)-2f_n'(V_J)w-f_n''(V_0)w^2\right]\\
&\quad -\frac12\int Q\partial_s(V_J- V_0 )\left[2f_n(V_J+w)-2f_n(V_J)-2f_n'(V_J)w\right]\\
&=I_1+I_2+I_3+I_4+I_5+I_6.
\end{align*}

First, we remark the negative contribution of $I_2$. Since $\partial _sQ\le 0$ by~\eqref{OMN2}, we have
\begin{equation} \label{surI2}
I_2 \le -\frac\lambda{16} s^{-3}\int Qw^2.
\end{equation} 

Second, we compute $I_1$ using the equation \eqref{eq:w2} of $w$
\begin{align*}
I_1& = \int Q^{\frac 32} \partial_{s}(Q^{-\frac12}w)[2 \nabla \psi\cdot \nabla \partial_s w
+(\Delta \psi) \partial_s w ]\\
& \quad+\int\left\{Q^{\frac 32} \partial_{s}(Q^{-\frac12}w)(\Delta w)+Q^2\nabla(Q^{-\frac12} w)\cdot\nabla[\partial_s(Q^{-\frac12} w)]\right\} \\
& \quad+\int Q\partial_{s}(Q^{-\frac12}w) Gw +\int Q^{\frac 32}\partial_{s}(Q^{-\frac12}w)\Ens_J
+\frac\lambda{16} s^{-2}\int Q^{\frac32}w\partial_s(Q^{-\frac12}w)\\
&=I_7+I_8+I_9+I_{10}+I_{11}.
\end{align*}

For $I_7$, we first observe that
\begin{align*}
&2\int Q^{\frac 32}\partial_s(Q^{-\frac12}w)(\nabla\psi\cdot\nabla\partial_s w)
=2\int Q^2\partial_s(Q^{-\frac12}w)(\nabla\psi\cdot\nabla\partial_s (Q^{-\frac 12}w))\\
&\quad+ \int Q (\partial_s(Q^{-\frac12}w))^2 (\nabla\psi\cdot\nabla Q)
+ \int Q \partial_s Q \partial_s(Q^{-\frac12}w) (\nabla\psi\cdot\nabla(Q^{-\frac 12}w))\\
&\quad+ \int Q^\frac12w\partial_s(Q^{-\frac12}w)(\nabla\psi\cdot\nabla\partial_s Q)
-\frac12\int Q^{-\frac 12}\partial_s Qw\partial_s(Q^{-\frac12}w)(\nabla\psi\cdot\nabla Q).
\end{align*}
Second, by integration by parts,
\begin{align*}
&2\int Q^2\partial_s(Q^{-\frac12}w)(\nabla\psi\cdot\nabla\partial_s (Q^{-\frac 12}w))
= \int Q^2 \nabla \psi \cdot \nabla  [  ( \partial _s ( Q^{- \frac {1} {2} } w))^2 ] 
\\&
=-\int Q^2\Delta \psi [\partial_s(Q^{-\frac12}w)]^2
-2\int Q[\partial_s(Q^{-\frac12}w)]^2(\nabla\psi\cdot\nabla Q).
\end{align*}
By the definition of $\norm$, we estimate
\begin{equation*}
\left|\int Q^2\Delta \psi [\partial_s(Q^{-\frac12}w)]^2\right|\lesssim \norm^2.
\end{equation*}
Using \eqref{gradQ}, we also have
\begin{equation*}
\left|\int Q (\partial_s(Q^{-\frac12}w))^2 (\nabla\psi\cdot\nabla Q)\right|\lesssim s^{-\frac 1k} \norm^2.
\end{equation*}
Now, by the expressions of $Q$ and $V_0$, we have
\begin{align*}
|\partial_s Q|&=(p+1)|\partial_s V_0| (1-\chi+V_0)^p
=2\frac{p+1}{p-1}(s+A(x))^{-1}V_0(1-\chi+V_0)^p
\\&\leq 2\frac{p+1}{p-1}s^{-1}Q,
\end{align*}
and thus
\begin{align*}
&\left|\int Q \partial_s Q \partial_s(Q^{-\frac12}w) (\nabla\psi\cdot\nabla(Q^{-\frac 12}w))\right|\\&
\leq 2\frac{p+1}{p-1} s^{-1}\int |\nabla \psi|Q^2|\partial_s(Q^{-\frac12}w)||\nabla(Q^{-\frac 12}w)|
\leq \frac{p+1}{p-1} s^{-1}\| \nabla\psi\|_{L^\infty}\norm^2.
\end{align*}
Similarly, using \eqref{dsQ}, \eqref{gradQ}, \eqref{graddsQ}
\begin{align*}
\left|\int Q^\frac12w\partial_s(Q^{-\frac12}w)(\nabla\psi\cdot\nabla\partial_s Q)\right|&
\lesssim s^{-\frac1k}\left(\int Q^2|\partial_s(Q^{-\frac 12}w)|^2+s^{-2}\int Qw^2\right)
\\&\lesssim s^{-\frac1k} \norm^2,
\end{align*}
and
\[
\left|\int Q^{-\frac 12}\partial_s Qw\partial_s(Q^{-\frac12}w)(\nabla\psi\cdot\nabla Q)\right|
\lesssim s^{-\frac1k}\norm^2.
\]
Using the same estimates and then \eqref{WW3}, we finish estimating $I_7$ as follows
\begin{align*}
\left|\int Q^{\frac 32} \partial_{s}(Q^{-\frac12}w)(\Delta \psi) \partial_s w\right|&\lesssim
\int Q^2|\partial_s(Q^{-\frac 12}w)|^2 + \int Q (\partial_s w)^2\\
&\lesssim\int Q^2|\partial_s(Q^{-\frac 12}w)|^2+s^{-2}\int Qw^2\lesssim \norm^2.
\end{align*}
Thus, for some constant $C>0$, using \eqref{c:psi},
\begin{equation*}
|I_7|\leq \frac{p+1}{p-1} s^{-1} \|\nabla\psi\|_{L^\infty}\norm^2 +Cs^{-\frac1k} \norm^2
\leq \frac{\lambda}8 s^{-1}\norm^2 +Cs^{-\frac1k} \norm^2.
\end{equation*}

Next, integrating by parts,
using the identities 
\begin{gather*} 
Q^2\nabla [\partial _s (Q^{- \frac {1} {2}} w )] = Q^{\frac {1} {2}} \nabla [ Q^{\frac {3} {2}} (\partial _s (Q^{- \frac {1} {2}} w ) ] -\frac {3} {2} \partial _s (Q^{- \frac {1} {2}} w ) \nabla Q , \\
- \nabla w + Q^{\frac {1} {2}} \nabla  ( Q^{- \frac {1} {2}} w) = - Q^{- \frac {1} {2}} \nabla (Q^{\frac {1} {2}} ),
\end{gather*} 
and integrating again by parts, we find
\begin{align*}
I_8&=-\int \nabla (Q^{\frac 32} \partial_{s}(Q^{-\frac12}w))\cdot \nabla w+\int Q^2\nabla(Q^{-\frac12} w)\cdot\nabla[\partial_s(Q^{-\frac12} w)]
\\
&=-\int [\nabla (Q^{\frac 32} \partial_{s}(Q^{-\frac12}w))\cdot \nabla(Q^{\frac 12})]\, Q^{-\frac 12}w
-\frac 32\int Q\partial_{s}(Q^{-\frac12}w)[ \nabla Q\cdot \nabla (Q^{-\frac 12}w)]\\
&=-\int Q \partial_s(Q^{-\frac 12} w)[\nabla(Q^{-\frac12}w)\cdot\nabla Q]+\int\Delta(Q^{\frac 12})Q w \partial_s(Q^{-\frac 12}w).
\end{align*}
By \eqref{gradQ} and the definition of $\norm$,
\[
\left|\int Q \partial_s(Q^{-\frac 12} w)[\nabla(Q^{-\frac12}w)\cdot\nabla Q]\right|\lesssim s^{-\frac 1k} \norm^2.
\]
Similarly, using \eqref{gradQ} and \eqref{graddsQ}, we have $|\Delta(Q^{\frac 12})Q^{\frac 32}|\lesssim|\nabla Q|^2+|\Delta Q|Q\lesssim s^{-\frac2k}Q^2$, and thus
\[
\left|\int\Delta(Q^{\frac 12})Q w \partial_s(Q^{-\frac 12}w)\right|
\lesssim \int Q^2[\partial_s(Q^{-\frac 12}w)]^2 + s^{-\frac4k}\int Qw^2
\lesssim \norm^2.
\]

For $I_9$, we start by an estimate of $G=f_n'(V_0)Q^{\frac 12}-(1-|\nabla \psi|^2) \partial_{ss}(Q^{\frac 12})$.
By the definition of $Q=(1-\chi+V_0)^{p+1}$ and \eqref{eq:V}, we observe
\begin{align*}
\partial_{ss}(Q^{\frac 12})
&=\partial_{ss}[(1-\chi+V_0)^{\frac{p+1}2}]\\
&=\frac{p+1}2\frac{p-1}2(\partial_s V_0)^2(1-\chi+V_0)^{\frac{p-3}2}
+\frac{p+1}2\partial_{ss}V_0(1-\chi+V_0)^{\frac{p-1}2}
\\
&=(1-|\nabla \psi|^2)^{-1}\left[\frac{p-1}2 V_0^{p+1}Q^{\frac{p-3}{2(p+1)}}
+\frac{p+1}2V_0^{p}Q^{\frac{p-1}{2(p+1)}}\right].
\end{align*}
Thus,
\begin{equation*} 
\begin{split} 
G & = pV_0^{p-1}Q^{\frac 12} -\frac{p-1}2 V_0^{p+1} Q^{\frac{p-3}{2(p+1)}} -\frac{p+1}2 V_0^{p} Q^{\frac{p-1}{2(p+1)}}
\\&
= V_0^{p-1} Q^{ \frac {p-3} {2(p+1)} }  \Bigl[  p ( 1- \chi +V_0) + \frac {p-1} {2} V_0   \Bigr] (1 - \chi ).
\end{split} 
\end{equation*} 
For $|x|>1$, we have $V_0\lesssim 1$ and $Q\lesssim 1$; since also $Q \gtrsim 1$, we see that $\|G\|_{L^\infty}\lesssim 1$.
Therefore,
\[
|I_9|\lesssim \|G\|_{L^\infty}\norm^2\lesssim   \norm^2.
\]

For $I_{10}$, by the Cauchy-Schwarz inequality
\[
|I_{10}|=\left|\int Q^{\frac 32}\partial_{s}(Q^{-\frac12}w)\Ens_J\right|
\lesssim \| Q^{\frac 12}\Ens_J\|_{L^2} \norm,
\]
and we need only estimate $\| Q^{\frac 12}\Ens_J\|_{L^2}$. From \eqref{v3}, for $|x|\geq R$, we have
\[
Q^{\frac 12}|\Ens_J|\lesssim |\Ens_J|\lesssim |x|^{-\frac{2k}{p-1}-2}.
\]
Since $1\leq N\leq 4$, this implies $\|Q^{\frac 12}\Ens_J\|_{{L^2(|x| > R})}\lesssim 1$.

Next, from \eqref{v1}, for $|x|\leq R$, we have
\[
Q^{\frac 12}|\Ens_J|\lesssim Q^{\frac 12}V_0^{\frac{p+1}2+(-J+\frac{1+J}{k})\frac{p-1}2}
\lesssim V_0^{p+1+(-J+\frac{1+J}{k})\frac{p-1}2}
\lesssim V_0^{p+1+\frac1k \frac{p-1}2-J(1-\frac 1k)\frac{p-1}2}.
\]
Recall that by~\eqref{on:lam}, 
\[
- J \frac {p-1} {2} \le - \frac {p+3} {2} - \lambda (p-1)
\]
and that~\eqref{on:kk} is equivalent to
\[
\frac{p+1}k-\lambda(p-1)\left(1-\frac 1k\right) \le  -\frac{\lambda(p-1)}2.
\]
Thus, for $|x|\leq R$,
\begin{align*}
Q^{\frac 12}|\Ens_J|
& \lesssim V_0^{p+1+\frac1k \frac{p-1}2-\frac{p+3}{2}(1-\frac 1k) -\lambda(p-1)(1-\frac 1k)}
\lesssim V_0^{\frac{p-1}2 + \frac {p+1}k-\lambda(p-1)(1-\frac 1k)} \\ &
\lesssim V_0^{\frac{p-1}2  - \lambda\frac{p-1}2}
\lesssim (s+A(x))^{-1+\lambda}\lesssim s^{-1+\lambda}.
\end{align*}
It follows that
\begin{equation}\label{e:EJ}
\|Q^{\frac 12}\Ens_J\|_{L^2}\lesssim s^{-1+\lambda}.
\end{equation}
For this term, we have obtained
\[
|I_{10}|\lesssim\|Q^{\frac 12}\Ens_J\|_{L^2}\norm\lesssim s^{-1+\lambda}\norm.
\]

Finally, by the Cauchy-Schwarz inequality,
\[
|I_{11}|\leq s^{-1}\frac\lambda{16} \int Q^2 |\partial_s(Q^{-\frac 12}w)|^2
+s^{-3}\frac\lambda{16} \int Qw^2.
\]
Using \eqref{surI2}, we obtain
\[
|I_{11}|\leq \frac{\lambda}{16} s^{-1} \norm^2  - I_2  .
\]

In conclusion for $I_1+I_2$, we find
\[
I_1+I_2\leq C s^{-1+\lambda}\norm+\frac{3\lambda}{16}   s^{-1}  \norm^2 +Cs^{-\frac1k}\norm^2
.
\]

To continue with the proof of \eqref{energy}, we estimate the term $I_3$.
To that end, recall that $\bar p=\min(2,p)$. 
First, by~\eqref{def:A1}--\eqref{sur:La1} and~\eqref{dsQ}
\begin{align*}
\left|\partial_sQ \right| A_1 
&\lesssim \left|\partial_sQ\right|\Lambda_1\lesssim V_0^{\frac{p+1}2}Q^{\frac{p}{p+1}}\Lambda_1
\lesssim  Q^{\frac{3p+1}{2(p+1)}} \Lambda_1\\
&\lesssim Q^{\frac{3p+1}{2(p+1)}} |w|^{p+1}+Q^{\frac{3p+1}{2(p+1)}}Q^{\frac{p-\bar p}{p+1}}|w|^{\bar p+1} 
+Q^{\frac{9p-1}{4(p+1)}}w^2.
\end{align*}
Using \eqref{ouf}, the first term is controled as follows
\[
\int Q^{\frac{3p+1}{2(p+1)}}|w|^{p+1}
\lesssim \norm^{p+1}+\hnorm^{p+1}.
\]
In the case $1<p\leq 2$, one has $\bar p=p$ and the second term is identical to the first one.
In the case $p>\bar p=2$, using $Q^{\frac{p-1}{2(p+1) }}\lesssim s^{-1}$ and \eqref{ouf2},
\begin{align*}
\int Q^{\frac{3p+1}{2(p+1)}}Q^{\frac{p-\bar p}{p+1}}|w|^{\bar p+1}
&= \int Q^{\frac{5p-3}{2(p+1)}}|w|^{\bar p+1}\lesssim
s^{-1}\int Q^{\frac{2p-1}{p+1}}|w|^{\bar p+1}\\
&\lesssim   \varepsilon s^{-1} \norm^2+ \varepsilon ^{-(p-2) }  s^{-1}  \left(\norm^{p+1}+ \hnorm^{p+1}\right),
\end{align*}
where $\varepsilon >0$ is to be chosen.
Last, we observe that
$ Q^{\frac{9p-1}{4(p+1)}} w^2\lesssim s^{-\frac 52} Qw^2$,
and thus
\[
\int Q^{\frac{9p-1}{4(p+1)}} w^2 \lesssim s^{-\frac 12}\norm^2.
\]
In conclusion, we have proved 
\begin{align}
|I_3|&\lesssim \int |\partial_s Q| \Lambda_1 \lesssim \int  Q^{\frac{3p+1}{2(p+1)}} \Lambda_1 \nonumber \\
&\lesssim ( s^{-\frac 12}  + \varepsilon s^{-1} ) \norm^2+ (1+ \varepsilon ^{-(p-2)}s^{-1} ) \left(\norm^{p+1}+ \hnorm^{p+1}\right).\label{suite}
\end{align}

We proceed similarly for $I_4$. Indeed, setting
\begin{align*}
A_2&=\left|f_n(V_J+w)-f_n(V_J)-f_n'(V_J)w\right||w|\\
&\lesssim \left|f_n(V_J+w)-f_n(V_J)-f_n'(V_J)w\right||w| +|f_n'(V_0)-f_n'(V_J)| w^2,
\end{align*}
by \eqref{taylor1} and Taylor's inequality, 
\[
A_2\lesssim |w|^{p+1}+V_0^{p-\bar p}|w|^{\bar p+1}+V_0^{p-2}|V_J-V_0|w^2= \Lambda_1.
\]
Using  \eqref{suite}, we conclude that
\begin{equation*} 
|I_4| \lesssim \int |\partial_s Q| \Lambda_1 
\lesssim ( s^{-\frac 12}  + \varepsilon s^{-1} ) \norm^2+ (1+ \varepsilon ^{-(p-2)}s^{-1} ) \left(\norm^{p+1}+ \hnorm^{p+1}\right).
\end{equation*}

Now, we estimate $I_5$ and we set
\[
A_3= | 2f_n(V_J+w)-2f_n(V_J)-2f_n'(V_J)w-f_n''(V_0)w^2 | .
\]
By the triangle inequality, Taylor inequality \eqref{taylor}, 
$|\partial_s V_0|\lesssim V_0^{\frac{p+1}2}$ (see \eqref{eq:V}),
we have
\begin{align*}
Q \left| \partial_sV_0 \right| A_3
&\lesssim QV_0^{\frac{p+1}2}\left| 2f_n(V_J+w)-2f_n(V_J)-2f_n'(V_J)w-f_n''(V_J)w^2\right| \\&
\quad +QV_0^{\frac{p+1}2}\left| f_n''(V_J)-f_n''(V_0)\right| w^2\\
&\lesssim QV_0^{\frac{p-1}2}\Lambda_1 \lesssim  Q^{\frac{3p+1}{2(p+1)}} \Lambda_1. 
\end{align*}
Using  \eqref{suite}, we conclude that
\begin{equation*} 
|I_5| \lesssim ( s^{-\frac 12}  + \varepsilon s^{-1} ) \norm^2+ (1+ \varepsilon ^{-(p-2)}s^{-1} ) \left(\norm^{p+1}+ \hnorm^{p+1}\right).
\end{equation*}

Finally, we estimate $I_6$ and we set
\[
A_4=\left|2f_n(V_J+w)-2f_n(V_J)-2f_n'(V_J)w\right|.
\]
By the triangle inequality and Taylor's inequality~\eqref{taylor} 
\begin{equation*} 
\begin{split} 
A_4 &\lesssim \left|2f_n(V_J+w)-2f_n(V_J)-2f_n'(V_J)w- f_n '' (V_J) w^2\right| +   |f_n '' (V_J)| w^2 \\
&\lesssim V_0^{-1}  |w|^{p+1} + V_0^{p - \bar p -1}  |w|^{ \bar p+1} + V_0^{p-2} w^2 .
\end{split} 
\end{equation*} 
Using~\eqref{v5bis}, $V_0\lesssim Q^{\frac {1} {p+1}}$, $Q \gtrsim 1$, $Q \lesssim s^{- \frac {2( p+1) } {p-1}}$ and $k\ge 1$, we obtain
\begin{equation*} 
\begin{split} 
Q|\partial_s(V_J- V_0 ) | A_4
&\lesssim QV_0^{ \frac{p-1}{2k}}\left(   |w|^{p+1} + V_0^{p - \bar p }  |w|^{ \bar p+1} + V_0^{p-1} w^2 \right)\\
&\lesssim Q^{1+\frac{p-1}{2k(p+1)}}   \left(   |w|^{p+1} + V_0^{p - \bar p }  |w|^{ \bar p+1} \right) + Q^{1+\frac{p-1}{ p+1 } (1+ \frac {1} {2k})} w^2  \\
&\lesssim Q^{ \frac{ 3p +1 }{2(p+1)}}   \left(   |w|^{p+1} + V_0^{p - \bar p }  |w|^{ \bar p+1} \right) + 
s^{- \frac {2k+1} {k}} Q  w^2 \\
&\lesssim Q^{ \frac{ 3p +1 }{2(p+1)}}\Lambda _1 + s^{- \frac {1} {k}}  s^{-2}Q  w^2.
\end{split} 
\end{equation*} 
Using  \eqref{suite} and $k\ge 2$, we conclude that
\begin{equation*} 
|I_6| \lesssim ( s^{-\frac 12}  + \varepsilon s^{-1} ) \norm^2+ (1+ \varepsilon ^{-(p-2)}s^{-1} ) \left(\norm^{p+1}+ \hnorm^{p+1}\right).
\end{equation*}
Choosing $\varepsilon \le \frac {\lambda } {16}$, then $\omega$ sufficiently small,
and collecting the above estimates, we have proved \eqref{energy}.

\medskip

\textbf{Step 3. Higher order energy terms.}
We claim that for any $ \ell =0,1,\ldots,N$,
\begin{equation}\label{e:dK}
\left| \frac {d\higher _\ell }{ds}\right|\lesssim
s^{-1+\lambda} \left( \norm +  \norm^2+   \hnorm^2\right)+\norm^{p+1}+\hnorm^{p+1}.
\end{equation}

Differentiating \eqref{eq:w} with respect to $s$, 
setting $z_0=\partial_s w$, we have
\begin{equation}\label{eq:z0}\begin{aligned}
&(1-|\nabla \psi|^2) \partial_{ss}z_0 - 2 \nabla \psi\cdot \nabla \partial_s z_0
-(\Delta \psi) \partial_s z_0 - \Delta z_0 \\&\quad
=f_n'(V_J+w)z_0 + (f_n'(V_J+w)-f_n'(V_J))\partial_s V_J + \partial_s\Ens_J.
\end{aligned}\end{equation}
Differentiating $\higher_0=\int (1-|\nabla \psi|^2) (\partial_s z_0)^2+|\nabla z_0|^2$,
 we find from \eqref{eq:z0} and integration by parts
\begin{align*}
\frac 12 \frac{d \higher_0}{ds}
&=\int \left\{f_n'(V_J+w)z_0 + (f_n'(V_J+w)-f_n'(V_J))\partial_s V_J + \partial_s\Ens_J\right\}\partial_s z_0\\
&=I_{12}+I_{13}+I_{14}.
\end{align*}

First, by the Cauchy-Schwarz inequality
\[
|I_{12}|\lesssim \|f_n'(V_J+w)z_0\|_{L^2}\higher_0^{\frac 12}
\lesssim \|f_n'(V_J+w)z_0\|_{L^2}\hnorm.
\]
From
\[|f_n'(V_J+w)|\lesssim |V_0|^{p-1}+|w|^{p-1}\lesssim s^{-\frac{p-1}p} Q^{\frac {p-1}{2p}}+|w|^{p-1}
\lesssim s^{-\frac{p-1}p} Q^{\frac 12}+|w|^{p-1}
\]
and then \eqref{WW3} with $\rho=1$, we have (recall that $1\leq N\leq 4$ and thus $H^2(\R^N)\hookrightarrow L^q(\R^N)$  for all $q\geq 2$, and $H^1(\R^N)\hookrightarrow L^q(\R^N)$ for all $2\le q\le 4$)
\begin{align*}
\|f'(V_J+w)z_0\|_{L^2}^2
&\lesssim s^{-\frac{2(p-1)}p} \int Q (\partial_s w)^2 + \int |w|^{2(p-1)} (\partial_s w)^2\\
&\lesssim s^{-\frac{2(p-1)}p} \norm^2+ \|w\| _{L^{ 4p }}^{2(p-1)}\|\partial_s w\|_{L^{ \frac {4p} {p+1} }}^2\\
&\lesssim s^{-\frac{2(p-1)}p} \norm^2+\|w\|_{H^2}^{2(p-1)}\|\partial_s w\|_{H^1}^2 .
\end{align*}
Thus, using also $-\frac{p-1}p\geq -1+\lambda$ (since $\lambda\leq \frac 1p$)
we have
\[
|I_{12}|\lesssim s^{-\frac{p-1}p} \norm\hnorm +\hnorm^{p+1}
\lesssim s^{-1+\lambda} \left(\norm^2+\hnorm^2\right)+\hnorm^{p+1}.
\]

Second, using \eqref{V1} and \eqref{v5bis},
$ |\partial _s V_J| \lesssim V_0^{1 + \frac {p-1} {2}} + V_0^{1 + \frac {p-1} {2k}} \lesssim V_0 Q^{\frac {p-1} {2 (p+1) }}$, so that  Taylor's inequality \eqref{taylor10} yields
\begin{align*}
|(f_n'(V_J+w)-f_n'(V_J))\partial_s V_J|
&\lesssim Q^{\frac {p-1} {2 (p+1) }} |w|^{p} 
+Q^{\frac {p-1} {2 (p+1) }} V_0 ^{p-1} |w|  \\
&\lesssim Q^{\frac {p-1} {2 (p+1) }} |w|^{p} 
+Q^{\frac {3 (p-1) } {2 (p+1) }}  |w| .
\end{align*}
We have
\begin{align*}
\int Q^{\frac {p-1} {p+1 }} |w|^{2p} 
&\lesssim \int (Q w^2) ^{\frac{p-1}{p+1}} |w|^{2\frac{p^2+1}{p+1}}
\lesssim \left(\int Q w^2\right)^{\frac{p-1}{p+1}}\left(\int w^{p^2+1}\right)^{\frac2{p+1}} \\ &
\lesssim \norm^{2\frac {p-1}{p+1}}\hnorm^{2\frac{p^2+1}{p+1}}
\lesssim ( \norm + \hnorm )^{2p} .
\end{align*}
Moreover, since $\frac {2(p-2)} {p+1} \le \frac {(2p-1) (p-1)} {p( p+1) } $, $Q \gtrsim 1$, $Q \lesssim s^{- \frac {2(p+1)} {p-1}}$, and $\lambda \le \frac {1} {p}$, we have
\begin{equation*} 
\begin{split} 
\int Q^{\frac {3(p-1) } {p+1 }} w ^2 & \lesssim   \int Q^{\frac {2(p-2)} {p+1}} Q w^2
 \lesssim   \int Q^{\frac {(2p-1) (p-1)} {p( p+1) }} Q w^2 \\ 
 & \lesssim s^{- \frac {2 (2p-1) } {p}}  \int Q w^2
 \lesssim s^{- \frac {2 (p-1) } {p}} \norm ^2  \lesssim s^{- 2(1-\lambda ) } \norm ^2 .
\end{split} 
\end{equation*} 
Thus,
\begin{align*}
|I_{13}|&\leq \left(\| Q^{\frac {p-1} {2 (p+1) }} |w|^{p}  \|_{L^2}+\| Q^{\frac {3 (p-1) } {2 (p+1) }} \|_{L^2}\right)\hnorm\\
&\lesssim ( s^{-1 + \lambda } \norm  +( \norm + \hnorm )^p ) \hnorm
\lesssim  s^{-1+\lambda} \left(\norm^2+\hnorm^2\right)+ \norm^{p+1}+\hnorm^{p+1}.
\end{align*}

Third, from \eqref{v3}, for $|x|\geq R$, $|\partial_s \Ens_J|\lesssim |x|^{-\frac{k(p+1)}{p-1}-2}$ 
and thus, since $1\leq N\leq 4$, $\|\partial_s \Ens_J\|_{L^2(|x|\geq R)} \lesssim 1$.
Now, from \eqref{v1}, for $|x|\leq R$, 
\[|\partial_s \Ens_J|\lesssim V_0^{\frac{p+1}2+(1-J+\frac {1+J} k)\frac{p-1}2}
\lesssim V_0^{p+\frac 1k \frac{p-1}2-J(1-\frac 1k)\frac{p-1}2},
\]
and thus, following the proof of \eqref{e:EJ}, we have
$\|\partial_s \Ens_J\|_{L^2(|x|\leq R)} \lesssim s^{-1+\lambda}$. 
Thus,
\[
|I_{14}|\lesssim \|\partial_s\Ens_J\|_{L^2} \hnorm\lesssim s^{-1+\lambda} \hnorm.
\]
The above estimates prove \eqref{e:dK} for $\higher_0$. 

We now prove~\eqref{e:dK} for $ \ell \in \{ 1, \dots, N\}$. 
Differentiating \eqref{eq:w} with respect to $x_ \ell $, 
setting $z_ \ell =\partial_{x_ \ell } w$, we have
\begin{equation}\label{eq:zl}\begin{aligned}
&(1-|\nabla \psi|^2) \partial_{ss}z_ \ell  - 2 \nabla \psi\cdot \nabla \partial_s z_ \ell 
-(\Delta \psi) \partial_s z_ \ell  - \Delta z_ \ell  \\&\quad
=f_n'(V_J+w)z_ \ell  + (f_n'(V_J+w)-f_n'(V_J))\partial_{x_ \ell } V_J + \partial_{x_ \ell }\Ens_J
\\ & \quad + 2 ( \nabla \psi \cdot \nabla \partial  _{ x_ \ell  } \psi  )\partial  _{ ss } w + 2 \nabla \partial  _{ x_ \ell  } \psi \cdot \nabla \partial _s w + (\Delta \partial  _{ x_ \ell  } \psi ) \partial _s w.
\end{aligned}
\end{equation}
Differentiating $\higher_ \ell =\int (1-|\nabla \psi|^2) (\partial_s z_ \ell )^2+|\nabla z_ \ell |^2$,
 we find from \eqref{eq:zl} and integration by parts
\begin{align*}
\frac 12 \frac{d \higher_ \ell }{ds}
&=\int \left\{f_n'(V_J+w)z_ \ell  + (f_n'(V_J+w)-f_n'(V_J))\partial_{x_ \ell } V_J + \partial_{x_ \ell }\Ens_J\right\}\partial_s z_ \ell \\
& \quad + 
\int  \left\{ 2 ( \nabla \psi \cdot \nabla \partial  _{ x_ \ell  } \psi  )\partial  _{ ss } w + 2 \nabla \partial  _{ x_ \ell  } \psi \cdot \nabla \partial _s w + (\Delta \partial  _{ x_ \ell  } \psi ) \partial _s w \right\} \partial_s z_ \ell 
\\
&=I_{15}+I_{16}+I_{17} + I _{ 18 } . 
\end{align*}
The term $I _{ 15 }$ is estimated exactly like $I _{ 12 }$. Next, it follows from~\eqref{V1}, \eqref{v4}, \eqref{v2} and the properties of $\chi $ that $ | \partial  _{ x_\ell } V_J | \lesssim V_0^{1 + \frac {p-1} {2k}}$, so that $I _{ 16 }$ is estimated like $I _{ 13 }$. 

Moreover, from \eqref{v3}, for $|x|\geq R$, $|\partial _{ x_\ell  } \Ens_J|\lesssim |x|^{-3}$ 
and thus, since $1\leq N\leq 4$, $\|\partial_s \Ens_J\|_{L^2(|x|\geq R)} \lesssim 1$.
Now, from \eqref{v1}, for $|x|\leq R$, 
\[|\partial_{ x_\ell  } \Ens_J|\lesssim V_0^{\frac{p+1}2+(-J+\frac {2+J} k)\frac{p-1}2}
\lesssim V_0^{ \frac {p-1} {2} +\frac 1k \frac{p-1}2-J(1-\frac 1k)\frac{p-1}2},
\]
and thus, following the proof of \eqref{e:EJ}, we have
$\|\partial_s \Ens_J\|_{L^2(|x|\leq R)} \lesssim s^{-1+\lambda}$. 
Thus we see that $I _{ 17 }$ is estimated like $I _{ 14 }$. 

Finally, 
\begin{equation*} 
 \left| 2 ( \nabla \psi \cdot \nabla \partial  _{ x_ \ell  } \psi  )\partial  _{ ss } w + 2 \nabla \partial  _{ x_ \ell  } \psi \cdot \nabla \partial _s w + (\Delta \partial  _{ x_ \ell  } \psi ) \partial _s w \right| \lesssim 
  | \partial  _{ ss } w | +  | \nabla \partial _s w | +  | \partial _s w | ,
\end{equation*} 
so that $I _{ 18 } \lesssim \hnorm ^2$. Therefore,  the estimate~\eqref{e:dK} holds for $ \ell \in \{ 1, \dots, N\}$.

\medskip

\textbf{Step 4. Conclusion.}
Since $\energy(S_n)=\higher(S_n)=0$, the following is well-defined
\[
S_n^\star=\sup\{ s\in [S_n,\delta] : \mbox{for all $s'\in [S_n,s]$, $\norm^2+\higher\leq \min(s^{ \lambda} ; \omega)$}\},
\]
and by continuity, $S_n^\star\in (S_n,\delta]$.
It follows from~\eqref{energy}, \eqref{e:dK}, \eqref{WW5}, and $\lambda \le \frac {1} {2}$ that
\begin{equation*}
\frac{d}{dt}\left(\energy+  \higher\right)
\leq  s^{-1+\lambda} \left( C \norm + \frac {\lambda } {4} s^{-\lambda } \norm^2 + C (\norm ^2 + \higher) + C s^{-\lambda } (\norm ^2 + \higher) ^{\frac {p+1} {2}} \right),
\end{equation*}
for some constant $C>0$ independent of $\delta $. By the definition of $S_n^\star$, we deduce that
\begin{equation*}
\frac{d}{dt}\left(\energy+  \higher\right)
\leq \lambda s^{-1+\lambda} \left(C s^{\frac 1 2\lambda}+C s^{\frac {p-1}2\lambda} +\frac14\right),
\end{equation*}
for some constant $C>0$ independent of $\delta $.
We fix   $0<\delta_0\leq \delta$  such that
\[
 C \delta_0^{\frac 1 2\lambda}+C \delta_0^{\frac {p-1}2\lambda} +\frac14\leq \frac 13,
\quad \delta_0^\lambda \leq  \omega.
\]
This gives, for all $S_n\leq s\leq \min(S_n^\star,\delta_0)$, $\frac{d}{dt} (\energy+  \higher )\leq \frac \lambda3 s^{-1+\lambda}$.

By integration, using $\energy(s_n)=\higher(s_n)=0$,  we find  for $S_n\leq s\leq \min(S_n^\star,\delta_0)$,
\begin{equation*}
\energy(s) + \higher (s) \leq \frac 13(s^{\lambda}-S_n^\lambda)\leq \frac 13 (s-S_n)^\lambda.\end{equation*}
Thus, from \eqref{coer}, it holds, for $S_n\leq s\leq \min(S_n^\star,\delta_0)$,
\[
\norm^2(s)+ \higher(s)\leq \frac 23 (s-S_n)^\lambda.
\]
It follows from \eqref{WW5} and the definition of $S_n^\star$ that $S_n^\star\geq \delta_0$
and so, for all $s\in [S_n,\delta_0]$,
\[
\hnorm(s)\lesssim  (s-S_n)^\frac\lambda2.
\]
This completes the proof of the proposition.
\end{proof}

\begin{proof} [Proof of Proposition~$\ref{eExisV}$]
We set
\[
Z_n(s)=V_J(S_n+s),\quad
\eta_n(s,y)=w_n(S_n+s),\quad \mathcal F_n(s)=\Ens_J(S_n+s).
\]
From   Proposition~\ref{pr:unif}, there  exist~$C>0$, $n_0>0$ and $0<\delta_0<1$ such that
\begin{equation}\label{e:unif2}
 \|\eta_n(s)\|_{H^2}+\|\partial_s \eta_n(s)\|_{H^1}+\|\partial_{ss} \eta_n(s)\|_{L^2}\leq C s^{\frac \lambda2}
\end{equation}
for all $n\geq n_0$ and  $s\in [0,\delta_0]$.
Moreover, from \eqref{eq:w}, 
\begin{multline}\label{eq:etan} 
 (1-|\nabla \psi|^2) \partial_{ss} \eta_n - 2 \nabla \psi\cdot \nabla \partial_s \eta_n
-(\Delta \psi) \partial_s \eta_n - \Delta \eta_n \\
  = f_n(Z_n+\eta_n)-f_n(Z_n)+\mathcal F_n.
\end{multline}
It follows from estimate~\eqref{e:unif2} that there exist a subsequence of $(\eta_n)$ (still denoted by $(\eta_n)$) and 
a map $\eta\in L^\infty((0,\delta_0),H^2(\R^N))\cap W^{1,\infty}((0,\delta_0),H^1(\R^N))\cap W^{2,\infty}((0,\delta_0),L^2(\R^N))$
such that
\begin{align}
&\eta_n\mathop{\longrightarrow}_{ n\to\infty} \eta \quad \mbox{in $L^\infty((0,\delta_0),H^2(\R^N))$ weak*}\\
&\partial_s\eta_n\mathop{\longrightarrow}_{n\to\infty} \partial_s\eta \quad \mbox{in $L^\infty((0,\delta_0),H^1(\R^N))$ weak*}\\
&\partial_{ss}\eta_n\mathop{\longrightarrow}_{n\to\infty} \partial_{ss}\eta \quad \mbox{in $L^\infty((0,\delta_0),L^2(\R^N))$ weak*}\\
&\eta_n(s)\mathop{\longrightarrow}_{ n\to\infty} \eta(s) \quad \mbox{weakly in $H^2(\R^N)$, for all $s\in [0,\delta_0]$} \label{fEstetan} \\
&\partial_s\eta_n(s)\mathop{\longrightarrow}_{n\to\infty} \partial_s\eta(s) \quad \mbox{weakly in $H^1(\R^N))$, for all $s\in [0,\delta_0]$.  \label{fEstdsetan}}
\end{align}
It is then easy to pass to the limit in \eqref{eq:etan}, and it follows that
\begin{equation*} 
 (1-|\nabla \psi|^2) \partial_{ss} \eta - 2 \nabla \psi\cdot \nabla \partial_s \eta
-(\Delta \psi) \partial_s \eta - \Delta \eta \\
  = f(V_J+\eta)-f_n(V_J)+\Ens_J
\end{equation*}
in $L^\infty((0,\delta_0),L^2(\R^N))$. Therefore, setting
\[
v(s)=V_J(s)+\eta(s),\quad s\in (0,\delta_0),
\]
it holds 
\begin{equation} \label{fRegv0} 
v\in L^\infty_\Loc ((0,\delta_0),H^2(\R^N))\cap W^{1,\infty}_\Loc ((0,\delta_0),H^1(\R^N))\cap W^{2,\infty}_\Loc ((0,\delta_0),L^2(\R^N))
\end{equation} 
and, using the definition of $\Ens_J$, we see that $v$ is a solution of equation~\eqref{eq:v} in $L^\infty_\Loc ((0,\delta_0),L^2(\R^N))$.
The estimate~\eqref{festV1} follows by letting $n\to \infty $ in~\eqref{e:unif} and using~\eqref{fEstetan} and~\eqref{fEstdsetan}. 
We now prove that $v$ satisfies~\eqref{fRegv}. By standard semigroup theory (see Section~\ref{sAppAreg}) it suffices to prove that $ |v| ^{p-1} v \in C ( (0, \delta _0), H^1 (\R^N ) )$.
Since by~\eqref{fRegv0} $v\in C((0,\delta _0), H^{2-\eta } (\R^N )) $ for every $\eta >0$, and $N\le 4$, 
we have by Sobolev's embeddings $v\in C((0,\delta _0), W^{1, q} (\R^N ) )$ for all $2\le q< 4$ and $ |v|^{p-1} \in C((0,\delta _0), L^r (\R^N ) )$ for $\max\{ 1, \frac {2} {p-1} \} \le r< \infty $. 
Choosing for instance $q= \frac {4(p+1)} {p+3}$ and $r= \frac {4(p+1) } {p-1}$ yields $ |v| ^{p-1} v \in C ( (0, \delta _0), H^1 (\R^N ) )$.

Finally, we prove~\eqref{fNFestinf}. 
We write 
\begin{equation*} 
 | \partial _s v | \ge  | \partial _s V_0  | -  | \partial _s ( V_J-   V_0 ) | -  | \partial _s ( v - V_J) | . 
\end{equation*} 
On the other hand,  $ V_0^{1 + \frac {p-1} {2}}  \lesssim  |\partial _s V_0 |$, so that $ V_0^{1 + \frac {p-1} {2k}}  \lesssim  |\partial _s V_0 |^{1 - \frac {(p-1) (k-1) } {k (p+1) }}$. Therefore, given any $\eta >0$, there exists a constant $C_\eta $ such that 
\begin{equation} \label{fUDP1} 
 V_0^{1 + \frac {p-1} {2k}} \le \eta  |\partial _s V_0| + C_\eta.
\end{equation}  
Since $|\partial_s ( V_j- V_0 ) |\lesssim V_0^{1+\frac{p-1}{2k}}$ by~\eqref{v5bis}, we see that there exists a constant $C$ such that
\begin{equation} \label{fUDP2} 
 | \partial _s v |^2  \ge \frac {1} {2}  | \partial _s V_0  |^2 - C | \partial _s ( v -  V_J ) |^2 -C . 
\end{equation} 
Next, we write
\begin{equation*} 
 |\nabla v | \le  |\nabla (v- V_J) | +  |\nabla V_J | \le  |\nabla (v- V_J) | +  | \nabla V_0 | + \sum_{ j=1 }^J  |\nabla v_j|.  
\end{equation*} 
It follows from~\eqref{v2} that $ |\nabla V_J |\lesssim 1$ for $ |x|\ge R$.
For $ |x|< R$, by~\eqref{v4} and $k\ge 2$,  $ |\nabla v_j| \lesssim V_0 $ for $j\ge 1$; and $ | \nabla V_0 | \lesssim V_0^{1+ \frac {p-1} {2k}}$ by~\eqref{V1}. Using again~\eqref{fUDP1}, we conclude that
\begin{equation} \label{fUDP3} 
 | \nabla  v |^2  \le \frac {1} {4}  | \partial _s V_0  |^2 - C | \nabla ( v -  V_J ) |^2 -C . 
\end{equation} 
Since $ | \partial _s ( v -  V_J ) | +  | \nabla ( v -  V_J ) | \in L^\infty  ((0, \delta _0) , H^1 (\R^N )  )$
by~\eqref{festV1}, the lower estimate~\eqref{fNFestinf} follows from~\eqref{fUDP2} and~\eqref{fUDP3}.
\end{proof} 

\section{Proof of Theorem~\ref{TH:2}}

In this section, we use the following notation. 
We let $\{\mathbf{e}_k;  k=1,\ldots,N\}$ be the canonical basis of $\R^N$.
If $N\ge 2$, then for $x\in \R^N$, we denote $x=(x_1,x_2,\ldots,x_N)$ and $\bar x=(x_2,\ldots,x_N)$.
We set $\bar\Delta u=\sum_{k=2}^N \partial_{x_kx_k} u$.
If $N=1$, we ignore $\bar x$ and $\bar \Delta $.

\subsection{Cut-off of the local hypersurface} \label{sAdjustment} 

Let $\varphi$ be a function satisfying \eqref{th2:2} (see statement of Theorem~\ref{TH:2}).
Without loss of generality, by the invariance by rotation of equation \eqref{wave}, we assume that
\begin{equation*}
\nabla \varphi(0)=\ell \mathbf{e}_1 \quad \mbox{where $0\leq \ell <1$.}
\end{equation*}
(For dimension $1$, the reduction is done by possibly changing $x\mapsto -x$.)
For a positive real $r<1$ small to be defined later, set
\[
 \widetilde{\varphi} (x) = (\varphi(x)-\ell x_1) \chi\left(\frac{|x|}r\right)  +\ell x_1 .
\]
On the one hand, from this definition and the properties of $\chi$, it holds
\begin{equation}\label{on:varphit}
 \widetilde{\varphi} (x)=\varphi(x) \mbox{ for $|x|<r$}, \quad
 \widetilde{\varphi} (x)=\ell x_1 \mbox{ for $|x|>2r$},\quad \nabla \widetilde{\varphi} (0)=\ell \mathbf{e}_1.
\end{equation}
On the other hand, from $\varphi(0)=0$ and $\nabla\varphi(0)=\ell\mathbf{e}_1$, there exists a constant $C>1$ such that for
$|x|<1$, it holds $|\varphi(x)-\ell x_1|\leq C |x|^2$ and $|\nabla \varphi(x)- \ell\mathbf{e}_1 |\leq C |x|$.  In particular, since
\[
\nabla  \widetilde{\varphi} (x)=(\nabla \varphi(x)-\ell\mathbf{e}_1) \chi\left(\frac{|x|}r\right)
+\ell \mathbf{e}_1+\frac 1r (\varphi(x)-\ell x_1)\chi'\left(\frac{|x|}r\right)  \frac {x} { |x|} ,
\]
it holds on $\R^N$,
\[
|\nabla  \widetilde{\varphi} (x)-\ell \mathbf{e}_1|\leq  Cr.
\]
We fix $r>0$ small enough so that
\begin{equation}\label{on:gradvarphit}
\|\nabla  \widetilde{\varphi}  -\ell\mathbf{e}_1\|_{L^\infty}
\leq  (1-\ell )
 \min \Bigl\{  \frac \lambda 8  \frac{p-1}{p+1} , \frac {1} {2}\Bigr\} .
\end{equation}
The first constraint on $ \widetilde{\varphi} $ is related to assumption \eqref{c:psi} in Proposition~\ref{pr:unif}, and the second implies
\begin{equation} \label{gradphi} 
 \| \nabla    \widetilde{\varphi}   \| _{ L^\infty  } \le \frac {\ell +1} {2} < 1.
\end{equation} 

\subsection{Construction of the function $\psi $} \label{sConsPsi} 
We claim that for any $y\in \R^N$, there exists $X_1(y)\in \R$ such that
\begin{equation}\label{def:X}
y_1 = \frac{X_1(y)-\ell  \widetilde{ \varphi} (X_1(y),\bar y)}{(1-\ell^2)^{\frac 12}}.
\end{equation}
(As observed before, we ignore $\bar y$ in dimension $1$.)
To prove the claim, we define
\begin{equation} \label{fDfnPhi} 
\Phi(x_1,  \bar y)=\frac{x_1-\ell  \widetilde{ \varphi} (x_1,\bar y)}{(1-\ell^2)^{\frac 12}},
\end{equation} 
and we compute, using~\eqref{gradphi},
\begin{equation} \label{fEstDerPhi} 
\partial_{x_1}\Phi(x_1 , \bar y)
=\frac{1-\ell \partial_{x_1} \widetilde{ \varphi} (x_1,\bar y)}{(1-\ell^2)^{\frac 12}}
\ge \frac{1-\ell }{(1-\ell^2)^{\frac 12}} =  \Bigl( \frac {1- \ell } {1+ \ell} \Bigr)^{\frac {1} {2}} >0 ,
\end{equation} 
and
\begin{equation} \label{fEstsup} 
\partial_{x_1}\Phi(x_1 , \bar y) \le \frac{1 + \ell  }{(1-\ell^2)^{\frac 12}}
\le  \Bigl( \frac {1+ \ell } {1- \ell} \Bigr)^{\frac {1} {2}} .
\end{equation} 
Thus, for fixed $\bar y\in \R^{N-1}$, the function $x_1\in \R\mapsto  \widetilde{\Phi } _{ \bar y } (x_1) =:  \Phi(x_1 , \bar y)\in \R$ is  increasing and surjective.
It has an inverse function $  \widetilde{\Phi } _{ \bar y }^{-1}$ on $\R$, which is also (strictly) increasing, and we set  $ X_1(y_1 , \bar y) =  \widetilde{\Phi } _{ \bar y }^{-1}(y_1)$ for $y_1\in \R$.
Setting $X_1(y)=X_1(y_1 , \bar y)$, we have proved the claim.
Note that 
\begin{equation*} 
X_1 ( \Phi(x_1,  \bar y), \bar y )=   \widetilde{\Phi } _{ \bar y }^{-1}(  \Phi(x_1,  \bar y) ) =   \widetilde{\Phi } _{ \bar y }^{-1}(  \widetilde{\Phi } _{ \bar y }(x_1) ) = x_1,
\end{equation*} 
so that by~\eqref{fDfnPhi} 
\begin{equation} \label{faltDfnPhi} 
x_1 = X_1  \left( \frac {x_1- \ell  \widetilde{\varphi }(x)  } {(1- \ell^2 )^{\frac {1} {2}}} , \bar x \right)
\end{equation} 
for all $x\in \R^N  $.
Moreover, it follows from~\eqref{fEstDerPhi}-\eqref{fEstsup} that
\begin{equation} \label{fEstDerXun} 
 \Bigl( \frac {1- \ell } {1+ \ell} \Bigr)^{\frac {1} {2}} \le \frac {\partial X_1} {\partial   y_1 }  \le  \Bigl( \frac {1+ \ell } {1- \ell} \Bigr)^{\frac {1} {2}}
\end{equation} 
on $\R^N $. 
Setting $X(y)=(X_1(y),\bar y)$,   it holds
\begin{equation}\label{def:XX}
y_1 = \frac{X_1(y)-\ell  \widetilde{ \varphi} (X(y))}{(1-\ell^2)^{\frac 12}}.
\end{equation}
Moreover, using~\eqref{fEstDerXun}, we see that that
\begin{equation} \label{fEstDerPhi3} 
 |X (y) |\ge \max\{  |X_1 (y)|,  |\bar y| \} \goto _{  |y|\to \infty  } \infty .
\end{equation} 

For all $y\in \R^N$, we define the function $\psi:\R\to \R$ by
\begin{equation}\label{def:psi}
\psi(y)=\frac{ \widetilde{\varphi} (X(y))-\ell X_1(y)}{(1-\ell^2)^{\frac 12}}.
\end{equation}
Equivalently, the functions $\psi$ and $ \widetilde{ \varphi} $ are uniquely related by the following relation on $\R^N$:
\begin{equation}\label{def:psi2}
 \widetilde{\varphi} (x)=(1-\ell^2)^{\frac 12} \psi\left(\frac{x_1-\ell  \widetilde{\varphi} (x)}{(1-\ell^2)^{\frac 12}},\bar x\right)+\ell x_1.
\end{equation} 
We check 
that $\psi $ is of class $\mathcal C^{q_0}$ where $q_0$ is defined in \eqref{th2:1}, and
satisfies the assumptions \eqref{on:psi} and \eqref{c:psi}.

First, since $\varphi$ is of class $\mathcal C^{q_0}$ and $\chi$ is of class $\mathcal C^\infty$,
it follows from their definitions  that $ \widetilde{ \varphi} $ and then the functions $X$ and $\psi$ are of class $\mathcal C^{q_0}$
in $\R^N$. Since $\varphi(0)= \widetilde{ \varphi} (0)=0$, from \eqref{def:psi2}, we also have $\psi(0)=0$.

Second, from~\eqref{on:varphit}, it follows that $ \widetilde{\varphi}  (x) =\ell x_1$ for any $|x|>2r$.
From~\eqref{fEstDerPhi3} and~\eqref{def:psi}, we see that $\psi(y)=0$  for $ |y|$ large.

Last, we estimate $|\nabla \psi|$.
From \eqref{def:psi2}
\begin{equation}  \label{estpsi1} 
(1- \ell \partial  _{ x_1 }  \widetilde{ \varphi}  ( x))
\partial_{y_1}\psi\left(\frac{x_1-\ell  \widetilde{\varphi} (x)}{(1-\ell^2)^{\frac 12}},\bar x\right)
=\partial_{x_1} \widetilde{\varphi} (x)-\ell, 
\end{equation} 
and for $j\neq 1$,
\begin{equation}  \label{estpsi2} 
\partial_{y_j}\psi\left(\frac{x_1-\ell  \widetilde{ \varphi} (x)}{(1-\ell^2)^{\frac 12}},\bar x\right)
=(1-\ell^2)^{-\frac 12}   \partial_{x_j} \widetilde{\varphi} (x) \left( 1 + 
 \ell \partial_{y_1}\psi\left(\frac{x_1-\ell \widetilde{ \varphi} (x)}{(1-\ell^2)^{\frac 12}},\bar x\right)  \right) .
\end{equation} 
It follows from~\eqref{gradphi} that $ | 1- \ell \partial  _{ x_1 }  \widetilde{ \varphi}  ( x) | \ge 1- \ell $, so that~\eqref{estpsi1} and~\eqref{on:gradvarphit} yield
\begin{equation*} 
 \| \partial  _{ y_1 } \psi  \| _{ L^\infty  }\le \frac {1} {1-\ell }  \| \partial  _{ x_1 }  \widetilde{ \varphi}  - \ell  \| _{ L^\infty  } 
 \le \frac \lambda 8  \frac{p-1}{p+1} .
\end{equation*} 
In particular, we see that $ \| \partial  _{ y_1 } \psi  \| _{ L^\infty  }\le  1$. Since $  \|   \partial_{x_j} \widetilde{\varphi}  \| _{ L^\infty  } \le (1-\ell ) \frac \lambda 8  \frac{p-1}{p+1}$ by~\eqref{on:gradvarphit}, we deduce from~\eqref{estpsi2} that
\begin{equation*} 
 \| \partial  _{ y_j } \psi  \| _{ L^\infty  }\le (1-\ell^2)^{-\frac 12}   (1-\ell ) \frac \lambda 8  \frac{p-1}{p+1} (1+\ell ) 
 = (1-\ell^2)^{\frac 12}   \frac \lambda 8  \frac{p-1}{p+1}\le  \frac \lambda 8  \frac{p-1}{p+1}
\end{equation*} 
so that~\eqref{c:psi} is proved.

\subsection{Definition of an appropriate solution of the transformed equation} \label{s:4:3} 
We assume~\eqref{defJ}, \eqref{onqk}, \eqref{on:lam}, \eqref{on:kk} and we consider the function $\psi$ defined in \eqref{def:psi}-\eqref{def:psi2}.
Note that $\psi $ is of class $\mathcal C^{q_0}$ where $q_0$ is defined in \eqref{th2:1}, and
satisfies the assumptions \eqref{on:psi} and \eqref{c:psi}. 
Let the function $A$ be given by~\eqref{fExfnA}. 
We consider the solution $v\in C ( (0,\delta_0),H^2(\R^N))\cap  C^1 ((0,\delta_0),H^1(\R^N))\cap C^2 ((0,\delta_0),L^2(\R^N))$ of~\eqref{eq:v} given by Proposition~\ref{eExisV}.

\subsection{Returning to the original variable} \label{s:4:4} 

\begin{figure}
\setlength{\unitlength}{1.2pt}
\begin{picture}(220,120)
\thinlines
\put(10, 10){\line(1, 1){100}}
\put(210, 10){\line(-1, 1){100}}
\multiput(10,80)(3,0){66}{\line(1,0){2}}
\put(110, 0){\vector(0, 1){120}}
\put(0, 10){\vector(1, 0){220}}

\multiput(80,10)(0,3){23}{\line(0,1){2}}
\multiput(140,10)(0,3){23}{\line(0,1){2}}

\multiput(95,10)(0,3){22}{\line(0,1){2}}
\multiput(125,10)(0,3){25}{\line(0,1){2}}

\put(214, 3){$ |x| $}
\put(120, 2){$  \frac {\varepsilon _0} {2}$}
\put(136, 3){$ \varepsilon _0$}
\put(160, 40){$ |x| = \tau _0 +  \varepsilon _0- t$}
\put(115, 115){$t$}
\put(10, 74){$ t= \tau _0$}
\put(15, 100){$t= \tau _0 +  \widetilde{\varphi }(x) $}

\linethickness{.25mm}

\qbezier(10, 100)(45, 60)(110, 80)

\qbezier(110, 80)(165, 100)(210, 60)

\end{picture}
\caption{The set $\mathcal T$ is the part of the cone $ |x| < \tau _0+ \varepsilon _0 -t$ below the surface $t= \tau _0+  \widetilde{\varphi } (x)$} 
\label{thesetT}
\end{figure}
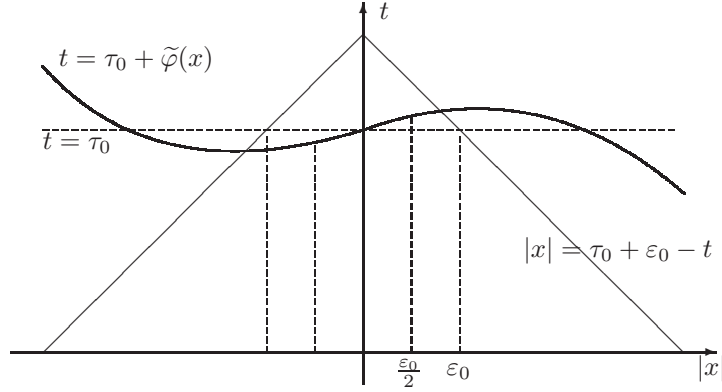

Let
\begin{equation} \label{fDfntaueps} 
\tau _0 =  \Bigl( \frac {1 - \ell } {1+ \ell } \Bigr)^{\frac {1} {2}} \frac {\delta _0} {6}, \quad \varepsilon _0= \frac {1- \ell } {2+ \ell} \tau _0 .
\end{equation} 
Recall (see~\eqref{on:varphit} and~\eqref{gradphi}) that $ \widetilde{\varphi } (0)= 0$ and $ |\nabla  \widetilde{\varphi } |\le \frac {\ell +1} {2}$, so that $ |  \widetilde{\varphi } (x) | \le \frac {\ell +1} {2}  |x|$. Thus 
 we see that 
\begin{equation} \label{fInfTzTF} 
\tau _0 + \inf  _{  |x|\le \tau _0+ \varepsilon _0 }  \widetilde{\varphi } (x) >0.
\end{equation}
It follows that the space-time region 
\[
\mathcal T=\{(t,x)\in \R^{1+N}_+ ; \, 0 \le t<\tau _0 +  \widetilde{\varphi} (x),\ |x|<\tau _0+\varepsilon_0 - t\}
\]
is an influence domain in the sense of~\S\ref{S1.3}. (See Figure~\ref{thesetT}.)
Moreover, let $ |x|\le \frac {\varepsilon _0} {2}$. We have $  \widetilde{\varphi } (x) \le \frac {\varepsilon _0} {2}$. 
Therefore, if $0 \le  t< \tau _0+  \widetilde{\varphi } (x)$, then $t < \tau _0 + \frac {\varepsilon _0} {2}$ so that $|x|< \tau _0 +\varepsilon _0 -t$. It follows that 
\begin{equation} \label{fupdbus} 
 |x|\le \frac {\varepsilon _0} {2} \Longrightarrow \max\{ t>0;\, (t,x) \in \mathcal T  \}= \tau _0 +  \widetilde{\varphi }(x) . 
\end{equation} 

Given $0\le \ell <1$ and $\tau _0\in \R$,  we define the Lorentz transform  $\Map  _{ \ell , \tau _0 } : \R^{1+N} \to \R^{1+N}$ by
\begin{equation*} 
\begin{cases} 
\Map _{ \ell, \tau _0 } (t,x)= (s, y)= (s, y_1, \bar y)  \text{ where} \\ \displaystyle 
s= \frac{t -\tau _0-\ell x_1}{(1-\ell^2)^{\frac 12}}  , \quad y_1= \frac{x_1-\ell (t -\tau _0) }{(1-\ell^2)^{\frac 12}} , \quad \bar y= \bar x.
\end{cases} 
\end{equation*} 
It is well known that $\Map _{ \ell, \tau _0 }$ is a $\mathcal C^\infty $ diffeomorphism with Jacobian determinant $ |\det J _{ \Map _{ \ell, \tau _0 } } | = 1$. 
We also define the transformation $\Map  _{ \psi  } : \R^{1+N} \to \R^{1+N}$ by
\begin{equation*} 
\begin{cases} 
\Map _{ \psi  } (t' ,x' )= (s', y')  \text{ where} \\ \displaystyle 
s' =\psi (x') -t' , y'= x'.
\end{cases} 
\end{equation*} 
Since $\psi $ is of class $\mathcal C^{q_0}$ where $q_0$ is defined in \eqref{th2:1} (see \S\ref{sConsPsi}), it follows easily that $\Map  _{ \psi  }$ is a diffeomorphism of class $\mathcal C^{q_0}$.
Moreover, $ |\det J _{ \Map _{ \psi  } }|= 1$. 
We define the map $\Map :  \R^{1+N} \to \R^{1+N}$ as the composition of the above two maps, i.e.
\begin{equation*} 
\Map= \Map _{ \psi  } \circ \Map  _{ \ell, \tau _0 }.
\end{equation*} 
The map $\Map $ has the following expression
\begin{equation} \label{fDfnMap} 
\begin{cases} 
\Map (t,x)= (s, y)= (s, y_1, \bar y)  \text{ where} \\ \displaystyle 
s=\psi (y) -\frac{t -\tau _0-\ell x_1}{(1-\ell^2)^{\frac 12}}  , \quad y_1= \frac{x_1-\ell (t -\tau _0) }{(1-\ell^2)^{\frac 12}} , \quad \bar y= \bar x
\end{cases} 
\end{equation} 
and it follows that $\Map : \R^{1+N } \to \R^{1+N }$ is a diffeomorphism of class $\mathcal C^{q_0}$ and that $ |\det J _{ \Map  }|= 1$. 

We prove that 
\begin{equation} \label{fBorneMapT1} 
s> 0 \Longleftrightarrow t < \tau _0 +  \widetilde{\varphi }(x)
\end{equation} 
and that
\begin{equation} \label{fBorneMapT} 
 \Map ( \mathcal T) \subset  \Bigl( 0, \frac {\delta _0} {2} \Bigr) \times \R^N .
\end{equation}  

In the case where $\ell=0$, by \eqref{def:X}, we have $X(y)=y$ and thus by \eqref{def:psi}, $\psi(y)= \widetilde{\varphi} (y)$.
Thus in this case, 
\begin{equation} \label{derdesder5} 
\Map (t,x)= ( \widetilde{\varphi} (x)-t + \tau _0 , x). 
\end{equation} 
Property~\eqref{fBorneMapT1} follows.
Moreover, $s \le  \widetilde{\varphi } (x) +\tau _0 \le  |x| +\tau _0 \le 2 \tau _0+\varepsilon _0 \le 3 \tau _0 <\frac {\delta _0} {2}$ by~\eqref{fDfntaueps}. Thus~\eqref{fBorneMapT} is proved in this case.

In the case where $\ell\neq0$, we observe that from \eqref{def:psi},
\[
s=\frac{ \widetilde{\varphi} (X(y))-\ell X_1(y)}{(1-\ell^2)^{\frac 12}}-\frac{t - \tau _0 -\ell x_1}{(1-\ell^2)^{\frac 12}} .
\]
Using \eqref{def:XX}, we replace $  \widetilde{\varphi} (X(y))=\frac1\ell ( X_1(y)-(1-\ell^2)^{\frac 12}y_1  ) $
so that
\begin{equation*} 
\begin{split} 
\ell (1-\ell^2)^{-\frac 12}s & =
\frac{ \ell  \widetilde{\varphi} (X(y))-\ell^2 X_1(y)}{(1-\ell^2) } - \frac{ \ell (t -\tau _0)-\ell ^2 x_1}{(1-\ell^2) } \\
& = X_1 ( y) - \frac {y_1} {(1- \ell^2)^{\frac {1} {2}}} +   \frac{ x_1 - \ell (t- \tau _0) }{(1-\ell^2) } -x_1 \\
 & =X_1(y)-x_1.
\end{split} 
\end{equation*} 
Recall that by~\eqref{faltDfnPhi}, we have $x_1=X_1\Big(\frac{x_1-\ell  \widetilde{\varphi} (x)}{(1-\ell^2)^{\frac 12}} , \bar x\Big)$, which means that
\begin{equation} \label{derdesder} 
\ell (1-\ell^2)^{-\frac 12}s=X_1\left(\frac{x_1-\ell (t- \tau _0) }{(1-\ell^2)^{\frac 12}} , \bar x\right)
-X_1\left(\frac{x_1-\ell  \widetilde{\varphi} (x)}{(1-\ell^2)^{\frac 12}} , \bar x\right),
\end{equation} 
hence, using~\eqref{fEstDerXun},
\begin{equation} \label{fEstDerS} 
-  \Bigl( \frac {1+ \ell } {1- \ell} \Bigr)^{\frac {1} {2}} \le \frac {\partial s} {\partial   t }  \le -  \Bigl( \frac {1- \ell } {1+ \ell} \Bigr)^{\frac {1} {2}} .
\end{equation} 
on $\R^{1+N} $. 
Thus we see that  $s>0$ is equivalent to $t< \tau _0 +  \widetilde{\varphi} (x)$, i.e.~\eqref{fBorneMapT1} holds. 
Moreover, by~\eqref{fEstDerS}, we have on $\mathcal T$
\begin{equation*} 
s \le \Bigl( \frac {1+ \ell } {1- \ell} \Bigr)^{\frac {1} {2}}  | t - \tau _0 -  \widetilde{\varphi } (x) |
= \Bigl( \frac {1+ \ell } {1- \ell} \Bigr)^{\frac {1} {2}}   ( \tau _0 +  \widetilde{\varphi } (x) -t ) .
\end{equation*}  
Using~\eqref{fDfntaueps}, we see that 
 $  \tau _0 +  \widetilde{\varphi } (x) -t< \widetilde{\varphi } (x) +\tau _0 \le 3 \tau _0 $, so that $s< \frac {\delta _0} {2}$. Thus~\eqref{fBorneMapT} is proved in all cases.

We now set
\begin{equation} \label{fDefnu} 
u(t,x)= v ( \Map (t,x ) ), \quad (t, x) \in \mathcal T.
\end{equation} 
We refer to \cite{Kibook}, Exercise~10.7.c   for a similar use of the Lorentz transform. 
Note that by~\eqref{fBorneMapT}, $u$  is well defined. 

Let  $\omega $ be an open subset of $\R^N $ and let $0\le a<b$. Suppose that $[a,b] \times  \overline{\omega } \subset   \mathcal T  $.
 We claim that 
\begin{gather} 
u\in H^2 ( (a,b) \times \omega  )  , \label{fuesth2} \\
u\in L^q ( (a,b) \times \omega  ) \quad  \text{for all }1\le q<\infty , \label{fuestlq} \\
\partial  _{ tt }u = \Delta u +  |u|^{p-1}u\quad  \text{in }L^2 ( (a,b)\times \omega  ).   \label{fuestsol}   
\end{gather} 
Since $[a,b] \times  \overline{\omega } $ is a compact subset of $\R^{1+N}$, it follows that $\Map ( [a,b] \times  \overline{\omega }  )$  is a compact subset of $\R^{1+N}$.
Moreover, it follows from~\eqref{fBorneMapT1}-\eqref{fBorneMapT} that $\Map ( [a,b] \times  \overline{\omega }  )$  is a compact subset of $ (0,\delta _0) \times \R^N  $. 
Let $1\le q<\infty $. Since $v\in C((0, \delta _0), H^2 (\R^N ) )$ and $H^2 (\R^N ) \hookrightarrow L^r (\R^N ) $ for every $r<\infty $ (because $N\le 4$), we have $v\in L^q ( \Map ((a,b)\times \omega ) )$; and so~\eqref{fuestlq} follows from~\eqref{fDefnu} and the change of variable formula. 
Next, let $\theta \in C^\infty _\Comp ( (0,\delta _0) \times \R^N  )$ such that $\theta (x)\equiv 1$ on $\Map ( [a,b] \times  \overline{\omega }  )$. 
Thus we may replace $v$ by $\theta v$ in formula~\eqref{fDefnu}, this does not change the values of $u$ on $(a, b) \times \omega $. Since $\theta v\in H^2 ( (0,\delta _0) \times \R^N  )$, we can approximate $\theta v$ in $H^2 ( (0,\delta _0) \times \R^N  )$ by a sequence $(w_n) _{ n\ge 1 }\subset C^\infty _\Comp ( (0,\delta _0) \times \R^N  )$ supported in a fixed compact of $(0,\delta _0) \times \R^N $. Setting $u_n= w_n \circ \Map $, we have
\begin{equation} \label{funcauchyl2} 
\begin{split} 
\int  _{ (a,b)\times \omega  } |u_n- u |^2 & = \int  _{ \Map ((a,b)\times \omega ) }  |w_n - v|^2  |\det J_\Map|^{-1} \\
&=  \int  _{ \Map ((a,b)\times \omega ) }  |w_n - \theta v|^2   \goto _{ n\to \infty  } 0 .
\end{split} 
\end{equation} 
Next, it follows from~\eqref{fDefnu} that
\begin{equation*} 
\begin{split} 
(1-\ell^2) \partial_{tt} & u_n  =   \Bigl[   \ell^2 ( \partial_{y_1}\psi (\cdot )  ) ^2 + 2\ell \partial_{y_1}\psi( \cdot  ) + 1 \Bigr] \partial_{ss} w_n ( \cdot   ) \\ & 
 + \ell^2 \partial_{y_1y_1}\psi( \cdot  )\partial_s w_n ( \cdot   ) 
+ 2\ell ( \ell \partial_{y_1}\psi( \cdot  )+ 1 ) \partial_{sy_1}w_n ( \cdot  ) + \ell^2  \partial_{y_1y_1}w_n ( \cdot  ) 
\end{split} 
\end{equation*} 
and
\begin{equation*} 
\begin{split} 
(1-\ell^2) \Delta u_n  = &  \Bigl[  (\partial_{y_1}\psi( \cdot  ) ) ^2 + (1-\ell^2) \sum_{k\neq 1} ( \partial_{y_k}\psi( \cdot  ) ) ^2+
2\ell  \partial_{y_1}\psi( \cdot  )
+ \ell^2 \Bigr] \partial_{ss} w_n ( \cdot  )\\
& +  \partial_{y_1y_1}\psi( \cdot  )   \partial_{s} w_n ( \cdot  )
 +  2 ( \partial_{y_1}\psi(\cdot)+ \ell ) \partial_{sy_1}w_n ( \cdot  ) \\
&+2 (1-\ell^2) \sum_{k\neq 1}\partial_{y_k}\psi( \cdot  )\partial_{sy_k} w_n( \cdot   )  + \partial_{y_1y_1}w_n ( \cdot  )  + (1-\ell^2) \bar \Delta w_n ( \cdot  ) \\
& + (1- \ell^2) ( \bar \Delta \psi (\cdot ) ) \partial _s w_n (\cdot )
\end{split} 
\end{equation*} 
where the argument of $\psi $ is $y$ and the argument of $w_n$ is $\Map$. 

Similar formulas hold for all first and second space-time derivatives of $u_n$, so arguing as in~\eqref{funcauchyl2} we conclude that $u_n$ is a Cauchy sequence in $H^2 ( (a,b)\times \omega  )$, from which~\eqref{fuesth2} follows.
In addition, the above two  formulas imply that
\begin{equation*} 
\begin{split} 
\partial_{tt}u_n & -\Delta u_n   \\ &
= [1-|\nabla \psi(\cdot)|^2] \partial_{ss} w_n (\cdot)
-2 \nabla \psi(\cdot)\cdot \nabla\partial_s w_n (\cdot)-\Delta \psi(\cdot)\partial_s w_n (\cdot)
 - \Delta w_n (\cdot) .
\end{split} 
\end{equation*} 
Since $u_n \to u$ in $H^2 ( (a,b)\times \omega  )$ and $w_n \to \theta v$ in $H^2 ( (0,\delta _0) \times \R^N  )$, we may pass to the limit in the above equation. Since $\theta v= v$ in $\Map ( (a,b)\times \omega ) $, we obtain using~\eqref{eq:v} 
\begin{equation*} 
\begin{split} 
\partial_{tt}u  -\Delta u   
&= [1-|\nabla \psi(\cdot)|^2] \partial_{ss} v (\cdot)
-2 \nabla \psi(\cdot)\cdot \nabla\partial_s v (\cdot)-\Delta \psi(\cdot)\partial_s v (\cdot)
 - \Delta v (\cdot) \\
 &= ( |v|^{p-1} v) (\cdot ) =  |u|^{p-1} u
\end{split} 
\end{equation*} 
in $L^2 ( (a,b)\times \omega  ) $. This proves~\eqref{fuestsol}.  

Set 
\begin{equation*} 
  \widetilde{\rho }  = \tau _0 + \frac {\varepsilon _0} {2}
\end{equation*} 
and 
\begin{equation*} 
  \widetilde{\tau } = \min  \Bigl\{  \frac {\varepsilon _0} {2} ,  \tau _0 +  \inf  _{  |x|\le \tau _0+ \varepsilon _0 } \{ \widetilde{\varphi } (x) \} \Bigr\} 
\end{equation*} 
so that $ \widetilde{\tau }  >0$ by~\eqref{fInfTzTF}. 
We see that $(0,  \widetilde{\tau }) \times B _{  \widetilde{\rho }  } \subset \mathcal T$ so that $u\in H^2 ( (0,  \widetilde{\tau }) \times B _{  \widetilde{\rho }  }) \cap L^q ( (0,  \widetilde{\tau }) \times B _{  \widetilde{\rho }  }) $ for all $q<\infty $. In particular, $u\in C([0,  \widetilde{\tau }], H^1( B _{  \widetilde{\rho }  } )) \cap C^1([0,  \widetilde{\tau }], L^2( B _{  \widetilde{\rho }  } ))$, so that $u(0) \in H^1( B _{  \widetilde{\rho }  } )$ and $\partial _t u(0) \in L^2( B _{  \widetilde{\rho }  } )$ are well defined. 

\subsection{Choice of a solution of the nonlinear wave equation} \label{s:4:5} 

We apply Section~\ref{S1.3} to extend  $u$, which is a solution of~\eqref{wave} on $\mathcal T$, to a solution of~\eqref{wave} on a maximal domain of influence that contains $\mathcal T$. 
For this, we consider any pair $( \widetilde{u} _0,  \widetilde{u} _1) \in H^1 (\R^N ) \times L^2 (\R^N ) $ such that $ \widetilde{u} _0 $ and $ \widetilde{u} _1$ coincide with $u(0)$ and $\partial _t u(0)$, respectively, on $B _{  \widetilde{\rho }  }$. 
The initial data $( \widetilde{u} _0,  \widetilde{u} _1)$ give rise to a solution $ \widetilde{u} $ of~\eqref{wave} defined on the maximal influence domain $\Omega_{\rm max}( \widetilde{u} _0, \widetilde{u} _1)$ in the sense of \S\ref{S1.3}. 
We claim that this maximal influence domain contains
\begin{equation*} 
\widetilde{\mathcal T}  = \mathcal T \cap \{ (t,x)\in [0,  \widetilde{\rho }  )\times \R^N ;\,  |x| <  \widetilde{\rho } - t \}
\end{equation*} 
and that $ \widetilde{u} $ coincides with $u$ on $\widetilde{\mathcal T} $.
Indeed, let $(t,x)\in \widetilde{\mathcal T}$ and consider the corresponding open backward cone $C(t,x)$. 
The cone $C(t,x)$ is an influence domain, and It follows easily, using Proposition~\ref{eLinCo4} and~\eqref{fuestlq}, that $u$ is a solution of~\eqref{wave} in $C(t,x)$ with initial data $(u_0, u_1)$, so that $C(t,x) \subset \Omega_{\rm max}( \widetilde{u} _0, \widetilde{u} _1)$. Since $(t,x)\in \widetilde{\mathcal T}$ is arbitrary, this proves the claim.
From now on, we denote by $u$ this solution. 

\subsection{Blowup on the local hypersurface and end of the proof} \label{s:4:6} 

We show blowup on the local hypersurface by proving~\eqref{fBUest}. For this, we further restrict the size of the hypersurface. 
Arguing as in the proof of~\eqref{fupdbus}, we see that 
\begin{equation*} 
 |x|\le \frac {\varepsilon _0} {4} \Longrightarrow \max\{ t>0;\, (t,x) \in  \widetilde{\mathcal T}   \}= \tau _0 +  \widetilde{\varphi }(x) . 
\end{equation*} 
Thus we see that if $ |x_0| \le \frac {\varepsilon _0} {4} $, then the open backward cone $C( \tau _0+  \widetilde{\varphi } (x_0), x_0 ) $ is a subset of $  \widetilde{\mathcal T} $. 

We fix $\ell <\sigma \le 1$ and $ |x_0| \le \frac {\varepsilon _0} {4} $, and we prove~\eqref{fBUest}. 
We use the geometric property  that the image by the map $\Map$ of a cone of slope $\sigma $ contains at least a small cone (estimate~\eqref{fInclCone}), and the lower estimate~\eqref{fNFestinf} for $v$ on this small cone.

Let $s_0\ge 0$ and $y_0 \in \R^N $ be given by $\Map ( \tau _0 +  \widetilde{\varphi } (x_0) , x_0 ) = ( s_0, y_0)$. 
We first note that $s_0= 0$ by~\eqref{def:psi} and~\eqref{fDfnMap}. Moreover, it follows from~\eqref{fDfnMap}, \eqref{gradphi} and~\eqref{fDfntaueps}  that 
\begin{equation} \label{fUDP4} 
 |y_0| \le \frac {\varepsilon _0} {4}  \Bigl( 1 + \frac {2} {(1- \ell^2)^{ \frac {1} {2} }}  \Bigr) \le 1 .
\end{equation} 
Given $0\le  t <  \tau _0 +  \widetilde{\varphi } (x_0)$, we set 
\begin{equation*} 
K(  t ) = \{ ( t' ,x)\in \R^{1 + N};\, t < t' < \tau _0 +  \widetilde{\varphi } (x_0) ,  |x- x_0 | < \sigma ( \tau _0 +  \widetilde{\varphi } (x_0)- t' ) \} 
\end{equation*} 
and, given $s>0$ and $\sigma '>0$ we set
\begin{equation*} 
L (s, \sigma ') = \{ ( s' , y ) \in \R^{1+N} ;\,  0 <  s' < s,  |y- y_0 |< \sigma ' s'  \} .
\end{equation*} 
We claim that there exist $\sigma ' >0$ and $\eta >0$ such that
\begin{equation} \label{fInclCone} 
L ( s (t)  ,  \sigma ' ) \subset  \Map ( K(  t  ) ) 
\end{equation} 
where
\begin{equation} \label{fInclCone2} 
s(t) = \eta ( \tau _0 +  \widetilde{\varphi } (x_0) - t ). 
\end{equation}
Assuming~\eqref{fInclCone}-\eqref{fInclCone2}, we conclude the proof of~\eqref{fBUest}. 
Given $(t,x)\in   \widetilde{\mathcal T} $, it follows from~\eqref{fDefnu} and~\eqref{fDfnMap} that
\begin{equation*} 
\partial _t u(t,x) = \frac {-1} { (1- \ell^2)^{\frac {1} {2}} } [ \partial _s v (\Map (t,x)) + \ell \partial  _{ y_1 } v( \Map (t,x))  ]
\end{equation*} 
so that, using $2\ell xy \le \ell ^2 x^2 + y^2$,
\begin{equation*} 
 |\partial _t u (t,x)|^2 \ge   |\partial _s v ( \Lambda (t,x) )|^2 -    |\partial  _{ y_1 }  v ( \Lambda (t,x) )|^2 .
\end{equation*} 
Therefore,
\begin{equation*} 
\int  _{ K(  t  )  }  | \partial _t u |^2  \ge  \int  _{ \Map ( K(  t ) ) } (  | \partial _s v |^2 -   | \partial  _{ y_1 } v |^2 ) 
\end{equation*} 
Applying~\eqref{fNFestinf} and~\eqref{fInclCone}, we deduce that
\begin{equation} \label{fUDP5} 
\int  _{ K(  t )  }  | \partial _t u |^2  \ge \frac {1} {4} \int  _{ L(s(t), \sigma ') }  |\partial _s V_0 |^2 - C  | \Map ( K(  t  ))  | - \int  _{ \Map ( K(  t ) ) } g^2 .
\end{equation} 
It follows from~\eqref{fInclCone2}, \eqref{dtVrsigma} and~\eqref{fUDP4}  that 
\begin{equation} \label{fUDP6} 
 \liminf  _{ t\uparrow \tau _0+  \widetilde{\varphi } (x_0)  } \frac {1} {\tau _0+  \widetilde{\varphi } (x_0) -t  } \int  _{ L(s(t), \sigma ') }  |\partial _s V_0 |^2 \ge  \liminf  _{ s\downarrow 0 } \frac {1} {s  } \int  _{ L(s , \sigma ') }  |\partial _s V_0 |^2  > 0 . 
\end{equation} 
Furthermore,
\begin{equation} \label{fUDP7} 
 | \Map ( K(  t  ))  | \lesssim  | K(t) | \lesssim (\tau _0 +  \widetilde{\varphi } (x_0) - t)^{1 + N} . 
\end{equation} 
Next,  $H^1 (\R^N ) \hookrightarrow L^\frac {2 (N+2) } {N} (\R^N ) $, so that $g^2\in  L^{\frac {N+2} {N}} ((0,\delta _0) \times \R^N )) $; and so by~\eqref{fUDP7} 
\begin{equation} \label{fUDP8} 
\int  _{ \Map ( K(  t ) ) } g^2 \lesssim  | \Map ( K(  t  ))  | ^{\frac {2} {N+2}} \lesssim  (\tau _0 +  \widetilde{\varphi } (x_0) - t)^{1 + \frac {N} {N+2}}. 
\end{equation} 
Estimate~\eqref{fBUest} follows from~\eqref{fUDP5}--\eqref{fUDP8}.

It remains to prove the claim~\eqref{fInclCone}-\eqref{fInclCone2}.
Let $(s', y) \in \R^{1+ N}_+ $ and $(t',x) \in  \R^{1+ N}$ such that $(s', y) = \Map (t', x)$. In particular, $t'\le  \tau _0+  \widetilde{\varphi } (x)$ by~\eqref{fBorneMapT1}. 
We prove that
\begin{equation} \label{derdesder2} 
s' \le    \Bigl( \frac {1+\ell} {1- \ell} \Bigr)^{\frac {1} {2}} (\tau _0 +  \widetilde{\varphi } (x_0)- t' +  |x-x_0|) . 
\end{equation} 
In the case $\ell =0$, this follows from~\eqref{derdesder5} and the inequality $ | \widetilde{\varphi }(x) - \widetilde{\varphi }(x_0)  |\le  |x- x_0|$ (see~\eqref{gradphi}).
In the case $\ell \not = 0$, then by~\eqref{fDfnMap} and~\eqref{derdesder},
\begin{equation*} 
\ell (1- \ell ^2)^{-\frac {1} {2}} s' = X_1 (y_1,  \bar{y})  - X_1 \Bigl(  y_1 - \frac {\ell ( \tau _0 +  \widetilde{\varphi }(x) - t') } { (1- \ell ^2)^{\frac {1} {2}}} ,  \bar{y} \Bigr) .
\end{equation*} 
Using the right-hand side inequality in~\eqref{fEstDerXun}, and then~\eqref{gradphi}, we deduce
\begin{equation*} 
s' \le   \Bigl( \frac {1+\ell} {1- \ell} \Bigr)^{\frac {1} {2}} (\tau _0 +  \widetilde{\varphi } (x)- t') 
\end{equation*} 
and~\eqref{derdesder2} by using again~\eqref{gradphi}.  

Next we claim that
\begin{equation} \label{derdesder3} 
 |x - x_0| \le  | y-y_0 | + \ell (\tau _0 +  \widetilde{\varphi } (x_0)- t') .
\end{equation} 
Indeed, by~\eqref{fDfnMap} for $(t' ,x)$ and for $(\tau _0 +  \widetilde{\varphi }(x_0) , x_0)$,
\begin{align*} 
y_1 -(y_0)_1 & = \frac {x_1- (x_0)_1} { (1- \ell^2)^{\frac {1} {2}}}  +  \frac {\ell ( \tau _0 +  \widetilde{\varphi }(x_0) - t') } { (1- \ell ^2)^{\frac {1} {2}}} \\
\bar y- \bar y_0 & = \bar x- \bar x_0 
\end{align*} 
so that
\begin{align*} 
 | x_1 -(x_0)_1 | & \le   | y_1 -(y_0)_1 | +  \ell ( \tau _0 +  \widetilde{\varphi }(x_0) - t')  \\
 |\bar x- \bar x_0 |& = | \bar y- \bar y_0 | .
\end{align*} 
Estimate~\eqref{derdesder3} follows by using the triangle inequality $\sqrt{ (a+b)^2 + c^2 } \le \sqrt{ a^2+c^2 } +  |b|$. 
Assuming now $(s' , y) \in L (s, \sigma ')$ for some $s>0$ and $\sigma '>0$, we deduce from~\eqref{derdesder3} that 
\begin{equation*} 
 |x - x_0| \le \sigma ' s' + \ell (\tau _0 +  \widetilde{\varphi } (x_0)- t') .
\end{equation*} 
Estimating $s'$ by~\eqref{derdesder2}, we obtain
\begin{equation*} 
 \Bigl( 1 - \sigma '   \Bigl( \frac {1+\ell} {1- \ell} \Bigr)^{\frac {1} {2}} \Bigr)  |x - x_0|  \le  
  \Bigl( \ell +  \sigma '   \Bigl( \frac {1+\ell} {1- \ell} \Bigr)^{\frac {1} {2}} \Bigr) ( \tau _0 +  \widetilde{\varphi }(x_0) - t') .
\end{equation*} 
Since $\sigma >\ell $, we see that if $\sigma '>0$ and $\delta >0$ are sufficiently small, then
\begin{equation}  \label{derdesder4} 
 |x - x_0|  \le   ( \sigma -\delta ) ( \tau _0 +  \widetilde{\varphi }(x_0) - t') .
\end{equation} 
It now remains to prove that if $ s' \le \eta ( \tau _0+  \widetilde{\varphi } (x_0) - t )$ for some sufficiently small $\eta >0$, then $t' \ge t$. 
By~\eqref{fEstDerS}, and then~\eqref{gradphi}, we deduce
\begin{equation*} 
s' \ge  \Bigl( \frac {1-\ell } {1 + \ell } \Bigr)^{\frac {1} {2}} (\tau _0 +  \widetilde{\varphi } (x ) - t' )
\ge  \Bigl( \frac {1-\ell } {1 + \ell } \Bigr)^{\frac {1} {2}} (\tau _0 +  \widetilde{\varphi } (x_0 ) - t' -  |x-x_0|) .
\end{equation*} 
Using~\eqref{derdesder4} we obtain
\begin{equation*} 
s' \ge (1- \sigma +\delta ) \Bigl( \frac {1-\ell } {1 + \ell } \Bigr)^{\frac {1} {2}} (\tau _0 +  \widetilde{\varphi } (x_0 ) - t' ) 
\end{equation*}
which proves the claim for $\eta = (1- \sigma +\delta ) ( \frac {1-\ell } {1 + \ell } )^{\frac {1} {2}}$.

Finally, we prove that the hypersurface $\{ (t,x)\in \R^{1+N}_+;\,  |x_0|< \frac {\varepsilon _0} {4},  t= \tau _0 +  \widetilde{\varphi } (x_0)   \}$ is contained in the upper boundary of the maximal influence domain $\Omega  _{ \mathrm {max} }$ of the solution~$u$. Indeed, otherwise there would exist $ |x_0|< \frac {\varepsilon _0} {4}$ and $t >\tau _0+  \widetilde{\varphi } (x_0) $ such that $ C(t, x_0) \subset \Omega  _{ \mathrm {max} }$ with the notation~\eqref{defcone}. In particular, 
\begin{equation*} 
\partial _tu \in C([0, \tau _0 +  \widetilde{\varphi } (x_0)), L^2 ( \{  |x-x_0| < 
\textstyle{ \frac {t- \tau _0-  \widetilde{\varphi }(x_0) } {2} }
\} )).
\end{equation*}  
This is absurd, since by~\eqref{fBUest}, given $\ell < \sigma  \le 1$, there exist a sequence $t_n \uparrow  \tau _0 +  \widetilde{\varphi } (x_0)$ and $\delta >0$ such that 
\begin{equation} \label{fUDP9} 
\int  _{ \{  |x-x_0| < \sigma (\tau _0 +  \widetilde{\varphi } (x_0) -t _n) \} }  |\partial _t u (t_n) |^2 \ge \delta .
\end{equation} 
This completes the proof of the theorem, where $\tau _0$ and $\varepsilon _0$ are given by~\eqref{fDfntaueps}, and $\varepsilon = \min \{ \frac {\varepsilon _0} {4}, r \}$ with $r$ defined in Section~\ref{TH:2} (recall that $\varphi =  \widetilde{\varphi } $ on $\{  |x|<r \}$).

\appendix
\section{The wave equation~\eqref{eq:v}} \label{sAppA} 
Let $\psi \in \cont^2 (\R^N ) \cap W^{2, \infty } (\R^N ) $ satisfy $ \| \nabla \psi  \| _{ L^\infty  } <1$.
It follows in particular that $(1-  |\nabla \psi |^2) ^{-1} \in C^1 (\R^N ) \cap W^{1, \infty } (\R^N ) $.

\subsection{The associated semigroup}
Let $X$ be the Hilbert space $H^1\times L^2$, equipped with the (equivalent) scalar product
\[
\langle (a,b),(\tilde a,\tilde b)\rangle_X = \int \left(\nabla a\cdot \nabla \tilde a + a \tilde a\right)
+\int   b \tilde b (1-|\nabla \psi|^2),
\]
and consider the linear operator $\mathcal A$ on $X$ defined by
\[
\mathcal A  = 
\begin{pmatrix}
0 & 1\\
\frac {\Delta-1} {1-|\nabla \psi|^2} & \frac {2\nabla \psi \cdot \nabla + \Delta \psi}{1-|\nabla \psi|^2}
\end{pmatrix},
\]
with domain $D(\mathcal A)=H^2\times H^1$.
We compute
\[
\langle\mathcal A (a,b),(a, b)\rangle_X =
\int \left( \nabla a \cdot \nabla b + ab\right) + \int (\Delta a - a) b
+ (2\nabla \psi \cdot \nabla b + ( \Delta \psi ) b) b=0,
\]
which proves that $\mathcal A$ is dissipative in $X$. Moreover, 
for any $(c,d)\in X$, there exist $(a,b)\in D(\mathcal A)$ such that
$(a,b) -\mathcal A(a,b) =(c,d) $. Indeed, this system reduces to
\begin{equation*} 
\begin{cases} 
b=a-c\\
2a-\Delta a-2\nabla \psi \cdot \nabla a - (\Delta \psi) a = 
-2\nabla \psi \cdot \nabla c - (\Delta \psi) c+ c+ (1-|\nabla \psi|^2)d.
\end{cases} 
\end{equation*} 
It is easy to solve the second equation by the Lax-Milgram theorem, and we obtain a solution $a\in H^1 (\R^N ) $. Since, by the equation, $\Delta a\in L^2 (\R^N ) $, we see that $a\in H^2 (\R^N ) $. 
The first equation then yields $b\in H^1 (\R^N ) $. 
In particular $\mathcal A$ is maximal dissipative, hence is the generator of a $C_0$ semigroup of contractions $(e^{ t  \mathcal A}) _{ t\ge 0 }$ on $X$. (See e.g. Chapter~1, Theorem~4.3, p.~14 in~\cite{Pazy}.)

\subsection{The nonlinear equation}
Using the notation $U= \binom {v} {\partial _s v} $, we rewrite equation \eqref{eq:v} as
\begin{equation} \label{syst:ab}
\partial_s  U 
=\mathcal A U +\mathcal F (U) 
\end{equation}
where
\begin{equation} \label{deff:ab}
\mathcal F \binom{a}{b} = (1-|\nabla \psi|^2 )^{-1} \begin{pmatrix} 
0\\ f(a) + a \end{pmatrix}.
\end{equation} 

\subsection{Regularity} \label{sAppAreg} 
Suppose $T>0$ and $U\in L^\infty ((0,T) , D( \mathcal A )) \cap W^{1, \infty } ((0,T), X) $ is such that $ \mathcal F( U) \in L^\infty ((0,T), X) $ and $U$ satisfies equation~\eqref{syst:ab} for a.a. $0<t<T$. 
If $ \mathcal F( U) \in C ((0,T), D( \mathcal A )) $, then $U\in C ((0,T) , D( \mathcal A )) \cap C^1 ((0,T), X) $. Indeed, $U$ is weakly continuous $(0,T) \to  D( \mathcal A )$. In particular, $U(t) \in D( \mathcal A ) $ for all $0<t<T$ and the result follows easily, see e.g. Chapter 4, Corollary~2.6, p.~108 in~\cite{Pazy}.

\subsection{The case of equation~\eqref{eq:vn}} \label{sAppAexist} 
Equation~\eqref{eq:vn} is equation~\eqref{syst:ab}, where $f$ is replaced by $f_n$ in~\eqref{deff:ab}.  
Since $f_n(0)=0$ and $f_n $ is globally Lipschitz $ \R \to \R$, we see that the map $u \mapsto f_n (u)$ is globally Lipschitz $L^2 (\R^N ) \to L^2 (\R^N ) $. In particular, $\mathcal F: X \to  X$ is globally Lipschitz, and the existence and uniqueness of a global, mild solution $U\in C([0, \infty ), X)$ of \eqref{syst:ab} with the initial condition  $U( 0) = U_0 \in X$ is a direct consequence of standard semigroup theory.
(See e.g. Chapter 6, Theorem~1.2, p.~184 in~\cite{Pazy}.)
Moreover, since $f_n $ is globally Lipschitz and $C^1$, it follows easily that the map $u \mapsto f_n (u)$ is continuous $H^1 (\R^N ) \to H^1 (\R^N ) $. Therefore $\mathcal F$ is continuous $X \to D(\mathcal A)$, so that $\mathcal F (U) \in C ( [0, \infty ),  D(\mathcal A))$. 
It follows, again by the semigroup theory, that if the initial value is in $D(\mathcal A)$, then $U\in C([0, \infty ), D(\mathcal A)) \cap C^1 ([0, \infty ), X)$ is a solution of~\eqref{syst:ab}. 
(See e.g. Chapter 4, Corollary~2.6, p.~108 in~\cite{Pazy}.)

\section{Uniqueness on light cones}

We state and prove a uniqueness property for solutions of the nonlinear wave equation on light cones (Proposition~\ref{eLinCo4}), for which we could not find a reference. 
We first recall in the following remark the relevant results concerning the local well-posedness of the Cauchy problem. 

\begin{rem}[Local well-posedness]  \label{eLinCo5} 
Let $N\ge 1$, let $p$ such that $1<p \le \frac {N+2} {N-2}$ ($1<p<\infty $ if $N=1,2$) and let $(u_0,u_1) \in H^1 (\R^N ) \times L^2 (\R^N ) $.
We summarize some results on the existence of $T>0$ and a local solution 
\begin{equation} \label{feLinCo5:1} 
u\in C([0,T], H^1 (\R^N ) ) \cap C^1([0,T], L^2 (\R^N ) ) 
\end{equation} 
of the wave equation
\begin{equation} \label{feLinCo5:2} 
\begin{cases} 
\partial_{tt} u - \Delta u=|u|^{p-1} u \\
u(0)= u_0, \quad \partial _t u(0)= u_1.
\end{cases} 
\end{equation} 
We also discuss the property
\begin{equation} \label{feLinCo5:3} 
u\in L^{\frac {2(N+1) } {N-2}} ( (0,T) \times \R^N  )
\end{equation}
in the case $N\ge 3$. 

\begin{enumerate}
\item Case $N=1,2$. There exist $T>0$ and a unique solution $u$ of~\eqref{feLinCo5:2}  in the class~\eqref{feLinCo5:1}. See e.g.~\cite[Theorem~6.2.2]{CaHa}. 

\item Case $N\ge 3$, $p < \frac {N+2} {N-2}$. 
There exist $T>0$ and a unique solution $u$ of~\eqref{feLinCo5:2}  in the class~\eqref{feLinCo5:1}, and this solution satisfies~\eqref{feLinCo5:3} by possibly choosing $T$ smaller. 
Indeed, existence follows from~\cite[Proposition~2.3]{GV89} and uniqueness from~\cite[Proposition~3.1]{GV85}. Moreover, applying Lemma~3.3 in~\cite{GV85} with $\rho = \frac {N} {2(N-1)}$, $r= \frac {2(N^2-1) } {N^2 -2N +3}$ and $q= \frac {2(N+1)} {N-1}$, we see that $u\in L^q ((0,T), \dot B^\rho  _{ r,2 }(\R^N ))$, hence~\eqref{feLinCo5:3} by Sobolev's embedding. 

\item Case $N=3$, $p=5$. 
There exist $T>0$ and a solution $u$ of~\eqref{feLinCo5:2}  in the class~\eqref{feLinCo5:1}-\eqref{feLinCo5:3}. See e.g.~\cite[Theorem~2.7]{KMwave}.
Moreover, solutions of~\eqref{feLinCo5:2}  in the class~\eqref{feLinCo5:1}-\eqref{feLinCo5:3} are unique. This last property is not explicitly stated in~\cite{KMwave}, but it easily follows from the proof. (It also follows from~Proposition~\ref{eLinCo4}.)

\item Case $N\ge 4$, $p= \frac {N+2} {N-2}$. 
There exist $T>0$ and a unique solution $u$ of~\eqref{feLinCo5:2}  in the class~\eqref{feLinCo5:1}, and this solution satisfies property~\eqref{feLinCo5:3} by possibly choosing $T$ smaller. 
Indeed, existence is established in~\cite{GSV} (see also~\cite[Theorem~2.7]{KMwave} for the case $N=4,5$ and~\cite[Theorem~3.3]{Bulut} for the case $N\ge 6$). Uniqueness is proved in~\cite[Theorem~2]{Planchon} for $N=4$, in~\cite[Theorem~3]{Planchon} for $N=5$ and in~\cite[Theorem~3.4]{Bulut} for $N\ge 6$. 
Property~\eqref{feLinCo5:3} follows from~\cite[Theorem~2.7]{KMwave} in the case $N=4,5$. In the case $N\ge 6$, it follows from~\cite[Theorem~3.3]{Bulut} that $u\in L^q ((0,T), \dot B^\rho  _{ r,2 }(\R^N ))$ with $\rho = \frac {N} {2(N-1)}$, $r= \frac {2(N^2-1) } {N^2 -2N +3}$ and $q= \frac {2(N+1)} {N-1}$, 
hence~\eqref{feLinCo5:3} by Sobolev's embedding. 

\end{enumerate} 
\end{rem} 

\begin{prop} [Uniqueness on light cones] \label{eLinCo4} 
Let $N\ge 1$ and let $p$ satisfy $1<p \le \frac {N+2} {N-2}$ ($1<p<\infty $ if $N=1,2$). Let $R>0$, $0<\tau <R$, and let $B_R$ be the open ball of center $0$ and radius $R$ in $\R^N $.
Let 
\begin{equation*} 
u,v\in C([0,\tau ], H^1 ( B_R ) ) \cap C^1 ([0,\tau ], L^2 (B_R ) ) \cap C^2 ([0,\tau ], H^{-1} ( B_R ) )) 
\end{equation*} 
 be two solutions of the wave equation $\partial  _{ tt } u = \Delta u +  |u|^{p-1} u$ in $H^{-1} ( B_R ) )$ for $\le t\le \tau $. 
If $N\ge 3$ and $p> \frac {N} {N-2}$, suppose in addition that $u,v\in L^{\frac {2(N+1) } {N-2}} ( (0,T) \times B_R  )$. If $u(0)= v(0)$ and $\partial _tu (0)= \partial _t v(0)$, then $u=v$ on $\{ (t,x)\in (0, \tau )\times B_R;\,  |x|< R-t \}$. 
\end{prop} 

The proof of Proposition~\ref{eLinCo4} relies on the following local estimates. 

\begin{lem} \label{eLinCo2} 
Let $R>0$, $0<\tau <R$, $h\in C([0,\tau ], L^q (B_R))$ for some $q\ge 1$, $q\ge \frac {2N} {N+2}$
(so that $h\in C([0,\tau ], H^{-1} (B_R))$), and let
\begin{equation*}
z\in C([0,\tau ], L^2 (B_R ))  \cap C^1([0,\tau ], H^{-1} (B_R))\cap C^2([0,\tau ], H^{-2} (B_R)) 
\end{equation*} 
satisfy $\partial  _{ tt }z = \Delta z +h $ in $H^{-2} (B_R)$ for all $0\le t\le \tau $ and if $z(0)= \partial _t z (0) =0$. If $h _{ | E(0,R, \tau ) } \in  L^2 ( E(0, R, \tau ) ) $ with the notation~\eqref{fTrCo1},  then $ z(t) \in H^1(B _{ R-t) }$ for all $0<t <\tau $, and
\begin{equation} \label{fAX6}
 \| z (t) \| _{ H^1(B_{R-t} )}  \le C e^{Ct} \| h \| _{ L^2 ( E(0, R, t ) ) }
\end{equation} 
 for all $0<t<\tau $. 
If $N\ge 2$ and $h\in L^{\frac {2(N+1)} {N+3}} ( E(0,R,\tau ))$, then $z _{ | E(0,R, \tau ) } \in  L^{\frac {2(N+1)} {N-1}} ( E(0, R, \tau ) ) $ and 
\begin{equation} \label{fAX6b}
 \| z \| _{ L^{\frac {2(N+1)} {N-1}} ( E(0, R, \tau  ) ) } \le C  \| h \| _{ L^{\frac {2(N+1)} {N+3}} ( E(0, R, \tau  ) ) } .
\end{equation} 
In~\eqref{fAX6} and~\eqref{fAX6b}, the constant $C$ independent of $h$, $R$, $\tau $ and $t$.
\end{lem} 

\begin{proof} 
We define $ \widetilde{h} \in C([0,\tau ], L^q ( \R^N ))$ by 
\begin{equation*} 
 \widetilde{h} = 
 \begin{cases} 
 h & \text{on }(0,\tau )\times B_R \\ 0  &\text{elsewhere}. 
 \end{cases} 
\end{equation*} 
We let $ \widetilde{z} \in C([0,\tau ], L^2 (\R^N  ))  \cap C^1([0,\tau ], H^{-1} (\R^N ))\cap C^2([0,\tau ], H^{-2} (\R^N )) $ be the solution of the wave equation $\partial  _{ tt }  \widetilde{z} - \Delta  \widetilde{z} =  \widetilde{h} $ on $\R^N $ with the initial conditions $ \widetilde{z} (0) = \partial _t  \widetilde{z} (0) =0 $.
Note that, given any $0<t\le \tau <R$ and $1\le r\le \infty $, $  \|  \widetilde{h}  \| _{ L^r ( E(0, R, t ) ) }=  \| h \| _{ L^r ( E(0, R, t ) ) } $. Therefore, estimate~\eqref{fAX6} with $z$ replaced by $ \widetilde{z} $ follows from the standard energy inequality for $ \widetilde{z} $; and~\eqref{fAX6b} with $z$ replaced by $ \widetilde{z} $ follows from the Strichartz estimates (see~\cite[Corollary~1.3]{KeelTao}).

To conclude the proof, we show that $z$ and $ \widetilde{z} $ coincide on $E( 0, R, \tau )$. 
We let $w (t)= (z(t) -  \widetilde{z} (t) )  _{ | B_R } $ for all $0\le t\le \tau $, so that
\begin{equation} \label{fAX4}
w\in C([0,\tau ], L^2 (B_R ))  \cap C^1([0,\tau ], H^{-1} (B_R))\cap C^2([0,\tau ], H^{-2} (B_R)) 
\end{equation} 
satisfies $\partial  _{ tt }w = \Delta w $ in $H^{-2} (B_R)$ for all $0\le t\le \tau $ and $w(0)= \partial _t w (0) =0$.
Thus we need to show that $w=0$ a.e. on $E( 0, R, \tau )$.
Let $\rho \in C^\infty _\Comp (\R^N )$, $\rho \ge 0$, be radially symmetric, supported in $B_1$, and satisfy $\int \rho =1$. Given $\varepsilon >0$, let $\rho _\varepsilon (x)= \varepsilon ^{-N} \rho ( \frac {x} {\varepsilon } )$. 
Let $0< \eta < R$ and $0<\varepsilon <\eta /2$. Since $\rho _\varepsilon $ is supported in $B_\varepsilon $, it follows that $\rho _\varepsilon \star w$ is well defined in $B _{ R-\eta }$, and we set $w_\varepsilon = ( \rho _\varepsilon \star w )  _{ | B _{ R-\eta  } }$.
We claim that 
\begin{gather} 
w_\varepsilon \in C^2([0,\tau ] \times  \overline{B _{ R-\eta }} ) \label{fAX1}  \\
\partial  _{ tt } w_\varepsilon =\Delta w_\varepsilon \quad  \text{on}\quad  [0,\tau ] \times  \overline{B _{ R-\eta }} 
\label{fAX2}  \\
 w_\varepsilon (0) = \partial _t  w_\varepsilon (0) =0 \quad  \text{on}\quad B _{ R-\eta }  .\label{fAX3}
\end{gather} 
By finite speed of propagation, it follows that $w_\varepsilon $ identically vanishes on $E(0, R- \eta, \tau )$. Letting $\varepsilon \to 0$, we deduce that $w$ vanishes a.e. on $E(0, R- \eta, \tau )$; and letting $\eta \to 0$ we see that $w$ vanishes a.e. on $E(0, R, \tau )$.
It remains to prove the claims~\eqref{fAX1}-\eqref{fAX3}. 
Given $m \in \N$ and $ \theta \in H^{-m } (B_R))$, recall that $\rho _\varepsilon \star \theta  \in H^{-m} (B _{ R-\eta }))$ is given by 
\begin{equation*} 
\langle \rho  _\varepsilon \star \theta , \varphi  \rangle  _{ H^{-m} (B _{ R-\eta }) ,  H^{m}_0 (B _{ R-\eta } ) }= 
\langle \theta , \rho _\varepsilon \star \varphi  \rangle  _{ H^{-m} (B _{ R } ) ,  H^{m}_0 (B _{ R }) }
\end{equation*} 
for all $\varphi \in C^\infty _\Comp (B _{ R-\eta  })$. It is well known that $\rho _\varepsilon \star \theta  \in C^\infty  (  \overline{B _{ R-\eta }} ))$, $D^\alpha ( \rho _\varepsilon \star \theta  ) = \rho _\varepsilon \star (D^\alpha \theta ) $ for all $\alpha \in \N^N$, and $ \| \rho _\varepsilon \star \theta  \| _{ C (  \overline{B _{ R-\eta }} ) }\lesssim  \| \theta  \| _{ H^{-m} (B_R) }$. 
On the other hand,  it follows from~\eqref{fAX4} that $D^\alpha \partial _t^\beta w\in C ([0,\tau ], H^{-2} (B_R))$ for all $\alpha \in \N^N$ and $\beta \in \N$ such that $ |\alpha |+ \beta \le 2$. Thus we see that $D^\alpha \partial _t^\beta w_\varepsilon \in C ([0,\tau ] \times  \overline{B_{R-\eta} } )$ and that $D^\alpha \partial _t^\beta w_\varepsilon = \rho _\varepsilon  \star D^\alpha \partial _t^\beta w$. 
Properties~\eqref{fAX1}-\eqref{fAX3} easily follow. 
\end{proof} 

\begin{proof} [Proof of Proposition~\rm{\ref{eLinCo4}}]
We need only prove the result for $\tau $ small, the general case follows by iteration. 

The case $p \le \frac {N} {N-2}$ (any $1<p<\infty $ if $N=1,2$). 
We note that $r= 2p $ satisfies $2<r \le \frac {2N} {N-2}$ ($2<r<\infty $ if $N=1,2$), so that by~\eqref{fAX6} and Sobolev's embedding
\begin{equation*} 
 \| (u-v) (t) \| _{ L^r (B _{ R-t }) }^2 \le C e^{Ct} \int _0^t  \| \,  |u|^{p-1}u - |v|^{p-1}v  \| _{ L^2 (B _{ R-s }) }^2 ds .
\end{equation*} 
Since $ \| \,  |u|^{p-1}u - |v|^{p-1}v  \| _{ L^2} \le C  ( \| u\| _{ L^r } +  \|v\| _{ L^r })^{p-1}  \|u-v \| _{ L^r }$ and $u$ and $v$ are bounded in $H^1$, hence in $L^r$, the result follows by Gronwall's inequality.

The case $N\ge 3$ and $p= \frac {N+2} {N-2}$. 
We note that, since $p \le \frac {N+2} {N-2}$, 
\begin{equation*} 
  \bigl|  |u|^{p-1} u -  |v|^{p-1} v \bigr|  \lesssim (  |u|^{p-1} +  |v|^{p-1})  |u-v|  \lesssim ( 1+  |u|^{\frac {4} {N-2}} +  |v|^{\frac {4} {N-2}})  |u-v| . 
\end{equation*} 
By H\"older's inequality, it follows that 
\begin{equation*} 
 \| \,  |u|^{p-1} u -  |v|^{p-1} v  \|  _{ L^{\frac {2(N+1)} {N+3 }}} \lesssim 
 \Bigl( \tau ^{\frac {2} {N+1}}+  \| u \| _{ L^{\frac {2(N+1)} {N-2 }} }^{ \frac {4} {N-2} } + \| u \| _{ L^{\frac {2(N+1)} {N-2 }}}^{ \frac {4} {N-2} } \Bigr)  \| u-v \| _{ L^{\frac {2(N+1)} {N-1 }} }
\end{equation*} 
where all the integrals are on $E(0,R,\tau  )$, with the notation~\eqref{fTrCo1}. 
Applying the Strichartz inequality~\eqref{fAX6b}, we deduce that 
\begin{equation*} 
 \| u-v \|  _{ L^{\frac {2(N+1)} {N-1}}  } \le C \Bigl( \tau ^{\frac {2} {N+1}}+  \| u \| _{ L^{\frac {2(N+1)} {N-2 }} }^{ \frac {4} {N-2} } + \| u \| _{ L^{\frac {2(N+1)} {N-2 }}}^{ \frac {4} {N-2} } \Bigr)  \| u-v \| _{ L^{\frac {2(N+1)} {N-1 }} }
\end{equation*} 
where all the integrals are on $E(0,R,\tau )$. 
Since
\begin{equation*} 
\| u \| _{ L^{\frac {2(N+1)} {N-2 }} ( E(0,R,\tau ) )} + \| v \| _{ L^{\frac {2(N+1)} {N-2 }} ( E(0,R,\tau ) )} \goto _{ \tau \downarrow 0 }0
\end{equation*} 
the conclusion follows by choosing $\tau $ sufficiently small. 
\end{proof}

\color{black}

\end{document}